\documentclass[11pt,reqno]{amsart}
\usepackage{mathrsfs}
\usepackage[breaklinks]{hyperref}
\usepackage{diagbox}
\usepackage{latexsym}
\usepackage{upref, eucal}

\usepackage[headheight=110pt,top=1in, bottom=.9in, left=0.8in, right=0.8in]{geometry}

\usepackage{url}
\usepackage{amssymb}

\renewcommand{\a}{\alpha}
\renewcommand{\b}{\beta}

\newcommand{\g}{\gamma}

\renewcommand{\(}{\left\(}
\renewcommand{\)}{\right\)}
\renewcommand{\[}{\left\[}
\renewcommand{\]}{\right\]}
\numberwithin{equation}{section}
\theoremstyle{plain}
\newtheorem{theorem}{Theorem}[section]
\newtheorem{lemma}[theorem]{Lemma}
\newtheorem{corollary}[theorem]{Corollary}
\newtheorem{proposition}[theorem]{Proposition}

\newtheorem{example}[]{Example}
\newtheorem{remark}[]{Remark}

\makeatletter
\def\proof{\@ifnextchar[{\@oproof}{\@nproof}}
\def\@oproof[#1][#2]{\trivlist\item[\hskip\labelsep\textit{#2 Proof of\
		#1.}~]\ignorespaces}
\def\@nproof{\trivlist\item[\hskip\labelsep\textit{Proof.}~]\ignorespaces}

\makeatother

\makeatletter
\def\@tocline#1#2#3#4#5#6#7{\relax
	\ifnum #1>\c@tocdepth 
	\else
	\par \addpenalty\@secpenalty\addvspace{#2}%
	\begingroup \hyphenpenalty\@M
	\@ifempty{#4}{%
		\@tempdima\csname r@tocindent\number#1\endcsname\relax
	}{%
		\@tempdima#4\relax
	}%
	\parindent\z@ \leftskip#3\relax \advance\leftskip\@tempdima\relax
	\rightskip\@pnumwidth plus4em \parfillskip-\@pnumwidth
	#5\leavevmode\hskip-\@tempdima
	\ifcase #1
	\or\or \hskip 1em \or \hskip 2em \else \hskip 3em \fi%
	#6\nobreak\relax
	\dotfill\hbox to\@pnumwidth{\@tocpagenum{#7}}\par
	\nobreak
	\endgroup
	\fi}
\makeatother

\allowdisplaybreaks

\begin{document}
	\title[]{Mordell-Tornheim zeta functions and functional equations for Herglotz-Zagier type functions}
	\author{Atul Dixit}
	\author{Sumukha Sathyanarayana}
	\author{N. Guru Sharan }
	\address{Department of Mathematics, Indian Institute of Technology Gandhinagar, Palaj, Gandhinagar, Gujarat-382355, India}
	\email{adixit@iitgn.ac.in; sumukha.s@iitgn.ac.in; gurusharan.n@iitgn.ac.in}
	
	\begin{abstract}
		The Mordell-Tornheim zeta function and the Herglotz-Zagier function $F(x)$ are two important functions in Mathematics. By generalizing a special case of the former, namely $\Theta(z, x)$, we show that the theories of these functions are inextricably woven. We obtain a three-term functional equation for $\Theta(z, x)$ as well as decompose it in terms of the Herglotz-Hurwitz function $\Phi(z, x)$. This decomposition can be conceived as a two-term functional equation for $\Phi(z, x)$. Through this result, we are not only able to get Zagier's identity relating $F(x)$ with $F(1/x)$ but also two-term functional equation for Ishibashi's generalization of $F(x)$, namely, $\Phi_k(x)$ which has been sought after for over twenty years. We further generalize $\Theta(z, x)$ by incorporating two Gauss sums, each associated to a Dirichlet character, and decompose it in terms of an interesting integral which involves the Fekete polynomial as well as the character polylogarithm. This result gives infinite families of functional equations of Herglotz-type integrals out of which only two, due to Kumar and Choie, were known so far. The first one among the two involves the integral $J(x)$ who special values have received a lot of attention, more recently, in the work of Muzzaffar and Williams, and in that of Radchenko and Zagier. Analytic continuation of our generalization of $\Theta(z, x)$ is also accomplished which allows us to obtain transformations between certain double series and Herglotz-type integrals or their explicit evaluations.

	\end{abstract}
	\maketitle
	\tableofcontents
	\section{Introduction}\label{intro}
	
	The Mordell-Tornheim zeta function is defined by
	\begin{align}\label{MT}
		\zeta_{MT}(s_1,s_2,s_3) := \sum_{n=1}^{\infty} \sum_{m=1}^{\infty} \frac{1}{n^{s_1} m^{s_2} (n+m)^{s_3}}.
	\end{align}
	It converges absolutely for Re$(s_1+s_3)>1$, Re$(s_2+s_3)>1$ and  Re$(s_1+s_2+s_3)>2$; see \cite[Theorem 2.2]{bradley-zhou}.
	The special case $s_1=s_2=s_3=s$ of this function is essentially the Witten zeta-function for the semisimple Lie algebra $\mathfrak{sl}(3)$ \cite{zagier-witten}. Tornheim \cite{tornheim}, and later Mordell \cite{mordell} studied special values of $\zeta_{MT}(s_1,s_2,s_3)$ at positive integers, which led Matsumoto \cite{matsumoto-bonn} to aptly coin, for the function in \eqref{MT}, the terminology \emph{Mordell-Tornheim zeta function}.

Matsumoto \cite[Theorem 1]{matsumoto} showed that $\zeta_{MT}(s_1, s_2, s_3)$ can be meromorphically continued to the entire $\mathbb{C}^3$ space with all its singularities lying on the subsets of $\mathbb{C}^3$ defined by one of the equations $s_1+s_3=1-\ell, s_2+s_3=1-\ell,$ and $s_1+s_2+s_3=2$, where $\ell \in \mathbb{N}\cup\{0\}$. However, the precise locations of the singularities of $\zeta_{MT}(s_1, s_2, s_3)$ are not known except in some special cases, for example, $\zeta_{MT}(s, s, s)$, see Romik \cite[Theorem 1.2, 1.3]{romik}. Having information about poles of any meromorphic function is always desirable. For example, Bringmann and Franke \cite{bringmann-franke} used the distribution of the poles of  $\zeta_{MT}(s, s, s)$ to obtain the asymptotic expansion of $r(n)$, the number of $n$-dimensional representations of the group SU$(3)$, counted up to equivalence, as $n\to\infty$. For an up-to-date discussion on these functions, and, in general, on the zeta-functions of root systems of which $\zeta_{MT}(s_1, s_2, s_3)$ is but one case, we refer the reader to a recent book by Komori, Matsumoto, and Tsumura \cite{komori-matsumoto-tsumura}.

In this paper, we consider a generalization of the Mordell-Tornheim zeta function defined for Re$(s_1+s_3)>1$, Re$(s_2+s_3)>1$, Re$(s_1+s_2+s_3)>2$ and $x>0$ by
\begin{align*}
	\Theta(s_1,s_2,s_3,x):=\sum_{n=1}^{\infty}  \sum_{m=1}^{\infty}\frac{1}{n^{s_1}m^{s_2}(n+mx)^{s_3}}.
\end{align*}
Clearly, $\Theta(s_1,s_2,s_3,1)=\zeta_{MT}(s_1, s_2, s_3)$. The meromorphic continuation of $\Theta(s_1,s_2,s_3,x)$ can be done similarly as in the case of $\zeta_{MT}(s_1, s_2, s_3)$. Also, it is easy to establish the following identities for Re$(s_1+s_3)>1$, Re$(s_2+s_3)>1$, Re$(s_1+s_2+s_3)>2$ and $x>0$:
	\begin{align}
	\Theta(s_1,s_2,s_3,x)&= \Theta(s_1-1,s_2,s_3+1,x)+x \Theta(s_1,s_2-1,s_3+1,x),\nonumber\\
	\Theta(s_1,s_2,s_3,x)&= x^{-s_3} \Theta\left(s_2,s_1,s_3,\frac{1}{x}\right). \label{2t}
\end{align}	
We shall focus on the special case of $\Theta(s_1,s_2,s_3,x)$, namely,
\begin{align}\label{Thetadef}
\Theta(z, x):=	\Theta(1,1,z-1,x):=\sum_{n=1}^{\infty}  \sum_{m=1}^{\infty}\frac{1}{nm(n+mx)^{z-1}}, 
\end{align}
which is absolutely convergent for Re$(z)>1$.

While there have been numerous generalizations of $\zeta_{MT}(s_1, s_2, s_3)$ in the literature, $\Theta(s_1,s_2,s_3,x)$ does not seem to have been studied before. Its special case $\Theta(z, x)$ that we focus on has some really nice properties not only when viewed as a function of  $z$ but also of $x$. 

To the best of our knowledge, the principal part of Laurent expansion of $\Theta(z, x)$ around $z=1$ is not known (even for $x=1$). Our first result proved in Section \ref{properties of MT} gives not only the principal part but also the constant term. 
\begin{theorem}\label{principal part}
	Let $x>0$ and \textup{Re}$(z)>1$. As $z\to1$,
	\begin{align}\label{principal part equation}
	\Theta(z, x)=\frac{2}{(z-1)^2}+\frac{2\gamma-\log(x)}{z-1}+\gamma^2-\gamma\log(x)-\frac{\pi^2}{6}+O(|z-1|),
	\end{align}
	where $\gamma$ is Euler's constant.
\end{theorem}

\noindent
Observe that for Re$(z)>1$, we have 
$\Theta(z,z,0,1)=\zeta^2(z)$,
and hence as $z\to1$,
\begin{align*}
\Theta(z,z,0,1)=\frac{1}{(z-1)^2}+\frac{2\gamma}{z-1}+\gamma^2-2\gamma_1+O(|z-1|),
\end{align*}
whereas from Theorem \ref{principal part}, we have
\begin{align*}
	\Theta(1, 1, z-1, 1)=\frac{2}{(z-1)^2}+\frac{2\gamma}{z-1}+\gamma^2-\frac{\pi^2}{6}+O(|z-1|).
\end{align*}
It readily follows from \eqref{2t} that
\begin{align*}
	\Theta(z, x)=x^{1-z}\Theta\left(z, \frac{1}{x}\right).
\end{align*}
Functional equations of this type are called \emph{two-term functional equations}. Such functional equations abound in Mathematics. Indeed, the inversion property of a modular form is one of these examples.  Two-term functional equations are also ubiquitous in Ramanujan's notebooks and his lost notebook. We state one such beautiful transformation on page 220 of the Lost Notebook. Let
\begin{align}\label{lambda}
	\lambda(x):=\psi(x)+\frac{1}{2x}-\log x,
\end{align}
where $\psi(x)=\Gamma'(x)/\Gamma(x)$ is the logarithmic derivative of Gamma function or the so-called \emph{digamma function}. Let $\xi(s)$ and $\Xi(t)$ denote Riemann's functions respectively defined by \cite[p.~16, Equations (2.1.12), (2.1.14)]{titch}
\begin{align}
	\xi(s)&:=(s-1)\pi^{-\tfrac{1}{2}s}\Gamma(1+\tfrac{1}{2}s)\zeta(s),\nonumber\\
	\Xi(t)&:=\xi(\tfrac{1}{2}+it),\nonumber
\end{align}
Then for $\a,\b$ positive\footnote{The result actually holds for any complex $\a$ and $\b$ lying in the slit-complex plane having branch cut from $(-\infty, 0]$.} such that $\a\b=1$,
\begin{align}\label{w1.26}
	\sqrt{\a}&\left\{\frac{\gamma - \log(2 \pi \a)}{2\a} +\sum_{n=1}^{\infty}\lambda(n\a)\right\}=\sqrt{\b}\left\{\frac{\gamma - \log(2 \pi \b)}{2\b} +\sum_{n=1}^{\infty}\lambda(n\b)\right\}\nonumber\\
	&=-\frac{1}{\pi^{3/2}}\int_{0}^{\infty}\left|\Xi\left( \frac{t}{2}\right) \Gamma\left(\frac{-1+it}{4} \right)\right|^2\frac{\cos\left(\frac{t}{2}\log\a \right)}{1+t^2}dt.
\end{align}
This transformation was first proved by Berndt and the first author in \cite{bcbad}.

Now there is another set of functional equations called \emph{three-term functional equations} which are comparatively rare. Nevertheless, some very important objects in Mathematics satisfy such equations, and the relations themselves are quite important. For example, consider the following famous formula of Ramanujan for odd zeta values \cite[p.~173, Ch. 14, Entry 21(i)]{ramnote}, \cite[p.~319-320, formula (28)]{lnb}, \cite[p.~275-276]{bcbramsecnote}, valid for Re$(\a), \textup{Re}(\b)>0$ with $\a\b=\pi^2$ and $m\in\mathbb{Z}\backslash\{0\}$:
\begin{align}\label{zetaodd}
	\a^{-m}\left\{\frac{1}{2}\zeta(2m+1)+\sum_{n=1}^{\infty}\frac{n^{-2m-1}}{e^{2\a n}-1}\right\}&=(-\b)^{-m}\left\{\frac{1}{2}\zeta(2m+1)+\sum_{n=1}^{\infty}\frac{n^{-2m-1}}{e^{2\b n}-1}\right\}\nonumber\\
	&\quad-2^{2m}\sum_{j=0}^{m+1}\frac{(-1)^jB_{2j}B_{2m+2-2j}}{(2j)!(2m+2-2j)!}\a^{m+1-j}\b^j,
\end{align}
and where $B_j$ are the Bernoulli numbers.  The second term on the right-hand side of \eqref{zetaodd} is essentially the generalized period polynomial associated to the weight $2k$ Eisenstein series, which can be recast in the form $\mathscr{P}_k(x):=\sum_{r=0}^{2k}\frac{B_rB_{2k-r}}{r!(2k-r)!}x^{r-1}$. Vlasenko and Zagier  \cite[p.~42, Equation (30)]{vz} have shown that $\mathscr{P}_k(x)$ not only satisfies a two-term functional equation but also a three-term one:
	\begin{align*}
	\mathscr{P}_k(x)=x^{2k-2}\mathscr{P}_k\left(\frac{1}{x}\right),\hspace{4mm} \mathscr{P}_k(x)-\mathscr{P}_k(x+1)-(x+1)^{2k-2}\mathscr{P}_k\left(\frac{x}{x+1}\right)=0.
\end{align*}
The period functions associated to Maass forms and real analytic Eisenstein series also satisfy three-term relations of the form \cite[p.~193]{lewzag} (see also \cite[Equation (1)]{betcon})
\begin{align*}
	\phi(x)-\phi(x+1)-(x+1)^{-2s}\phi\left(\frac{x}{x+1}\right)=0.
\end{align*}
Folsom \cite[Theorem 1]{folsom} has recently obtained a three-term relation for a period function associated to a  twisted Eisenstein series. Zagier \cite{zagier3term} has written an interesting article on such three-term relations and their applications. 

What is interesting to know here is that our function $\Theta(z, x)$ also satisfies a three-term functional equation in the variable $x$. This is stated next and proved in Section \ref{three-term Feq of theta}.
		\begin{theorem}\label{theta-3t}
	For $\textup{Re}(z)>1$ and $x>0$,
	\begin{align}\label{theta-3t eqn}
		&\Theta(z, x)-	\Theta(z, x+1)-(x+1)^{1-z}	\Theta\left(z, \frac{x}{x+1}\right)\nonumber\\
		&=-(x+1)^{1-z}\Theta(z,1)+\left(1+x^{1-z}-(x+1)^{1-z}\right)\zeta(z+1)\nonumber\\
		&\quad-\frac{1}{\Gamma(z-1)}\int_{0}^{\infty}t^{z-2}\left\{\textup{Li}_{2}\left(\frac{e^{-(x+1)t}-e^{-xt}}{e^{-(x+1)t}-1}\right)+\textup{Li}_{2}\left(\frac{e^{-(x+1)t}-e^{-t}}{e^{-(x+1)t}-1}\right)\right\}\, dt.
	\end{align}
\end{theorem}
The Mellin transform of the sum of dilogarithms appearing in Theorem \ref{theta-3t} looks intractable as far as its closed-form evaluation is concerned.

The function  $\Theta(z, x)$  can be decomposed in terms of a new generalization of the Herglotz-Zagier function $F(x)$. The function $F(x)$ is defined by
\begin{equation}\label{herglotzdef}
	F(x):=\sum_{n=1}^{\infty}\frac{\psi(nx)-\log(nx)}{n}\hspace{5mm}\left(x\in\mathbb{C}\backslash(-\infty,0]\right).
\end{equation} 
Its importance warrants a discussion on some of the beautiful properties it satisfies.
It arose in Zagier's Kronecker limit formula for real quadratic fields, and is similar to  a function studied by Herglotz \cite{herglotz}, hence the name, the \emph{Herglotz-Zagier function}. Since \cite[p.~259, formula 6.3.18]{as}
\begin{equation*}
	\psi(s)=\log(s)-\frac{1}{2s}+O\left(\frac{1}{s^2}\right)\hspace{8mm}( |\arg(s)|<\pi),	
\end{equation*}
as $s\to\infty$, it is clear that the series in \eqref{herglotzdef} converges absolutely for any $x\in\mathbb{C}\backslash(-\infty,0]$.
It satisfies beautiful two- and three-term functional equations due to Zagier \cite[Equations (7.4), (7.8)]{zagier}:
\begin{align}
	&F(x)+F\left(\frac{1}{x}\right)=2F(1)+\frac{1}{2}\log^{2}(x)-\frac{\pi^2}{6x}(x-1)^2,\label{fe2}\\
	&F(x)-F(x+1)-F\left(\frac{x}{x+1}\right)=-F(1)+\textup{Li}_2\left(\frac{1}{1+x}\right),\label{fe1}
	\end{align}
	where \cite[Equation (7.12)]{zagier}
	\begin{equation}\label{ef1}
F(1)=-\frac{1}{2}\g^2-\frac{\pi^2}{12}-\g_1,
\end{equation}
with $\g$ and $\g_1$ being the Euler and the first Stieltjes constants respectively, and $\textup{Li}_{2}(z)$ is the dilogarithm function defined, for $|z|<1$, by
\begin{align}\label{li2z}
	\textup{Li}_{2}(z):=\sum_{n=1}^{\infty}\frac{z^n}{n^2},
\end{align}
and for any $z\in\mathbb{C}$ by
	\begin{equation}\label{didef}
	\textup{Li}_2(z):=-\int_{0}^{z}\frac{\log(1-u)}{u}\, du.
\end{equation}
Very recently, Radchenko and Zagier \cite{raza} have extensively studied further properties of $F(x)$. Note that the transformation \eqref{w1.26} is associated with the function $\lambda(x)$ (defined in \eqref{lambda}) in the same way as 
 \eqref{fe2} is associated with \eqref{herglotzdef}. Thus Ramanujan seems to be the first mathematician to understand the importance of Herglotz-type functions and their two-term functional equations; see \cite{gupta-kumar} for more details. However, Ramanujan was probably unaware of three-term functional equations.
 
 Also observe that while the three-term relation for $F(x)$ involves dilogarithm on the right-hand side, that for $\Theta(z, x)$ in Theorem \ref{theta-3t} involves an integral of a sum of dilogarithm functions.
	
As mentioned before, $\Theta(z, x)$ can be expressed in terms of a new generalization of $F(x)$. For $\textup{Re}(z)>0, z\neq1,$ and $x\in\mathbb{C}\backslash(-\infty,0]$, we define this generalized Herglotz function by
\begin{align}\label{hhf-def}
	\Phi(z, x):=\sum_{n=1}^{\infty} \frac{1}{n}\left( \zeta(z,nx)- \frac{(nx)^{1-z}}{z-1} \right),
\end{align}	 
where $\zeta(z, x)$ is the Hurwitz zeta function, and call it \emph{the Herglotz-Hurwitz function}. Observe that for $z\in\mathbb{C}\backslash\{1\}$, the asymptotic expansion of $\zeta(z,x)$ \cite[p.~25]{mos} for large $|x|$ and $|\arg$ $x|<\pi$ is given by
\begin{equation*}
	\zeta(z,x)=\frac{1}{\Gamma(z)}\left(x^{1-z}\Gamma(z-1)+\frac{1}{2}\Gamma(z)x^{-z}+\sum_{k=1}^{m-1}\frac{B_{2k}}{(2k)!}\Gamma(z+2k-1)x^{-2k-z+1}\right)+O(x^{-2m-z-1}),
\end{equation*}
which is why the series in \eqref{hhf-def} converges absolutely for $x\in\mathbb{C}\backslash(-\infty,0]$. Note that $\Phi(z, x)$ has a removable singularity at $z=1$.

As alluded to before, we give below the decomposition theorem for $\Theta(z, x)$ which can also be viewed as a two-term functional equation for $\Phi(z, x)$.
\begin{theorem}\label{decompostion theorem}
	For \textup{Re}$(z)>1$ and $x>0$,
		\begin{align}\label{decomposition}
		\Phi(z, x)+x^{1-z}\Phi\left(z, \frac{1}{x}\right)=\sum_{n=1}^{\infty} \sum_{m=1}^{\infty} \frac{1}{mn (n+mx)^{z-1}}-\frac{(1+x^{1-z})}{z-1}\zeta(z)+(x+x^{-z})\zeta(z+1).
	\end{align}
\end{theorem}
The above result, proved in Section \ref{decomp}, provides an analytic continuation of $\Theta(z, x)$ in the region Re$(z)>0, z\neq 1$, since all other functions in \eqref{decomposition} are analytic there.  

Letting $z\to1$ in \eqref{decomposition} leads to \eqref{fe2} as shown in Corollary \ref{k=0}. But now a natural question arises - \emph{what do we get upon differentiating both sides of \eqref{decomposition} $k$ times with respect to $z$ and then letting $z\to 1$?}

To answer this question, we first need to consider the $k^{\textup{th}}$ order Herglotz function of Ishibashi \cite{ishibashi} defined for $k\in\mathbb{N}$ by
\begin{equation}\label{ishibashi-def}
	\Phi_k(x):=\sum_{n=1}^{\infty}\frac{k\psi_{k-1}(nx)-\log^{k}(nx)}{n}\hspace{5mm}\left(x\in\mathbb{C}\backslash(-\infty,0]\right),
\end{equation}
where  $\psi_k(x)$ is the \emph{generalized digamma function} defined by \cite[Proposition 10]{dilcher}
\begin{equation*}
	\psi_k(x):=\frac{\Gamma_k'(x)}{\Gamma_k(x)}=-\gamma_k-\frac{\log^{k}(x)}{x}-\sum_{n=1}^{\infty}\left(\frac{\log^{k}(n+x)}{n+x}-\frac{\log^{k}(n)}{n}\right),
\end{equation*}
with $\gamma_k$ being the $k^{\textup{th}}$ Stieltjes constant with $\gamma_0=\gamma$, the Euler's constant. Observe that $\psi_0(x)=\psi(x)$.

 The anti-derivative of $\psi_k(x)$, namely, $R_k(x)$, is a generalization of the function $R(x)$ studied by Deninger \cite{deninger}, and is defined as the solution to the difference equation
\begin{equation*}
	f(x+1)-f(x)=\log^{k}(x),\hspace{5mm}(f(1)=(-1)^{k+1}\zeta^{(k)}(0)).
\end{equation*}
See \cite{dss} for the details and history of $\psi_k(x)$ and \cite{dilcher} for the properties of the corresponding gamma function $\Gamma_k(x)$. Ishibashi evaluated all Laurent series coefficients of a zeta function associated to an indefinite quadratic form in terms of $\Phi_k(x)$, for which only the constant term was known earlier, due to Zagier \cite{zagier}. 

Clearly, $\Phi_1(x)=F(x)$, and so a natural question that one thinks of is, \emph{do there exist two- and three-term functional equations for $\Phi_k(x)$ for $k>1$?}

We partially answer this question by obtaining a two-term functional equation involving $\Phi_k(x)$ and $\Phi_k(1/x)$ and a three-term functional equation involving  $\Phi'_k(x)$, $\Phi'_k(x-1)$  and $\Phi'_k((x-1)/x)$ in Theorems \ref{thm-2fe-ishibashi} and \ref{Thm 3term-curious} respectively. However, these functional equations are novel in that the two-term one involves  a weighted linear combination of  $\Phi_j(x)$ and $\Phi_j(1/x)$ whereas the three-term relation involves a weighted linear combination of  $\Phi'_j(x)$, $\Phi'_j(x-1)$ and $\Phi'_j((x-1)/x)$ for \emph{all} $j$ such that $0\leq j\leq k$. 
The existence of such functional equations relating just $\Phi_k(x)$ with $\Phi_k(1/x)$ or $\Phi'_k(x)$  with $\Phi'_k(x-1)$  and $\Phi'_k((x-1)/x)$ seems improbable for $k>1$. The only way we know of getting them is by considering a one-variable generalization of $F(x)$, that is, the Herglotz-Hurwitz function 
$\Phi(z, x)$, and then differentiating \eqref{decomposition} for two-term (or \eqref{three term in z} for the three-term) $k$ times with respect to $z$  before letting $z\to 1$. This process is what we call the \textit{Hurwitz lift}.

To give an idea of how involved the two-term functional equation for higher $k$ is,  we give below the special case $k=1$ of Theorem \ref{thm-2fe-ishibashi} which is a  functional equation involving $\Phi_2$ as well as $\Phi_1$ while the statement of Theorem \ref{thm-2fe-ishibashi} itself is saved for Section \ref{2fe-ishibashi}. For $x\in\mathbb{C}\backslash(-\infty,0]$, we have
\small\begin{align}\label{k=1}
	&-\frac{\log(x)}{2} \left( \Phi_1(x) -\Phi_1 \left( \frac{1}{x} \right) \right) + \frac{1}{2} \left( \Phi_2(x) +\Phi_2 \left( \frac{1}{x} \right) \right)  \nonumber\\
	&=\frac{\pi^2}{12}\log(x) \left(x -\frac{1}{x}\right) + \zeta'(2) \left(x +\frac{1}{x}\right) - \frac{1}{2} \log(x)\left( \frac{1}{2} \log^2(x) +2 \Phi_1(1) +\frac{\pi^2}{3} \right) \nonumber\\
	& \quad-\frac{\gamma^2}{2} \log^2(x) +2 c_1 \log(x) + \frac{1}{3}\log^3(x) + \Phi_2(1) -2 \zeta'(2) +L_0^*(x),
\end{align} 
where 
\begin{align*}
	L_0^*(x):= \lim_{z \to 1}\frac{1}{\Gamma(z)} \int_{0}^{\infty}t^{z-2}\log \left( \frac{1-e^{-xt}}{1-e^{-t}} \right) \log \left( \frac{1-e^{-t}}{t}  \right) \, dt
\end{align*}
and $c_1$ is the constant term of the Laurent series of $\Gamma(s)\zeta(s)$ around $s=1$. Recall that $\Phi_1(1)=F(1)$ which is an explicit constant given in \eqref{ef1}. More on the limit $L_{0}^{*}(x)$, and, in general, on $L_{k}^{*}(x)$ for $k\geq1$, will be said in Section \ref{2fe-ishibashi}.

After Zagier, Novikov \cite{novikov}  gave a different expression for the Kronecker limit formula for a real quadratic field. It involves a function defined for $u\in\mathbb{C}\backslash(1,\infty), u\neq0, v\in\mathbb{C}\backslash[1,\infty)$ and $x>0$ by
\begin{align}\label{script F def}
	\mathscr{F}(x;u,v) :=  \int_{0}^{1} \frac{ \log(1-ut^x) }{v^{-1}-t} \, dt. 
\end{align}
In their beautiful work, Kumar and Choie \cite{kumar-choie} termed this function as the \emph{Herglotz-Zagier-Novikov function}. This function is not only related to the Herglotz-Zagier function $F(x)$ (see \cite[p.~4]{kumar-choie} but also to the integrals
\begin{align}\label{JT-def}
	J(x):=\int_{0}^{1}\frac{\log(1+t^x)}{1+t}\, dt,\hspace{4mm}
	T(x):=\int_{0}^{1}\frac{\tan^{-1}(t^x)}{1+t^2}\, dt,
\end{align}
where $x>0$. While the explicit evaluation of a particular case of $J(x)$, namely, $J(4+\sqrt{15})$, occurred as early as in Herglotz's work \cite{herglotz} from 1923, the integral $T(x)$ has been considered more recently and seems to have first appeared in the work of Muzzaffar and Williams \cite{muzwil}. What is common and interesting about these two integrals is that an elementary approach towards their closed-form evaluation has evaded the mathematical community so far. Before the  work of Radchenko and Zagier \cite{raza}, both integrals were handled through algebraic number-theoretic techniques only. Radchenko and Zagier used techniques in analytic number theory to evaluate $J(x)$. They also found a nice relation between $J(x)$ and $F(x)$ \cite[Equation (1.1)]{raza}, namely, for $x>0$,
\begin{align}\label{zaginteval}
	J(x)=F(2x)-2F(x)+F\left(\frac{x}{2}\right)+\frac{\pi^2}{12x}.
\end{align}
Elegant functional equations for $J(x)$ and $T(x)$ were derived by Kumar and Choie who showed that
\begin{align}\label{JT-FE}
	J(x) +J\left(\frac{1}{x}\right) = \log^2(2),\hspace{8mm}
	T(x)+T\left(\frac{1}{x}\right) = \frac{\pi^2}{16}.
\end{align}
At first glance, it looks surprising why the right-hand sides in \eqref{JT-FE} are independent of $x$. We not only resolve this mystery but also obtain a doubly-infinite family of functional equations of which the ones in \eqref{JT-FE} are but two special cases. In doing so, we give a unified treatment of integrals generalizing $J(x)$ and $T(x)$ as  explained below.

Let $\chi_1,\chi_2$ be Dirichlet characters modulo $r$. Let $\zeta_r=e^{2\pi i/r} $ and $G(\ell,\chi)=\sum_{j=1}^{r} \chi(j) e^{2 \pi i j \ell/r}$ be the Gauss sum associated to a Dirichlet character $\chi$. For Re$(z)>0$ and $x>0$, define
\begin{align}
	J_{r,\chi_1, \chi_2}(z,x) := \frac{1}{\Gamma(z)} \int_{0}^{1} \frac{\left( \log \left( \frac{1}{t} \right) \right)^{z-1}}{t(1-t^r)} \left( \sum_{\ell=1}^{r} G(\ell,\chi_2) t^{\ell}\right) \left(  \sum_{j=1}^{r} \chi_1(j) \log \left( 1-\zeta_r^j t^{x} \right) \right) \, dt. \label{J r chi def}
\end{align}

When $\chi_2$ is primitive, the separablility of Gauss sum $G(\ell,\chi_2)$ \cite[p.~168, Theorem 8.15]{Apostol} implies
$\sum_{\ell=1}^{r} G(\ell,\chi_2) t^{\ell}=G(1,\chi_2)\sum_{\ell=1}^{r}\overline{\chi}_2(\ell)t^{\ell}$, and hence
 \begin{align*}
 	J_{r,\chi_1, \chi_2}(z,x) = \frac{G(1,\chi_2)}{\Gamma(z)} \int_{0}^{1} \frac{\left( \log \left( \frac{1}{t} \right) \right)^{z-1}}{t(1-t^r)} \left( \sum_{\ell=1}^{r}\overline{\chi}_2(\ell)t^{\ell}\right) \left(  \sum_{j=1}^{r} \chi_1(j) \log \left( 1-\zeta_r^j t^{x} \right) \right) \, dt. 
 \end{align*}
 Note that the polynomial in $t$ occurring in the numerator of the integrand of \eqref{J r chi def} is the generalized Fekete polynomial associated to the Dirichlet character $\overline{\chi}_2$ of modulus $r$; see \cite[Definition 3.1]{minac-jnt}. In general, it is denoted by
 \begin{equation*}
 F_{\chi}(t):=\sum_{\ell=1}^{r}\chi(\ell)t^{\ell}.
\end{equation*} 
 These polynomials are well-known and have been studied for over two centuries. Indeed, Gauss used them in one of his proofs of the law of quadratic reciprocity. The importance of these polynomials lies not only in the arithmetic of quadratic number fields (such as in the class number formula) but also in the theory of quadratic $L$-functions. 
 
 It is well-known \cite[Proposition 3.3]{minac-jnt} that for a primitive character $\chi$ modulo $r$ and for Re$(z)>0$, the Dirichlet $L$-function satisfies
 \begin{align}\label{lz chi integral}
 	L(z, \chi)=\frac{1}{\Gamma(z)}\int_{0}^{1}\frac{ \left(\log \left( \frac{1}{t} \right) \right)^{z-1}}{t(1-t^r)}\left(\sum_{\ell=1}^{r}\chi(\ell)t^{\ell}\right)\, dt.
 \end{align} 
 Fekete showed that for a quadratic character $\chi$ of conductor $p$, where $p$ is a prime, $L(z, \chi)$ has no real zeros on $(0, \infty)$, provided $F_{\chi}(t)$ does not have any real zeros in the interval $(0, 1)$. Thus the study of their zeros may have implications on the existence of Siegel zeros. For more information, we refer the reader to \cite{sound}.

We now give the functional equation for $J_{r,\chi_1, \chi_2}(z,x)$ which, as mentioned before, give the relations in \eqref{JT-FE} as special cases.

\begin{theorem}\label{J r chi FE}
	Let $r\in\mathbb{N}, r>1,$ and let $\chi_1,\chi_2$ be  any Dirichlet characters modulo $r$. For \textup{Re}$(z)>1$ and $x>0$, 
	\begin{align}\label{fe-general}
		J_{r,\chi_1, \chi_2}(z,x) +x^{1-z} J_{r,\chi_2, \chi_1}\left(z,\frac{1}{x}\right) = - \sum_{n=1}^{\infty} \sum_{m=1}^{\infty} \frac{G(m,\chi_1)G(n,\chi_2)}{m n (n+mx)^{z-1}}.
	\end{align}
	In particular, if $\chi_1$ and $\chi_2$ are primitive, then
		\begin{align}\label{fe-primitive}
		J_{r,\chi_1, \chi_2}(z,x) +x^{1-z} J_{r,\chi_2, \chi_1}\left(z,\frac{1}{x}\right) = -G(1, \chi_1)G(1, \chi_2) \sum_{n=1}^{\infty} \sum_{m=1}^{\infty} \frac{\bar\chi_1(m) \bar\chi_2(n)}{m n (n+mx)^{z-1}}.
	\end{align}
\end{theorem}

As shown in Section \ref{integral-J}, the integral $J_{r,\chi_1, \chi_2}(z,x)$ is analytic in Re$(z)>0$. Moreover, using Crandall's method, we show that the right-hand side of \eqref{fe-general} can be analytically continued to an entire function of $z$. See Proposition \ref{crandall} and Remark \ref{gauss - analytic continuation}. From the above discussion, we conclude that the analytically continued version of the identity in \eqref{fe-general}, which is stated in Theorem \ref{extended version}, is valid in Re$(z)>0$. 
This has important consequences. Indeed, as shown in Corollary \ref{unnamed corollary}, the functional equations for $J$ and $T$ in \eqref{JT-FE} follow as special cases.

In Theorem \ref{J r chi FE}, one can now let $r$ be \emph{any} natural number, $\chi_1, \chi_2$ be \emph{any} Dirichlet characters modulo $r$, and $z$, with Re$(z)>0$, can be \emph{any} complex number! This gives us an infinite supply of such functional equations. A couple of them are illustrated in Section \ref{ei} whereas some more are given in Table 2.

One can specialize $z$ (such that Re$(z)>1$) as well as $r$ and the characters in \eqref{fe-general}, and see if either of the left- or right-hand sides can be explicitly evaluated to obtain explicit evaluations of series/integrals. For example, we obtain the following identity, which is new to the best of our knowledge:
\begin{align}\label{new series evaluation}
	\sum_{n=1}^{\infty}\sum_{m=1}^{\infty}\frac{\big(e^{\frac{\pi i n}{3}}-e^{\frac{5\pi i n}{3}}\big)\big(e^{\frac{\pi i m}{3}}-e^{\frac{5\pi i m}{3}}\big)}{n m (n+m)}&=\frac{-1}{18}\bigg[3\log^{3}(3)-5\pi^2\log(3)+\pi\sqrt{3}\big\{\psi'\left(\tfrac{1}{6}\right)+\psi'\left(\tfrac{1}{3}\right)-\psi'\left(\tfrac{2}{3}\right)-\psi'\left(\tfrac{5}{6}\right)\big\}\nonumber\\
	&\qquad\quad+\left(18\log(3)-30\pi i\right)\textup{Li}_2\big(\tfrac{3-i\sqrt{3}}{6}\big)+\left(18\log(3)+30\pi i\right)\textup{Li}_2\big(\tfrac{3+i\sqrt{3}}{6}\big)\nonumber\\
	&\qquad\quad-36\left(\textup{Li}_3\big(\tfrac{3-i\sqrt{3}}{6}\big)+\textup{Li}_3\big(\tfrac{3+i\sqrt{3}}{6}\big) \right) \bigg].
\end{align}
Moreover, differentiating \eqref{fe-general} $k$ times with respect to $z$ and then specializing $z$ with Re$(z)>1$ produces further interesting identities. The general identity for $k=1$ and $x=1$ is given in Section \ref{extend}. Two special cases are given below:
\begin{align}\label{loglog2z2}		
	\int_{0}^{1} \frac{\log(\tfrac{1}{y})\log(\log(\tfrac{1}{y}))\log(1+y)}{1+y} \, dy =  \frac{1}{8} (1-\gamma)\zeta(3)
	- \frac{1}{2} \sum_{n=1}^{\infty} \sum_{m=1}^{\infty}  \frac{(-1)^{m+n}\log(m+n)}{mn(n+m)},
\end{align}
whereas if $\chi$ is the non-trivial Dirichlet character modulo $4$, we obtain
\begin{align}\label{loglog4z2}		
	\int_{0}^{1} \frac{\log(\tfrac{1}{y})\log(\log(\tfrac{1}{y}))\tan^{-1}(y)}{1+y^2} \, dy = \frac{(1-\gamma)}{16}\left(4\pi G-7\zeta(3)\right)
	- \frac{1}{2} \sum_{n=0}^{\infty} \sum_{m=0}^{\infty}  \frac{(-1)^{m+n}\log(2m+2n+2)}{(2m+1)(2n+1)(2m+2n+2)},
\end{align}
where $G$ is Catalan's constant. 
Vardi \cite{vardi} (see also \cite[p.~570, formula \textbf{4.325.4}]{gr}) has proved
\begin{align}\label{vardi-integral1}		
	\int_{0}^{1} \frac{\log(\log(\tfrac{1}{y}))}{1+y^2} \, dy =\frac{\pi}{2}\log\left(\frac{\Gamma(3/4)}{\Gamma(1/4)}\sqrt{2\pi}\right),
\end{align}
using techniques from analytic number theory.  Even though the integrals in \eqref{loglog2z2} and \eqref{loglog4z2} are reminiscent to that in \eqref{vardi-integral1}, they are quite different in nature.

We also mention another striking feature of \eqref{fe-primitive} in junction with \eqref{lz chi integral}. Let $x=1$ and $\chi_1=\chi_2=\bar\chi$ in \eqref{fe-primitive}, where $\chi$ is a primitive Dirichlet character modulo $r$. Then we obtain an integral representation for a character analogue of $\Theta(z, 1)$, namely,
\begin{align}\label{int rep char theta}
 \sum_{n=1}^{\infty} \sum_{m=1}^{\infty} \frac{\chi(m) \chi(n)}{m n (n+m)^{z-1}}=\frac{-2}{G(1, \bar\chi)\Gamma(z)} \int_{0}^{1} \frac{\left( \log \left( \frac{1}{t} \right) \right)^{z-1}}{t(1-t^r)} \left( \sum_{\ell=1}^{r}\chi(\ell)t^{\ell}\right) \left(  \sum_{j=1}^{r} \bar\chi(j) \log \left( 1-\zeta_r^j t\right) \right) \, dt.
\end{align}
The special case of \eqref{int rep char theta} where we let $z=1$ is considered in Corollary \ref{from introduction}.

Notice the similarity between the integral representation for $L(z, \chi)$ given in \eqref{lz chi integral} and the one for the character analogue of $\Theta(z, 1)$ given above. However, while \eqref{int rep char theta} admits a generalization in the parameter $x$ in the form of the functional equation given in \eqref{fe-primitive}, equation \eqref{lz chi integral} lacks such extension. 

In this regard, we mention that an integral representation for the left-hand side of \eqref{int rep char theta} given by Bailey and Borwein \cite[Equations (61), (62)]{bailey-borwein} for some specific characters contains two sums of the type $ \sum_{j=1}^{r} \bar\chi(j) \log \left( 1-\zeta_r^j t\right)$ but it differs from the one obtained here in that it does not  contain the Fekete polynomial $\sum_{\ell=1}^{r}\chi(\ell)t^{\ell}$.

	\section{Properties of the generalized Mordell-Tornheim zeta function $\Theta(z, x)$}
		
	\subsection{Towards the Laurent series expansion of $\Theta(z, x)$}\label{properties of MT}
	The principal part and the constant term in the Laurent series expansion of $\Theta(z, x)$ around $z=1$ is obtained here.
	\begin{proof}[Theorem \textup{\ref{principal part}}][]
		For $x, y>0$ and $\textup{Re}(z)>0$ ,
		\begin{align*}
			\int_{0}^{\infty} t^{z-1} e^{-xt-yt} \, dt = (x+y)^{-z} \Gamma(z).
		\end{align*} 
		Replace $y$ by $ny$, divide both sides by $n$ and sum over $n$ from $1$ to $\infty$ to get
		\begin{align*}
			\int_{0}^{\infty} t^{z-1} e^{-xt} \log(1-e^{-yt}) \, dt &= -\Gamma(z) \sum_{n=1}^{\infty} \frac{1}{n(x+ny)^{z}}.
		\end{align*}
		Since
		\begin{align}\label{before summing}
			\int_{0}^{\infty} t^{z-1} e^{-xt} \log(t)  \, dt = x^{-z} \Gamma(z) \psi(z) - x^{-z} \log(x) \Gamma(z),
		\end{align}
		we have
		\begin{align*}
			\int_{0}^{\infty} t^{z-1} e^{-xt} \log\left(\frac{1-e^{-t}}{t} \right) \, dt= -\Gamma(z) \sum_{n=1}^{\infty} \frac{1}{n(x+n)^{z}} -x^{-z} \Gamma(z) \psi(z) + x^{-z} \log(x) \Gamma(z).
		\end{align*} 
		Now replace $x$ by $mx$, divide both sides of the equation by $m$ and sum over $m$ from $1$ to $\infty$ to get for Re$(z)>0$,
		\begin{align*}
			&\frac{1}{\Gamma(z)}\int_{0}^{\infty} t^{z-1} \log\left(1-e^{-xt} \right)  \log\left(\frac{1-e^{-t}}{t} \right) \, dt \\
			&=  \Theta(z+1, x) +x^{-z} \psi(z) \zeta(z+1) +x^{-z} \zeta'(z+1) -x^{-z} \log(x) \zeta(z+1),
		\end{align*}
	where $\Theta(z, x)$ is defined in \eqref{Thetadef}. Now interchange the order of summation (which is valid due to absolute convergence) and then replace $z$ by $z-1$ so that for $\textup{Re}(z)>1$,
\begin{align}\label{befend}
	&\frac{1}{\Gamma(z-1)}\int_{0}^{\infty} t^{z-2} \log\left(1-e^{-xt} \right)  \log\left(\frac{1-e^{-t}}{t} \right) \, dt\nonumber\\ 
	&= \Theta(z, x) +x^{-(z-1)} \psi(z-1) \zeta(z)   +x^{-(z-1)} \zeta'(z) -x^{-(z-1)} \log(x) \zeta(z) . 
\end{align}		
Denote the integral on the left-hand side of \eqref{befend} by $E(z, x)$.  Our next task is show that $E(z, x)$ is uniformly convergent in the region Re$(z)\geq\epsilon$ for every $\epsilon>0$. This follows from the facts that as $t\to\infty$, 
\begin{align*}
\log\left(\frac{1-e^{-t}}{t} \right) =O(\log(t)),\hspace{3mm} t^{z-2} \log\left(1-e^{-xt} \right) =O\left(\frac{t^{-M}}{\log(t)}\right),
\end{align*}
for $M>1$, and as $t\to0$,
\begin{equation}\label{bounds}
	\log\left(\frac{1-e^{-t}}{t}\right)=O(t),\hspace{4mm} \log\left(1-e^{-xt}\right)=O_x\left(t^{-\epsilon/2}\right),
\end{equation}
This also shows that $E(1, x)$ is finite. 
Therefore, letting $z\to1$ in \eqref{befend}, we see that the left-hand side tends to $0$. The proof of \eqref{principal part equation} is then completed upon invoking the  well-known power/Laurent series expansions of the elementary/special functions occurring on the right-hand side of \eqref{befend} as $z\to 1$, and thereby observing that
\begin{align*}
&x^{-(z-1)} \psi(z-1) \zeta(z)+x^{-(z-1)} \zeta'(z) -x^{-(z-1)} \log(x) \zeta(z)\\
&=\frac{-2}{(z-1)^2}+\frac{-2\gamma+\log(x)}{z-1}-\gamma^2+\gamma\log(x)+\frac{\pi^2}{6}+O(|z-1|).
\end{align*}
	\end{proof}

	\subsection{The auxiliary integrals $H_1(z, x)$ and $H_2(z, x)$ and functional equations}\label{three-term Feq of theta}
	
 Two auxiliary integrals $H_1(z, x)$ and $H_2(z, x)$ respectively defined in \eqref{h1} and \eqref{h2} play a significant role 	in the proof of Theorem \ref{theta-3t}. In what follows, we first show how they turn up in our analysis.
	
	Let $x=1$ in \eqref{befend}, then subtract the resulting equation from \eqref{befend}, and the multiply both sides to obtain, for Re$(z)>1$,
		\begin{align}\label{izx}
		I(z, x)&= \frac{1}{(z-1)} \left( \Theta(z, x) +x^{-(z-1)} \psi(z-1) \zeta(z)  +x^{-(z-1)} \zeta'(z)\right. \notag \\
		& \hspace{20mm}\left. -x^{-(z-1)} \log(x) \zeta(z) -\Theta(z, 1)- \psi(z-1) \zeta(z) - \zeta'(z) \right) ,
	\end{align}
	where
	\begin{align}\label{int izx}
		I(z, x):=\frac{1}{\Gamma(z)}\int_{0}^{\infty} t^{z-2} \log\left(\frac{1-e^{-xt}}{1-e^{-t}} \right)  \log\left(\frac{1-e^{-t}}{t} \right) \, dt
	\end{align}
	The integral $I(z, x)$ has nice properties. To see this, use the elementary identity
	$ab=\frac{1}{2}\left\{(a+b)^2-a^2-b^2\right\}$ to get
	\begin{align}\label{i=h1-h2}
		I(z, x) =H_{1}(z, x)-H_{2}(z, x),
	\end{align}
	where
	\begin{align}
		H_1(z, x)&:=\frac{1}{2\Gamma(z)}\int_{0}^{\infty}t^{z-1}\left\{\log^{2}\left(\frac{1-e^{-xt}}{t}\right)-\log^{2}\left(\frac{1-e^{-t}}{t}\right)\right\}\, \frac{dt}{t},\label{h1}\\
		H_2(z, x)&:=\frac{1}{2\Gamma(z)}\int_{0}^{\infty}t^{z-1}\log^{2}\left(\frac{1-e^{-xt}}{1-e^{-t}}\right)\, \frac{dt}{t}.\label{h2}
	\end{align}%
	That the second integral $H_2(z, x)$ converges absolutely for  Re$(z)>1$ can be argued from the fact that
	\begin{equation}\label{log-asy}
	\log\left(\frac{1-e^{-xt}}{1-e^{-t}}\right)=
	\begin{cases}
	O(\log(x)),\hspace{5mm}\text{as}\hspace{1mm}t\to0^{+},\\
	O(e^{-\textup{min}(1, x)t})\hspace{1mm},\text{as}\hspace{1mm}t\to\infty.
	\end{cases}
	\end{equation}
	For $H_1(z, x)$, note that as $t\to0$,
	\begin{equation*}
\log\left(\frac{1-e^{-t}}{t}\right)=O(t),\hspace{4mm} \log\left(\frac{1-e^{-xt}}{t}\right)=O_x(1).
	\end{equation*}
Hence
\begin{equation*}
\log^{2}\left(\frac{1-e^{-xt}}{t}\right)-\log^{2}\left(\frac{1-e^{-t}}{t}\right)=O_x(1)
\end{equation*}
as $t\to0$, so that the convergence at $0$ is secured as long as Re$(z)>1$. Also, as $t\to\infty$,
\begin{align*}
	\log^{2}\left(\frac{1-e^{-xt}}{t}\right)-\log^{2}\left(\frac{1-e^{-t}}{t}\right)&=\log\left(\frac{1-e^{-xt}}{1-e^{-t}}\right)\left(\log(1-e^{-xt})+\log(1-e^{-t})-2\log(t)\right)\\
	&=O\left(e^{-\frac{1}{2}\textup{min}(1, x)t}\right),
\end{align*}
where we used \eqref{log-asy}. This shows that $H_1(z, x)$ also converges absolutely for Re$(z)>1$.
	
Theorem \ref{almost} below gives ``almost'' closed-form evaluations of these integrals $H_1(z, x)$ and $H_2(z, x)$. 
\begin{theorem}\label{almost}
	Let $\Phi(z, x)$ be defined in \eqref{hhf-def}. 
	For $\textup{Re}(z)>1$ and $x>0$,
\begin{align}
		H_1(z, x)&= \frac{x^{1-z}-1}{z-1} (\Phi(z,1)-\zeta(z+1)+\psi(z) \zeta(z) + \zeta'(z)) -\frac{x^{1-z}}{z-1}\zeta(z) \log(x),\label{h1eval}\\
		H_2(z, x)&=\frac{1}{z-1}((x+1)^{1-z}-x^{1-z}-1)\zeta(z+1)\nonumber\\
		&\quad-\frac{1}{\Gamma(z)}\int_{0}^{\infty}t^{z-2}\left\{\textup{Li}_2\left( \frac{e^{-(x+1)t}-e^{-xt}}{1-e^{-xt}} \right)+\textup{Li}_2\left( \frac{e^{-(x+1)t}-e^{-t}}{1-e^{-t}} \right)\right\}dt.\label{h2eval}
	\end{align}
\end{theorem} 	
\begin{proof}
	
	First we evaluate $H_1(z,x)$. By the fundamental theorem of calculus, 
	\begin{align*}
	\log^{2}\left(\frac{1-e^{-xt}}{t}\right)-\log^{2}\left(\frac{1-e^{-t}}{t}\right)&=\int_{1}^{x}\frac{d}{dy}\log^{2}\left(\frac{1-e^{-yt}}{t}\right)\, dy\\
	&=2t\int_{1}^{x} \log\left(\frac{1-e^{-yt}}{t}\right) \frac{1}{e^{yt}-1}\, dy .
	\end{align*}
	Hence,
\begin{align}\label{H1-beginning}
	H_1(z,x)&= \frac{1}{\Gamma(z)} \int_{0}^{\infty}  t^{z-1}\int_{1}^{x} \log\left(\frac{1-e^{-yt}}{t}\right) \frac{1}{e^{yt}-1}\, dy ~dt\notag\\
	&= \frac{1}{\Gamma(z)} \int_{1}^{x} \int_{0}^{\infty} t^{z-1} \log\left(\frac{1-e^{-yt}}{t}\right) \frac{1}{e^{yt}-1} \,dt~dy\notag\\
	&= \frac{1}{\Gamma(z)} \int_{1}^{x} y^{-z} \int_{0}^{\infty} t^{z-1} \left(\log(1-e^{-t})+\log\left(\frac{y}{t}\right)\right) \frac{1}{e^{t }-1} \, dt~dy\notag \\
	&= \frac{(x^{1-z}-1)}{\Gamma(z)(1-z)} \int_{0}^{\infty} t^{z-1} \left(\log(1-e^{-t}) -\log(t) \right) \frac{1}{e^{t }-1} \, dt +  \frac{1}{\Gamma(z)} \int_{0}^{\infty} \frac{t^{z-1}\, dt}{e^{t }-1} \int_{1}^{x} y^{-z} \log(y) \, dy,
	\end{align}
where in the second step, we used Fubini's theorem to interchange the order of integration.

Next, from \cite[p.~609, Equation \textbf{25.11.27}]{nist}, for Re$(z)>-1, z\neq1$, and Re$(a)>0$, we have
\begin{align*}
\zeta(z,a)- \frac{a^{1-z}}{z-1} &= \frac{a^{-z}}{2}  + \frac{1}{\Gamma(z)} \int_{0}^{\infty} t^{z-1} e^{-at} \left( \frac{1}{e^t-1} -\frac{1}{t} +\frac{1}{2} \right) \, dt,
\end{align*}
which can be used to prove that\footnote{For $z=1$, this was derived by Zagier \cite[Equation (7.2)]{zagier}.} for Re$(z)>0$,
\begin{align}
	\Phi(z, x)= \frac{-1}{\Gamma(z)} \int_{0}^{\infty} t^{z-1} \left( \frac{1}{1-e^{-t}} -\frac{1}{t} \right) \log(1-e^{-xt}) \, dt ,\label{almost integral3}
\end{align}
where $\Phi(z, x)$ is defined in \eqref{hhf-def}. Now take $x=1$ and   Re$(z)>1$ in \eqref{almost integral3} so as to get upon simplification
\begin{equation}\label{zeta zn over n}
	\int_{0}^{\infty}\frac{t^{z-1}\log(1-e^{-t})}{e^t-1}\, dt=-\Gamma(z)\sum_{n=1}^{\infty}\frac{\zeta(z, n)}{n}+\Gamma(z)\zeta(z+1).
\end{equation}
Replace $x$ by $n$ in \eqref{before summing} and sum over $n$ from $1$ to $\infty$ to get
\begin{equation}\label{diff gamma zeta}
	\int_{0}^{\infty}\frac{t^{z-1}\log(t)}{e^t-1}\, dt=\Gamma(z)\left(\psi(z)\zeta(z)+\zeta'(z)\right).
\end{equation}
Moreover,
\begin{equation}\label{indefinite}
\int_{1}^{x} y^{-z} \log(y) \, dy=\frac{x^{-z}}{(z-1)^2}\left(x^{z}-x+x(1-z)\log(x)\right),
\end{equation}
Hence, from \eqref{H1-beginning}, \eqref{zeta zn over n}, \eqref{diff gamma zeta}, and \eqref{indefinite}, we arrive at \eqref{h1eval} since 
\begin{align*}
	\sum_{n=1}^{\infty} \frac{\zeta(z, n)}{n}= \Phi(z,1)+\frac{\zeta(z)}{z-1}.
\end{align*}

To prove \eqref{h2eval}, we begin with Hill's formula for the dilogarithm whch is given for $0<\theta,\phi<1$ by \cite[p.~176]{zagier} (see, also, \cite[p.~9, Equation (1.25)]{lewin})
	\begin{align*}
		\frac{1}{2} \log^2\left( \frac{1-\theta}{1-\phi} \right)= \textup{Li}_2 (\theta \phi) -\textup{Li}_2(\theta) -\textup{Li}_2(\phi) -\textup{Li}_2\left( \frac{\theta\phi-\theta}{1-\theta} \right)- \textup{Li}_2\left( \frac{\theta\phi-\phi}{1-\phi} \right).
	\end{align*}
	Invoke this formula with $\theta=e^{-xt}$ and $\phi=e^{-t}$, where $x>0$ and $t>0$, so that
	\begin{align*}
		H_2(z, x)&=\frac{1}{\Gamma(z)}\int_{0}^{\infty}t^{z-2}\bigg\{\textup{Li}_2 (e^{-(x+1)t}) -\textup{Li}_2(e^{-xt}) -\textup{Li}_2(e^{-t})\nonumber\\ &\qquad\qquad\qquad\qquad-\textup{Li}_2\left( \frac{e^{-(x+1)t}-e^{-xt}}{1-e^{-xt}} \right)- \textup{Li}_2\left( \frac{e^{-(x+1)t}-e^{-t}}{1-e^{-t}} \right)\bigg\}\, dt.
	\end{align*}
	For each of the first three instances of dilogarithm, one can use its series definition from \eqref{li2z}, interchange the order of summation and integration to have upon simplification,
	\begin{align*}
		\frac{1}{\Gamma(z)}\int_{0}^{\infty}t^{z-2}\left(\textup{Li}_2 (e^{-(x+1)t}) -\textup{Li}_2(e^{-xt}) -\textup{Li}_2(e^{-t})\right)\, dt=\frac{1}{z-1}((x+1)^{1-z}-x^{1-z}-1)\zeta(z+1).
	\end{align*}
	This establishes the result.
\end{proof}

Next, we obtain two- and three-term functional equations for $H_1(z, x)$ and $H_2(z, x)$. 
	\begin{theorem}
	For $\textup{Re}(z)>1$ and $x>0$,
	\begin{align}
		H_1(z, x)&+x^{1-z}H_1\left(z,\frac{1}{x}\right)=\frac{1}{z-1}(1-x^{1-z})\log(x)\zeta(z),\label{2 term H1}\\
		H_1(z, x)&-	H_1(z, x+1)-(x+1)^{1-z}	H_1\left(z, \frac{x}{x+1}\right)
		=\frac{1}{z-1}((x+1)^{1-z}-x^{1-z})\log(x+1)\zeta(z).\label{3 term H1}
	\end{align}
\end{theorem}
\begin{proof}
Replace $x$ by $1/x$ in the definition of $H_1(z, x)$ in \eqref{h1} and then employ the change of variable $t=xu$ so that
\begin{align*}
H_1\left(z, \frac{1}{x}\right)&=\frac{x^{z-1}}{2\Gamma(z)}\int_{0}^{\infty}u^{z-2}\left\{\log^{2}\left(\frac{1-e^{-u}}{xu}\right)-\log^{2}\left(\frac{1-e^{-xu}}{xu}\right)\right\}\, du\nonumber\\
&=\frac{x^{z-1}}{2\Gamma(z)}\int_{0}^{\infty}u^{z-2}\left\{\log^{2}\left(\frac{1-e^{-u}}{u}\right)-\log^{2}\left(\frac{1-e^{-xu}}{u}\right)+2\log(x)\log\left(\frac{1-e^{-xt}}{1-e^{-t}}\right)\right\}\, du\nonumber\\
&=-x^{z-1}H_1(z, x)-\frac{1}{z-1}x^{z-1}\left(x^{1-z}-1\right)\log(x)\zeta(z),
\end{align*}
where the second expression on the extreme right side above results from expanding $\log(1-e^{-xt})$ (or $\log(1-e^{-t})$) as a power series, interchanging the order of summation and integration and then writing the integral as $(nx)^{1-z}\Gamma(z-1)$ (or $n^{1-z}\Gamma(z-1)$). This results in \eqref{2 term H1}.

To prove \eqref{3 term H1}, observe that using the definition of $H_1(z, x)$ in \eqref{h1} three times, we have
\begin{align*}
&H_1(z, x)-	H_1(z, x+1)+x^{1-z}	H_1\left(z, \frac{x+1}{x}\right)\nonumber\\
&=\frac{1}{2\Gamma(z)}\int_{0}^{\infty}t^{z-2}\left\{\log^{2}\left(\frac{1-e^{-xt}}{t}\right)-\log^{2}\left(\frac{1-e^{-(x+1)t}}{t}\right)\right\}\, dt\nonumber\\
&\quad+\frac{x^{1-z}}{2\Gamma(z)}\int_{0}^{\infty}t^{z-2}\left\{\log^{2}\left(\frac{1-e^{-\left(\frac{x+1}{x}\right)t}}{t}\right)-\log^{2}\left(\frac{1-e^{-t}}{t}\right)\right\}\, dt. 
\end{align*}
Employing the change of variable $t=xu$ in the second integral, we see that
\begin{align}\label{to combine}
&\frac{x^{1-z}}{2\Gamma(z)}\int_{0}^{\infty}t^{z-2}\left\{\log^{2}\left(\frac{1-e^{-\left(\frac{x+1}{x}\right)t}}{t}\right)-\log^{2}\left(\frac{1-e^{-t}}{t}\right)\right\}\nonumber\\
&=\frac{1}{2\Gamma(z)}\int_{0}^{\infty}u^{z-2}\left\{\log^{2}\left(\frac{1-e^{-(x+1)u}}{xu}\right)-\log^{2}\left(\frac{1-e^{-xu}}{xu}\right)\right\}\, du
\end{align}
But
\begin{align}\label{to subs}
\log^{2}\left(\frac{1-e^{-(x+1)u}}{xu}\right)-\log^{2}\left(\frac{1-e^{-xu}}{xu}\right)&=\log^{2}\left(\frac{1-e^{-(x+1)u}}{u}\right)-\log^{2}\left(\frac{1-e^{-xu}}{u}\right)\nonumber\\
&\quad-2\log(x)\log\left(\frac{1-e^{(x+1)u}}{1-e^{-xu}}\right).
\end{align}
Substituting \eqref{to subs} in \eqref{to combine} and then putting the resulting expression in \eqref{3 term beginning}, we find that
\begin{align}\label{3 term beginning}
	H_1(z, x)-	H_1(z, x+1)+x^{1-z}	H_1\left(z, \frac{x+1}{x}\right)
	&=-\frac{\log(x)}{\Gamma(z)}\int_{0}^{\infty}u^{z-2}\log\left(\frac{1-e^{(x+1)u}}{1-e^{-xu}}\right)\, du\nonumber\\
	&=\frac{1}{z-1}\log(x)\left((x+1)^{1-z}-x^{1-z}\right)\zeta(z).
\end{align}
In order to obtain \eqref{3 term H1}, we use \eqref{2 term H1} with $x$ replaced by $\frac{x+1}{x}$ so that
\begin{align*}
H_1\left(z, \frac{x+1}{x}\right)=-\left(\frac{x+1}{x}\right)^{1-z}H_1\left(z, \frac{x}{x+1}\right)+\frac{1}{z-1}\log\left(\frac{x+1}{x}\right)\left(1-\left(\frac{x+1}{x}\right)^{1-z}\right)\zeta(z),
\end{align*}
and substitute this expression in \eqref{3 term beginning}. This results in \eqref{3 term H1} after simplification.
\end{proof}
\begin{theorem}
	For $\textup{Re}(z)>1$ and $x>0$,
	\begin{align}
		&H_2(z, x)=x^{1-z}H_2\left(z,\frac{1}{x}\right),\label{H2 2 term} \\
		&H_2(z, x)-	H_2(z, x+1)-(x+1)^{1-z}	H_2\left(z, \frac{x}{x+1}\right)\nonumber\\
		&=\frac{1}{z-1}((x+1)^{1-z}-x^{1-z}-1)\zeta(z+1)\nonumber\\
		&\quad+\frac{1}{\Gamma(z)}\int_{0}^{\infty}t^{z-2}\left\{\textup{Li}_{2}\left(\frac{e^{-(x+1)t}-e^{-xt}}{e^{-(x+1)t}-1}\right)+\textup{Li}_{2}\left(\frac{e^{-(x+1)t}-e^{-t}}{e^{-(x+1)t}-1}\right)\right\}\, dt.\label{H2 3 term}
	\end{align}
\end{theorem}	
\begin{proof}
The two-term functional equation for $H_2(z, x)$, that is \eqref{H2 2 term}, follows readily upon letting $x$ by $1/x$ in \eqref{h2} and employing the change of variable $t=xu$ in the resulting equation.

Next, we prove \eqref{H2 3 term}. Using \cite[Equation (1.12)]{lewin},
\begin{align}\label{to be also referred}
\textup{Li}_2(w)+\textup{Li}_2\left(\frac{-w}{1-w}\right)=-\frac{1}{2}\log^{2}(1-w)\hspace{5mm}(w<1).
\end{align}
twice, once with $w=\displaystyle\frac{e^{-(x+1)t}-e^{-xt}}{1-e^{-xt}}$, and again, with $w=\displaystyle\frac{e^{-(x+1)t}-e^{-t}}{1-e^{-t}} $, we see that
\begin{align}
	\textup{Li}_2\left( \frac{e^{-(x+1)t}-e^{-xt}}{1-e^{-xt}} \right)+ \textup{Li}_2\left( \frac{e^{-(x+1)t}-e^{-t}}{1-e^{-t}} \right) \notag  &=-\frac{1}{2} \left( \log^2 \left( \frac{1-e^{-(x+1)t}}{1-e^{-xt}} \right)+\log^2 \left( \frac{1-e^{-(x+1)t}}{1-e^{-t}} \right)  \right)\\
	&\quad-  \textup{Li}_2\left( \frac{e^{-(x+1)t}-e^{-t}}{e^{-(x+1)t}-1} \right)- \textup{Li}_2\left( \frac{e^{-(x+1)t}-e^{-xt}}{e^{-(x+1)t}-1} \right). \label{Li formula for 3 term}
\end{align}
Thus, from \eqref{h2eval} and \eqref{Li formula for 3 term}, 
\begin{align}\label{almost there}
H_2(z, x)&=\frac{1}{z-1}((x+1)^{1-z}-x^{1-z}-1)\zeta(z+1)\nonumber\\
&\quad+\frac{1}{2\Gamma(z)}\int_{0}^{\infty}t^{z-2}\log^2 \left( \frac{1-e^{-(x+1)t}}{1-e^{-xt}} \right)\, dt+\frac{1}{2\Gamma(z)}\int_{0}^{\infty}t^{z-2}\log^2 \left( \frac{1-e^{-(x+1)t}}{1-e^{-t}} \right) \, dt\nonumber\\
&\quad+\int_{0}^{\infty}\frac{t^{z-2}}{\Gamma(z)}\left\{\textup{Li}_2\left( \frac{e^{-(x+1)t}-e^{-t}}{e^{-(x+1)t}-1} \right)+ \textup{Li}_2\left( \frac{e^{-(x+1)t}-e^{-xt}}{e^{-(x+1)t}-1} \right)\right\}\, dt
\end{align}
Now
\begin{align}
	\frac{1}{2\Gamma(z)}\int_{0}^{\infty}t^{z-2}\log^2 \left( \frac{1-e^{-(x+1)t}}{1-e^{-xt}} \right)\, dt &=\frac{x^{1-z}}{2\Gamma(z)}	\int_{0}^{\infty}t^{z-2}\log^2 \left( \frac{1-e^{-\left(\frac{x+1}{x}\right)t}}{1-e^{-t}} \right)\, dt
	=x^{1-z}H_2\left(\frac{x+1}{x}\right),\label{there}\\
	\frac{1}{2\Gamma(z)}\int_{0}^{\infty}t^{z-2}\log^2 \left( \frac{1-e^{-(x+1)t}}{1-e^{-t}} \right)\, dt &=H_2(x+1).	\label{there1}
\end{align}	
Thus from \eqref{almost there}, \eqref{there}, \eqref{there1} and the two-term equation $H_2\left(z,\frac{x+1}{x}\right)=\left(\frac{x+1}{x}\right)^{1-z}H_2\left(z,\frac{x}{x+1}\right)$, resulting from \eqref{H2 2 term}, we arrive at \eqref{H2 3 term}. 
\end{proof}
We are now ready to prove the three-term functional equation for $\Theta(z, x)$. It is only because of the machinery developed so far that the proof appears to be short and straightforward.

	\begin{proof}[Theorem \textup{\ref{theta-3t}}][]
	From \eqref{i=h1-h2}, \eqref{3 term H1} and \eqref{H2 3 term}, it is clear that $I(z, x)$ also satisfies a three-term functional equation. This fact along with the relation between $I(z, x)$ and $\Theta(z, x)$ given in \eqref{izx} implies that $\Theta(z, x)$ also satisfies a three-term functional equation, which, after simplification and a lot of cancellation, results in \eqref{theta-3t eqn}. 
	\end{proof}
	\section{The Herglotz-Hurwitz function $\Phi(z, x)$ and the Ishibashi function $\Phi_k(x)$}\label{hhf}
	
	Vlasenko and Zagier \cite{vz} considered the higher Herglotz functions 	
	\begin{equation*}
		F_k(x):=\sum_{n=1}^{\infty}\frac{\psi(nx)}{n^{k}}\hspace{5mm}\left(k\in\mathbb{N},k>1, x\in\mathbb{C}\backslash(-\infty,0]\right),
	\end{equation*}
	and obtained \cite[Equations (11), (12)]{vz} their two- and three-term functional equations, although their two-term relation was derived earlier by Maier \cite[p.~114, Equation (3)]{maier}. There are numerous other generalizations of $F(x)$. See  \cite{hhf1} for more details. 
	
	In what follows, we study a new generalization of the Herglotz-Zagier function $F(x)$ which is intimately connected with our generalization of the Mordell-Tornheim zeta function, namely, $\Theta(z, x)$.
	\subsection{The decomposition theorem for $\Theta(z, x)$: Two-term relation for $\Phi(z, x)$}\label{decomp}

	\begin{proof}[Theorem \textup{\ref{decompostion theorem}}][]
	Recalling the definition of $\Phi(z, x)$ from \eqref{hhf-def}, we see that for Re$(z)>1$,
	\begin{align*}
	\Phi(z, x)+x^{1-z}\Phi\left(z, \frac{1}{x}\right)&=\sum_{m=1}^{\infty} \frac{1}{m}\left( \zeta(z,mx)- \frac{(mx)^{1-z}}{z-1} \right)	+x^{1-z}\sum_{n=1}^{\infty} \frac{1}{n}\left( \zeta\left(z,\frac{n}{x}\right)- \frac{(\frac{n}{x})^{1-z}}{z-1} \right)	\\	
	&=\sum_{m=1}^{\infty} \sum_{n=0}^{\infty}\frac{1}{m(n+mx)^z}-\frac{x^{1-z}}{z-1}\zeta(z)+\sum_{n=1}^{\infty} \sum_{m=0}^{\infty}\frac{x^{1-z}}{n\left(m+\frac{n}{x}\right)^z}-\frac{\zeta(z) }{z-1}\\
&=\sum_{m=1}^{\infty} \sum_{n=1}^{\infty}\frac{1}{m(n+mx)^z}+\sum_{n=1}^{\infty} \sum_{m=1}^{\infty}\frac{x}{n\left(n+mx\right)^z}\\
&\quad-\frac{(1+x^{1-z})}{z-1}\zeta(z)+(x+x^{-z})\zeta(z+1),
\end{align*}
which leads to the right-hand side of \eqref{decomposition} upon interchanging the order of summation in the first double sum (which is valid due to absolute convergence) and then simplifying.
	\end{proof}

	\subsection{Two-term functional equation for $\Phi_k(x)$}\label{2fe-ishibashi}
An auxiliary polynomial important in the study of $\Phi_k(x)$ was studied by  Ishibashi \cite[Theorem 1]{ishibashi}. This polynomial in $\log(t)$ is defined by $S_k(t):=\sum_{j=0}^{k-1}a_{k, j}\log^{j}(t)$, where $a_{k, j}$ are recursively defined by
\begin{equation*}
	a_{k, j}=-\sum_{r=0}^{k-2}\binom{k-1}{r}\Gamma^{(k-r-1)}(1)a_{r+1, j}\hspace{5mm} (0\leq j\leq k-2),
\end{equation*}
and $a_{1,0}=1, a_{k, k-1}=1$. We call  $S_{k}(t)$, the \emph{$k^{\textup{th}}$ Ishibashi polynomial}. Observe that $S_1(t)=1$ and $S_2(t)=\g+\log(t)$.

In what follows, we derive a lemma needed for obtaining the two-term functional equation for $\Phi_k(x)$.

\begin{lemma}\label{Lemma coeff}
	\textup{(i)}	For $k,j \in \mathbb{N}$ such that $k-j \geq 1$, we have
	\begin{align}\label{ab}
		a_{k,j}= \binom{k-1}{j} \lim_{z \to 1} \frac{d^{k-j-1}}{dz^{k-j-1}} \left( \frac{1}{\Gamma(z)} \right).
	\end{align}
	\hspace{2.2cm}\textup{(ii)} 	For $k \in \mathbb{N}\cup \left\{0\right\}$, 
	\vspace{-5mm}
	\begin{align*}
		S_{k+1}(t)= \lim_{z \to 1} \frac{d^k}{dz^k} \left( \frac{t^{z-1}}{\Gamma(z)} \right).
	\end{align*}
\end{lemma}
\begin{proof}
	Let $b_{k,j}$ denote the right-hand side of \eqref{ab}. Clearly $b_{1,0}=b_{k,k-1}=1$ for any $k \in \mathbb{N}.$\\
	For $0\leq j\leq k-1$, consider
	\begin{align*}
		&-\sum_{r=0}^{k-1}\binom{k}{r}\Gamma^{(k-r)}(1)b_{r+1, j}= -\sum_{r=j}^{k-1}\binom{k}{r}\binom{r}{j} \lim_{z \to 1} \frac{d^{k-r}}{dz^{k-r}} \left( \Gamma(z) \right)  \lim_{z \to 1} \frac{d^{r-j}}{dz^{r-j}} \left( \frac{1}{\Gamma(z)} \right)
	\end{align*}
	since $\binom{r}{j}=0$ for $r<j$. Now replace $r$ by $r+j$ and then separate the $r=0$ term to get
	\begin{align*}
		-\sum_{r=0}^{k-1}\binom{k}{r}\Gamma^{(k-r)}(1)b_{r+1, j}
		&=b_{k+1,j} - \frac{k!}{j! (k-j)!} \sum_{r=0}^{k-j} \frac{(k-j)!}{r! (k-j-r)!} \lim_{z \to 1} \frac{d^{k-j-r}}{dz^{k-j-r}} \left( \Gamma(z) \right)  \lim_{z \to 1} \frac{d^{r}}{dz^{r}} \left( \frac{1}{\Gamma(z)} \right)\\
		&=b_{k+1,j} - \binom{k}{j}  \sum_{r=0}^{k-j} \binom{k-j}{r} \lim_{z \to 1} \frac{d^{k-j-r}}{dz^{k-j-r}} \left( \Gamma(z) \right)  \lim_{z \to 1} \frac{d^{r}}{dz^{r}} \left( \frac{1}{\Gamma(z)} \right)\\
		&=b_{k+1,j},
	\end{align*}
	where, in the last step, we used Leibnitz theorem and the fact that $k-j\geq1$. 
	Since $a_{k,j}$ and $b_{k,j}$ have the same initial conditions and satisfy the same recursive conditions, we see that $a_{k,j}=b_{k,j}$ for all $k \in \mathbb{N}$ and $j \in \mathbb{N} \cup \left\{0\right\}$.
	To prove part (ii) of the lemma, we use Leibnitz Theorem again so that for $k \in \mathbb{N} \cup \left\{0\right\}$,
	\begin{align*}
		\lim_{z\to1}\frac{d^k}{dz^k} \left( \frac{t^{z-1}}{\Gamma(z)} \right) = \lim_{z \to 1}\sum_{j=0}^{k} \binom{k}{j} \frac{d^j}{dz^j} (t^{z-1}) \frac{d^{k-j}}{dz^{k-j}} \left( \frac{1}{\Gamma(z)} \right)
		= \sum_{j=0}^{k} a_{k+1,j}  \log^j(t)
		= S_{k+1}(t),
	\end{align*}
	where in the second step, we used part (i) of this lemma.
\end{proof}

\noindent
	In the statement of Theorem \ref{thm-2fe-ishibashi}, we will encounter the constants $c_k, k\geq0,$ defined by
	\begin{align}\label{gz}
 \Gamma(s)\zeta(s)-\frac{1}{s-1}=\sum_{k=0}^{\infty}c_k(s-1)^k.
 \end{align}
 Arakawa and Kaneko \cite[Corollary 5]{arakawa-kaneko} have shown that
	\begin{equation}\label{ck}
	c_k:=\frac{1}{k!}\int_{0}^{\infty}\left\{\log^{k}(t)-\log^{k}\left(\frac{t}{1-e^{-t}}\right)\right\}\left(\frac{1}{e^t-1}-\frac{1}{t}\right)\, dt.	
	\end{equation}
We call them the \emph{Arakawa-Kaneko constants}. Note that $c_0=0$. Also, 
\begin{align*}
c_1=\int_{0}^{\infty}\left(\frac{1}{1-e^{-t}}-\frac{1}{t}\right)\log(1-e^{-t})\, dt-\int_{0}^{\infty}\log(1-e^{-t})\, dt=F(1)+\frac{\pi^2}{6}=-\frac{1}{2}\g^2+\frac{\pi^2}{12}-\g_1,
\end{align*}
where in the second step, we used the special case $z=1$ of \eqref{almost integral3}, and in the last step, we used \eqref{ef1}.

Before stating Theorem \ref{thm-2fe-ishibashi}, we derive some lemmas. The first one shows that the $j^{\textup{th}}$ Taylor coefficient of the Herglotz-Hurwitz function $\Phi(z, x)$ expanded as a power series in $z-1$ is essentially the $j+1$-th order Ishibashi-Herglotz function $\Phi_{j+1}(x)$.

\begin{lemma}\label{power series coefficient of hhf}
	Let $\Phi(z, x)$ and $\Phi_k(x)$ be defined in \eqref{hhf-def} and \eqref{ishibashi-def} respectively. Then for any $j\in\mathbb{N}\cup\{0\}$,
\begin{align*}
	\lim_{z \to 1}\frac{d^j}{dz^j} \Phi(z,x) = \frac{(-1)^{j+1}}{j+1} \Phi_{j+1}(x).
	\end{align*}
\end{lemma}
\begin{proof}
From the definition of $\Phi(z, x)$ in \eqref{hhf-def}, we have
\begin{align*}
 \Phi(z,x)=\sum_{n=1}^{\infty} \frac{1}{n}\left[\left( \zeta(z,nx)-\frac{1}{z-1}\right)+\left(\frac{1}{z-1} -\frac{(nx)^{1-z}}{z-1}\right)\right].
\end{align*}
Using the absolute and uniform convergence of the above series in Re$(z)\geq\epsilon$, for every $\epsilon>0$, we see that
\begin{align*}
	\lim_{z \to 1}\frac{d^j}{dz^j} \Phi(z,x)
	&=\sum_{n=1}^{\infty}\frac{1}{n}\Bigg[\lim_{z \to 1} \left( \frac{d^j}{dz^j}(\zeta(z,nx))+\frac{(-1)^{j+1}j!}{(z-1)^{j+1}}\right)\nonumber\\&\quad+\lim_{z \to 1} \left(\frac{(-1)^{j}j!}{(z-1)^{j+1}} -\sum_{\ell=0}^{j} \binom{j}{\ell} \frac{(-1)^{j}(j-\ell)! ( nx )^{1-z} \log^{\ell}(nx)}{(z-1)^{j-\ell+1}} \right)\Bigg]\nonumber \\
	&	=\sum_{n=1}^{\infty}\frac{1}{n}\Bigg[ (-1)^{j+1} \psi_j(nx)+\frac{(-1)^j \log^{j+1}(nx)}{ j+1}\Bigg] \nonumber\\
	&	= \frac{(-1)^{j+1}}{j+1} \Phi_{j+1}(nx),
\end{align*}
where in the second step above, we used \cite[Equations (4.8), (4.9), (4.11)]{dss}.
\end{proof}

\begin{lemma}\label{gammazetalemma}
	Let $j$ be a non-negative integer. Then 
	\begin{align}\label{gammazetalemma eqn}
		&\lim_{z \to 1 } \frac{d^j}{dz^j} \left(\frac{1}{\Gamma(z)} \frac{d}{dz}\left(\Gamma(z) \zeta(z) (x^{1-z}-1)\right)\right) \nonumber \\
		& = j \sum_{\ell=1}^{j} \frac{a_{j,\ell-1}}{\ell} \left\{\sum_{n=0}^{\ell -1} c_{\ell -n} \frac{(\ell+1)!}{(n+1)!} (-1)^{n+1} \log^{n+1}(x)  +\frac{(-1)^{\ell} }{\ell +2} \log^{\ell +2}(x)\right\} + \frac{1}{2} a_{j+1,0}\log^2(x).
	\end{align}
\end{lemma}
\begin{proof}
	Using Leibnitz's rule for successive differentiation in the first step, separating the $\ell=0$ term, and using Lemma \ref{Lemma coeff} and \eqref{gz} in the second and third steps below respectively, we see that
	\begin{align*}
		&\lim_{z \to 1 } \frac{d^j}{dz^j} \left(\frac{1}{\Gamma(z)} \frac{d}{dz}\left(\Gamma(z) \zeta(z) (x^{1-z}-1)\right)\right) \nonumber \\
		& = \lim_{z\to 1} \left[\sum_{\ell=1}^{j} \binom{j}{\ell} \left\{  \frac{d^{\ell +1}}{dz^{\ell+1}}  \left(\Gamma(z) \zeta(z) (x^{1-z}-1)\right) \frac{d^{j-\ell}}{dz^{j-\ell}} \left(\frac{1}{\Gamma(z)} \right) \right\} + \frac{d}{dz}  \left(\Gamma(z) \zeta(z) (x^{1-z}-1)\right) \frac{d^{j}}{dz^{j}} \left(\frac{1}{\Gamma(z)} \right)\right] \nonumber \\
		&   =  \sum_{\ell=1}^{j} \binom{j}{\ell} \left\{ \lim_{z \to 1}  \frac{d^{\ell +1}}{dz^{\ell+1}}  \left(\left(\Gamma(z) \zeta(z)-\frac{1}{z-1} \right)(x^{1-z}-1)+\frac{x^{1-z}-1}{z-1}\right)  \right\} \frac{a_{j, \ell-1}}{\binom{j-1}{\ell-1}} \nonumber \\
		&\quad  +\lim_{z \to 1}  \frac{d}{dz}  \left(\left(\Gamma(z) \zeta(z)-\frac{1}{z-1} \right)(x^{1-z}-1)+\frac{x^{1-z}-1}{z-1}\right) a_{j+1,0} \nonumber \\
		&   =  \sum_{\ell=1}^{j} \binom{j}{\ell} \left\{ \sum_{n=0}^{\infty} c_n  \lim_{z \to 1}  \frac{d^{\ell +1}}{dz^{\ell+1}}  (z-1)^n (x^{1-z}-1)+  \lim_{z \to 1}  \frac{d^{\ell +1}}{dz^{\ell+1}}\frac{\left(x^{1-z}-1\right)}{z-1} \right\} \frac{a_{j, \ell-1}}{\binom{j-1}{\ell-1}} + \frac{1}{2} a_{j+1,0}\log^2(x)\nonumber \\
		&   =  \sum_{\ell=1}^{j} \binom{j}{\ell} \left\{ \sum_{n=0}^{\infty} c_n \lim_{z \to 1 }\sum_{r=0}^{\ell+1} \binom{\ell+1}{r} \frac{d^r}{dz^r}(z-1)^n   \frac{d^{\ell +1-r}}{dz^{\ell+1-r}}  (x^{1-z}-1)+\frac{(-1)^{\ell+2}}{\ell+2} \log^{\ell+2}(x)\right\} \frac{a_{j, \ell-1}}{\binom{j-1}{\ell-1}}\nonumber\\
		&\quad+ \frac{1}{2} a_{j+1,0}\log^2(x).
	\end{align*}
	
	Note that 			$\lim_{z \to 1 }\frac{d^r}{dz^r}(z-1)^n =0$ unless $r=n$, and in that case it is equal to $n!$. Also, when $r=\ell+1$,   we have $ \lim_{z \to 1 }\frac{d^{\ell +1-r}}{dz^{\ell+1-r}}  (x^{1-z}-1)=0$, and is equal to $ (-1)^{\ell-n+1} \log^{\ell-n+1}(x)$ otherwise. 
	 Therefore,
	\begin{align*}
		&\lim_{z \to 1 } \frac{d^j}{dz^j} \left(\frac{1}{\Gamma(z)} \frac{d}{dz}\left(\Gamma(z) \zeta(z) (x^{1-z}-1)\right)\right) \nonumber \\
		& =  \sum_{\ell=1}^{j} \binom{j}{\ell} \left\{ \sum_{n=0}^{\ell} c_n \binom{\ell+1}{n} n!(-1)^{\ell-n+1} \log^{\ell-n+1}(x) +\frac{(-1)^{\ell}}{\ell+2} \log^{\ell+2}(x)\right\} \frac{a_{j, \ell-1}}{\binom{j-1}{\ell-1}}  +  \frac{1}{2} a_{j+1,0}\log^2(x).
	\end{align*}
	Replacing $n$ by $\ell-n$, we arrive at \eqref{gammazetalemma eqn}.
\end{proof}

 The two-term functional equation for involving the Ishibashi functions $\Phi_j(x), 1\leq j\leq k$, is stated next.
		\begin{theorem}\label{thm-2fe-ishibashi}
	Let $\Phi_k(x)$ be defined in \eqref{ishibashi-def}. For $k\geq0$ and $x>0$,
		\begin{align}\label{two-term eqn for Ishibashi}
		&\sum_{j=0}^k\frac{(-1)^{j+1}}{j+1}\binom{k}{j}\frac{\log^{k-j}(x)}{2^{k-j}}\left(\Phi_{j+1}(x)+(-1)^{k-j}\Phi_{j+1}\left(\frac{1}{x}\right)\right)\nonumber\\
		&=\sum_{j=0}^k\binom{k}{j}\frac{\log^{k-j}(x)}{2^{k-j}}\bigg[jL_{j-1}^*(x)-\frac{1}{2}a_{j+1, 0}\log^{2}(x)+\left(x+\frac{(-1)^{k-j}}{x}-2\right)\zeta^{(j)}(2)\nonumber\\
		&\quad-j\sum_{\ell=1}^{j}\frac{a_{j, \ell-1}}{\ell}\left\{\sum_{n=0}^{\ell-1}(-1)^{n+1}c_{\ell-n}\frac{(\ell+1)!}{(n+1)!}\log^{n+1}(x)+\frac{(-1)^{\ell}}{\ell+2}\log^{\ell+2}(x)\right\}+2\frac{(-1)^{j+1}}{j+1}\Phi_{j+1}(1)\bigg],
	\end{align}
	where, for any $k\in\mathbb{N}\cup\{-1, 0\}$,
	\begin{align}\label{lk*}
		L_k^*(x):=	
		\begin{cases}
	\lim\limits_{z \to 1}\displaystyle\frac{d^k}{dz^k} I(z, x),\text{if}\hspace{2mm}k\in\mathbb{N}\cup\{0\},\\
	0, \hspace{20mm}\text{if}\hspace{2mm}k=-1,
		\end{cases}
	\end{align}
	with $I(z, x)$ and $c_k$ defined in \eqref{int izx} and \eqref{ck} respectively.
\end{theorem}
\begin{proof}
Multiply both sides of \eqref{decomposition} by $x^{\frac{z-1}{2}}$, recall the definition of $\Theta(z, x)$ in \eqref{Thetadef}, and differentiate the resulting equation $k$ times with respect to $z$ using Leibnitz rule to see that for \textup{Re}$(z)>1$ and $x>0$,
	\begin{align}\label{afterdiff1}
		&x^{{z-1}{2}} \sum_{j=0}^{k} \binom{k}{j} \frac{\log^{k-j}(x)}{2^{k-j}} \frac{d^j}{dz^j} \Phi(z,x)+x^{\tfrac{1-z}{2}} \sum_{j=0}^{k} \binom{k}{j} \frac{\log^{k-j}\left(\frac{1}{x}\right)}{2^{k-j}} \frac{d^j}{dz^j} \Phi\left(z,\frac{1}{x}\right) \nonumber \\&= x^{\tfrac{z-1}{2}} \sum_{j=0}^{k} \binom{k}{j} \frac{\log^{k-j}(x)}{2^{k-j}} \frac{d^j}{dz^j}\left(	\Theta(z, x)-\frac{(1+x^{1-z})}{z-1}\zeta(z)\right) \nonumber \\&\quad+ x^{\tfrac{z-1}{2}} \sum_{j=0}^{k} \binom{k}{j} \frac{\log^{k-j}(x)}{2^{k-j}} \left(x+(-1)^{k-j}x^{-z}\right) \frac{d^j}{dz^j}\zeta(z+1).
	\end{align}
	We need to evaluate 
	\begin{align*}
	\lim_{z\to1}	\frac{d^j}{dz^j}\left(	\Theta(z, x)-\frac{(1+x^{1-z})}{z-1}\zeta(z)\right).
	\end{align*}
To that end, using the functional $\psi(z-1)=\psi(z)-\frac{1}{z-1}$ twice in \eqref{izx} and subtract $2\zeta(z)/(z-1)$ from both sides of the resulting equation to get
\begin{align*}
	 \Theta(z, x) -\frac{(1+x^{1-z}) \zeta(z)}{z-1}
	 &=(z-1)I(z, x)- \left\{x^{1-z} (\psi(z) \zeta(z) + \zeta'(z) - \log(x) \zeta(z))- \psi(z) \zeta(z) - \zeta'(z)\right\}\nonumber\\
	 &\quad+\Theta(z, 1)-\frac{2 \zeta(z)}{z-1},
\end{align*} 
where $I(z, x)$ is defined in \eqref{int izx}. Now using \eqref{decomposition} with $x=1$ , and observing that the expression in curly braces is nothing but $\frac{1}{\Gamma(z)}\frac{d}{dz}\left((x^{1-z}-1)\Gamma(z)\zeta(z)\right)$, we find that
\begin{align*}
	\Theta(z,x)-\frac{(1+x^{1-z}) \zeta(z)}{z-1}&=(z-1) I(z,x)- \frac{1}{\Gamma(z)} \frac{d}{dz} \left( (x^{1-z}-1)\Gamma(z)\zeta(z) \right)  +2 \Phi(z,1)-2 \zeta(z+1).
\end{align*}
Differentiating both sides of the above equation $j$ times with respect to $z$ and then letting $z \to 1 $, we get
\begin{align}\label{lastbutonestep}
	&\lim_{z \to 1} \frac{d^j}{dz^j}\left(	\Theta(z,x)-\frac{(x^{1-z}+1) \zeta(z)}{z-1}\right)\nonumber\\
	&= \lim_{z \to 1}\left( (z-1)  \frac{d^j}{dz^j} I(z,x)\right)+ j \lim_{z \to 1}\left(   \frac{d^{j-1}}{dz^{j-1}} I(z,x)\right)\nonumber\\
	&\quad-\lim_{z \to 1} \frac{d^j}{dz^j}\left( \frac{1}{\Gamma(z)} \frac{d}{dz} \left( (x^{1-z}-1)\Gamma(z)\zeta(z) \right)\right) +2\lim_{z \to 1} \frac{d^j}{dz^j}( \Phi(z,1)- \zeta(z+1)).
	\end{align}
	Note that $I(z, x)=\frac{1}{\Gamma(z)}\left(E(z, x)-E(z, 1)\right)$, where $E(z, x)$ is the integral occurring in \eqref{befend}. The latter was shown to be uniformly convergent in Re$(z)\geq\epsilon$ for every $\epsilon>0$ in the proof of Theorem \ref{principal part}. Hence $I(z, x)$ also converges uniformly in this region and hence represents an analytic function of $z$. This implies $\lim_{z \to 1} (z-1)  \frac{d^j}{dz^j} I(z,x)=0$. This, along with the definition of $L_{k}^{*}(x)$ in \eqref{lk*} and applications of Lemmas \ref{power series coefficient of hhf} and \ref{gammazetalemma}, implies that
	\begin{align*}
		&\lim_{z \to 1} \frac{d^j}{dz^j}\left(	\Theta(z,x)-\frac{(x^{1-z}+1) \zeta(z)}{z-1}\right)\nonumber\\
		&= j L_{j-1}^*(x)+2	\frac{(-1)^{j+1}}{j+1} \Phi_{j+1}(1)-2 \zeta^{j+1}(2)\nonumber\\
		&- 	j \sum_{\ell=1}^{j} \frac{a_{j,\ell-1}}{\ell} \left\{\sum_{n=0}^{\ell -1} c_{\ell -n} \frac{(\ell+1)!}{(n+1)!} (-1)^{n+1} \log^{n+1}(x)  +\frac{(-1)^{\ell} }{\ell +2} \log^{\ell +2}(x)\right\} -\frac{\log^2(x)}{2} a_{j+1,0}.
		\end{align*}
		Finally, letting $z\to 1$ in \eqref{afterdiff1}, substituting the above evaluation there, using Lemma \ref{power series coefficient of hhf} twice, and simplifying, we arrive at \eqref{two-term eqn for Ishibashi}. 
\end{proof}

An immediate special case of the above theorem is the two-term functional equation for $F(x)$ given by Zagier. 	
\begin{corollary}\label{k=0}
The functional equation in \eqref{fe2} holds.
\end{corollary}
\begin{proof}
	Put $k=0$ in Theorem \ref{thm-2fe-ishibashi} and simplify.
\end{proof}

\begin{corollary}\label{k=1 cor}
	The functional equation in \eqref{k=1} holds.
\end{corollary}
\begin{proof}
	Put $k=1$ in Theorem \ref{thm-2fe-ishibashi} and simplify.
\end{proof}
\subsubsection{An alternative representation for $L_{k}^*(x)$}

The result derived next expresses the limit $L_k^*(x)$ encountered in Theorem \ref{thm-2fe-ishibashi} as a finite sum involving the constants $a_{k+1, j}$ occurring in the definition of the Ishibashi polynomial $S_{k+1}(t)$ and certain integrals of logarithm. 

\begin{proposition}\label{another lk*}
Let $L_k^*(x)$ be defined in \eqref{lk*}. We have
	\begin{align}\label{lk**}
	L_k^*(x):=	\sum_{j=0}^{k}a_{k+1, j}\vspace{1mm}	\int_{0}^{\infty} \frac{\log^{j}(t)}{t} \log \left( \frac{1-e^{-xt}}{1-e^{-t}} \right) \log \left( \frac{1-e^{-t}}{t}  \right) \, dt.
\end{align}
\end{proposition}
\begin{proof}
	 Using Lemma \ref{Lemma coeff} and the discussion following \eqref{lastbutonestep}, we obtain
\begin{align*}
	L_k^*(x)=		\int_{0}^{\infty} \frac{S_{k+1}(t)}{t} \log \left( \frac{1-e^{-xt}}{1-e^{-t}} \right) \log \left( \frac{1-e^{-t}}{t}  \right) \, dt,
\end{align*}
which leads to \eqref{lk**} using the definition of the Ishibashi polynomial $S_{k+1}(t)$.
\end{proof}

\begin{remark}
	The above proposition and part (ii) of  Lemma \ref{Lemma coeff} justifies the use of the normalization factor $1/\Gamma(z)$ before the integrals in, say, \eqref{h1} and \eqref{h2}.
\end{remark}

	\subsection{Three-term functional equation involving derivatives of Ishibashi functions}

For $j\in\mathbb{N}$, let $\Phi_j(x)$ be the Ishibashi function defined in \eqref{ishibashi-def}. Note that
	\begin{align*}
	\Phi'_{j}(x) = j\sum_{n=1}^{\infty} \left(\psi'_{j-1}(nx)- \frac{\log^{j-1}(nx)}{nx}\right).
\end{align*}
Guinand \cite[p.~4]{apg3} derived the two-term functional equation for an infinite series involving $\psi'(x)$, namely, 
\begin{equation*}
\sum_{n=1}^{\infty}\left(\psi'(1+nx)-\frac{1}{nx}\right)=	\frac{1}{x^2}\sum_{n=1}^{\infty}\left(\psi'\left(1+\frac{n}{x}\right)-\frac{x}{n}\right)+\frac{\log(x)}{x},
\end{equation*}
which is really a two-term relation for $\Phi'_1(x)$ since 	$\psi(1+x)=	\psi(x) +1/x$. This result was generalized in \cite[Theorem 2.12]{dss} for the first derivatives of the generalized digamma functions $\psi_{j}(x)$. In what follows, we give a \emph{three-term functional equation} involving $\Phi'_j(x)$.
\begin{theorem}\label{Thm 3term-curious}
	Let $\gamma_j$ denote the $j^{\textup{th}}$ Stieltjes constant. 
For a non-negative integer $k$ and $x>1$, 
	\begin{align}\label{curious 3 term}
		&	 \sum_{j=0}^{k} \binom{k}{j}  \frac{(-1)^{j+1}}{j+1}  \left(\frac{\log(x)}{2}\right)^{k-j} \left(  \Phi'_{j+1}(x)-  \Phi'_{j+1}(x-1) + \frac{1}{x^2}   (-1)^{k-j} \Phi'_{j+1}\left(\frac{x-1}{x}\right)\right) \nonumber \\
		&=  \sum_{j=0}^{k} \binom{k}{j} \left(\frac{\log(x)}{2}\right)^{k-j} \sum_{\ell=0}^{j} \binom{j}{\ell} (-1)^{j} \gamma_{j-\ell} \left\{\frac{\log^\ell(x)}{x}-\frac{\log^\ell(x-1)}{x-1}+\frac{(-1)^{k-j}\log^\ell(\tfrac{x-1}{x})}{x (x-1)}\right\} \nonumber \\
		&+ \frac{1}{(k+1) x} \left\{\left(\frac{-\log(x)}{2}\right)^{k+1}-\left(\frac{\log(x)}{2}- \log(x-1)\right)^{k+1} \right\}.
	\end{align}
\end{theorem}

Specializing $k=0$ in Theorem \ref{Thm 3term-curious} and integrating the resulting equation with respect to $x$ produces Zagier's three-term relation for the Herglotz-Zagier function $F(x)$ given in \eqref{fe1}. Due to the presence of terms involving logarithm in front of the combination of $\Phi'_{j+1}(x)$, for $k\geq1$, we are unable to integrate \eqref{curious 3 term} with respect to $x$ to lead us to the three-term relation for the Ishibashi functions themselves. 

	\begin{proof}[Theorem \textup{\ref{Thm 3term-curious}}][]
	The first part of the proof is similar to that in \cite[Section 7]{zagier} and hence we will be brief.
	For Re$(z)>1$ and $x>0$, define $A_z(x,s)$ to be
	\begin{align}\label{A def}
		A_z(x,s):= \frac{1}{\Gamma(z)} \int_{0}^{\infty} \frac{t^{z-1+s}}{(e^{t}-1)(e^{xt}-1)} \, dt.
	\end{align}
	This integral converges for  $\textup{Re}(s)>2-\textup{Re}(z)$ and is analytic in this region. Note that for Re$(z)>1$, we have \cite[p.~251]{Apostol}
	\begin{align}
		\int_{0}^{\infty} \frac{t^{z-1}}{e^{xt}-1}  \, dt = \frac{\Gamma(z)\zeta(z)}{x^z}. \label{gammazetaintegral with x} 
	\end{align}
	Then, use \eqref{gammazetaintegral with x} with $z$ replaced by $z-1+s$ to get,
	\begin{align*}
		A_z(x,s) - \frac{\Gamma(z-1+s)\zeta(z-1+s)}{\Gamma(z) x^{z-1+s}} = \frac{1}{\Gamma(z)} \int_{0}^{\infty} t^{z-2+s} \left( \frac{t}{e^t-1}-1 \right) \frac{1}{e^{xt}-1} \, dt.
	\end{align*}
Hence $A_z(x,s)$ has a power series around $s=1$ given by
	\begin{align} 
		A_z(x,s) = \frac{\zeta(z)}{x^z} + \frac{1}{\Gamma(z)} \int_{0}^{\infty} t^{z-1} \left( \frac{t}{e^t-1}-1 \right) \frac{1}{e^{xt}-1} \, dt + o(s-1). \label{Az mid}
	\end{align}
	
\noindent
	Differentiating both sides of \eqref{almost integral3} with respect to $x$ and simplifying, we find that
	\begin{align}
		\frac{d}{dx}\Phi(z,x) 
		&= -\frac{1}{\Gamma(z)} \int_{0}^{\infty} t^{z-1} \left( \frac{t}{e^{t}-1}-1 \right) \frac{1}{e^{xt}-1} \, dt - \frac{z \zeta(z+1)}{x^{z+1}}, \label{phi derivative}
	\end{align}
	where we have used \eqref{gammazetaintegral with x} with $z$ replaced by $z+1$.
	Substitute \eqref{phi derivative} in \eqref{Az mid} and let $s\to1$ to obtain
	\begin{align}\label{Az main}
		A_z(x,1) = \frac{\zeta(z)}{x^z} - \frac{z \zeta(z+1)}{x^{z+1}} - \frac{d}{dx}\Phi(z,x).
	\end{align}	
Using 
	\begin{align*}
		\frac{1}{(e^{t}-1)(e^{xt}-1)}
		= \frac{1}{(e^{(x-1)t}-1)(e^{t}-1)} - \frac{1}{(e^{(x-1)t}-1)(e^{xt}-1)} -\frac{1}{(e^{xt}-1)},
		\end{align*}
 in \eqref{A def} and letting $s\to1$, we find that for $x>1$,		
	\begin{align*}
		A_z(x,1) &= A_z(x-1,1) - \frac{1}{x^{z+1}} A_z\left(\frac{x-1}{x},1\right) - \frac{\Gamma(z+1)\zeta(z+1)}{\Gamma(z) x^{z+1}}.
	\end{align*}
	Use \eqref{Az main} to transform the above equation into
	\begin{align}\label{three term in z}
		x^{\frac{z+1}{2}}	\left( G(z,x-1)-G(z,x) \right) =   x^{\frac{-z-1}{2}} G(z, \tfrac{x-1}{x}),
	\end{align}
	where $G(z,x):= \frac{\zeta(z)}{x^z}  - \frac{d}{dx}(\Phi(z,x)).$ 
	Differentiating \eqref{three term in z} $k$ times with respect to $z$, we have
	\begin{align}\label{diff3}
		&x^{\frac{z+1}{2}} \sum_{j=0}^{k} \binom{k}{j} \left(\frac{\log(x)}{2}\right)^{k-j}	\left( \frac{d^j}{dz^j} G(z,x-1)- \frac{d^j}{dz^j}G(z,x) \right)=  x^{\frac{-z-1}{2}} \sum_{j=0}^{k} \binom{k}{j} \left(\frac{-\log(x)}{2}\right)^{k-j} \frac{d^j}{dz^j} G\left(z, \frac{x-1}{x}\right) .  
	\end{align}
	We can see that for $0 \leq j \leq k$,
	\begin{align*}
		\frac{d^j}{dz^j}G(z,x) 
		= \sum_{\ell=0}^{j} \binom{j}{\ell} x^{-z} (-1)^\ell \log^\ell(x) \frac{d^{j-\ell}}{dz^{j-\ell}}\zeta(z)- \frac{d^j}{dz^j}\left(\frac{d}{dx}\Phi(z,x)\right).
	\end{align*}
	Substituting the above representation in \eqref{diff3} and rearranging, we see that for $x>1$,
	\begin{align}\label{afterdiff}
		&x^{\frac{z+1}{2}} \sum_{j=0}^{k} \binom{k}{j} \left(\frac{\log(x)}{2}\right)^{k-j} \sum_{\ell=0}^{j} \binom{j}{\ell} (-1)^\ell\left((x-1)^{-z}  \log^\ell(x-1)-x^{-z}  \log^\ell(x)\right)  \frac{d^{j-\ell}}{dz^{j-\ell}}\zeta(z)\nonumber \\
		&-x^{\frac{-z-1}{2}} \sum_{j=0}^{k} \binom{k}{j} \left(\frac{-\log(x)}{2}\right)^{k-j} \sum_{\ell=0}^{j} \binom{j}{\ell} (-1)^\ell  \left(\frac{x-1}{x}\right)^{-z}  \log^\ell\left(\frac{x-1}{x}\right)\frac{d^{j-\ell}}{dz^{j-\ell}}\zeta(z) \nonumber\\
		&+ x^{\frac{z+1}{2}} \sum_{j=0}^{k} \binom{k}{j} \left(\frac{\log(x)}{2}\right)^{k-j} \left(\frac{d^j}{dz^j}\left(\frac{d}{dx}\Phi(z,x)\right)- \frac{d^j}{dz^j}\left(\frac{d}{dx}\Phi(z,x-1)\right)\right) \nonumber\\
		& + x^{\frac{-z-1}{2}} \sum_{j=0}^{k} \binom{k}{j} \left(\frac{-\log(x)}{2}\right)^{k-j} \frac{d^j}{dz^j}\left(\frac{d}{dx}\Phi\left(z,\frac{x-1}{x}\right)\right)=0.
	\end{align}
	Note that Lemma \ref{power series coefficient of hhf} implies that for $j \geq 0$,
	\begin{align}\label{phiderivative}
		\lim_{z \to 1}\frac{d^j}{dz^j}\left(\frac{d}{dx}\Phi(z,x)\right)
		= (-1)^{j+1} \sum_{n=1}^{\infty} \left(  \psi_j'(nx)- \frac{\log^j(nx)}{nx}\right) = \frac{(-1)^{j+1}}{j+1} \Phi'_{j+1}(x),
	\end{align}
Hence, from \eqref{afterdiff} and \eqref{phiderivative}, for $x>1$,
	\begin{align}\label{withLHS}
		&		x \sum_{j=0}^{k} \binom{k}{j}  \frac{(-1)^{j}}{j+1}  \left(\frac{\log(x)}{2}\right)^{k-j} \hspace{-3mm}\left(  \Phi'_{j+1}(x)-  \Phi'_{j+1}(x-1)\right) + \frac{1}{x} \sum_{j=0}^{k} \binom{k}{j}  \frac{(-1)^{j}}{j+1}  \left(\frac{-\log(x)}{2}\right)^{k-j} \hspace{-2mm}\Phi'_{j+1}\left(\frac{x-1}{x}\right)  \nonumber\\
		& = \lim_{z \to 1} \Bigg[x^{\frac{z+1}{2}} \sum_{j=0}^{k} \binom{k}{j} \left(\frac{\log(x)}{2}\right)^{k-j} \sum_{\ell=0}^{j} \binom{j}{\ell} (-1)^\ell \left((x-1)^{-z}  \log^\ell(x-1)-x^{-z}  \log^\ell(x)\right) \frac{d^{j-\ell}}{dz^{j-\ell}}\zeta(z)\nonumber \\
		&\qquad\qquad-x^{\frac{-z-1}{2}} \sum_{j=0}^{k} \binom{k}{j} \left(\frac{-\log(x)}{2}\right)^{k-j} \sum_{\ell=0}^{j} \binom{j}{\ell} (-1)^\ell  \left(\frac{x-1}{x}\right)^{-z}  \log^\ell\left(\frac{x-1}{x}\right)\frac{d^{j-\ell}}{dz^{j-\ell}}\zeta(z) \Bigg].
	\end{align}
	Our next task is evaluate the above limit which we denote by $L$. Observe that 
	\begin{align*}
	L& = x \lim_{z \to 1} \sum_{j=0}^{k} \binom{k}{j} \left(\frac{\log(x)}{2}\right)^{k-j} \sum_{\ell=0}^{j} \binom{j}{\ell} (-1)^{\ell} \frac{d^{j-\ell}}{dz^{j-\ell}}\zeta(z) \nonumber \\
		&\hspace{2cm}\times \left\{(x-1)^{-z}  \log^{\ell}(x-1)-x^{-z}  \log^{\ell}(x)-(-1)^{k-j}\frac{(x-1)^z}{x} \log^{\ell}\left(\frac{x-1}{x}\right)\right\}.
	\end{align*}
Add and subtract $\frac{(-1)^{j-\ell+1}(j-\ell)!}{(z-1)^{j-\ell+1}}$ from $\frac{d^{j-\ell}}{dz^{j-\ell}}\zeta(z)$, and use the fact that
	$
		\lim_{z \to 1}\left\{ \zeta^{(m)}(z) - \frac{(-1)^m m!}{(z-1)^{m+1}}\right\}= (-1)^m\gamma_m$
	for $m\geq0$, to arrive at 
	\begin{align}\label{beforelimit}
		L&	= x  \sum_{j=0}^{k} \binom{k}{j} \left(\frac{\log(x)}{2}\right)^{k-j} \sum_{\ell=0}^{j} \binom{j}{\ell} (-1)^{j+1} \gamma_{j-\ell} \left\{\frac{\log^\ell(x)}{x}-\frac{\log^\ell(x-1)}{x-1}+\frac{(-1)^{k-j}\log^\ell(\tfrac{x-1}{x})}{x (x-1)}\right\} \nonumber \\
		&\quad+	 x \lim_{z \to 1}\frac{ S(z,x)}{(z-1)^{k+1}},
	\end{align}
	where 
	\begin{align}\label{S(z,x)}
		S(z,x):=&	\sum_{j=0}^{k} \binom{k}{j} \left(\frac{\log(x)}{2}\right)^{k-j} \sum_{\ell=0}^{j} \binom{j}{\ell} (-1)^{j} \ell! (z-1)^{k-\ell} \nonumber \\
		&\hspace{1cm} \times \left\{(x-1)^{-z}  \log^{j-\ell}(x-1)-x^{-z}  \log^{j-\ell}(x)-(-1)^{k-j}\frac{(x-1)^{-z}}{x} \log^{j-\ell}\left(\frac{x-1}{x}\right)\right\}.
	\end{align}
	Now we show that $S(z,x)$ has a zero of order $k$ at $z=1$.
	Let $0 \leq r \leq k$. Interchange the order of summation in \eqref{S(z,x)} and then differentiate $r$ times with respect to $z$ using the Leibnitz rule to get
	\begin{align}\label{r times}
		&\frac{d^r}{dz^r}(S(z,x))= k!\sum_{i=0}^{r} \binom{r}{i}	\sum_{\ell=0}^{k}  \sum_{j=\ell}^{k}  \left(\frac{\log(x)}{2}\right)^{k-j} \hspace{-4mm}\frac{(-1)^{j}}{(k-j)!(j-\ell)!} \left(\frac{d^{r-i}}{dz^{r-i}}(z-1)^{k-\ell}\right) \bigg\{ \frac{(-1)^{i+1}}  {x^{z}} \log^{j-\ell+i}(x) \nonumber \\
		&\hspace{3cm} +\frac  {(-1)^i}{(x-1)^{z} } \log^{j-\ell+i}(x-1) -\frac{(-1)^{k-j+i}}{x(x-1)^z} \log^{j-\ell}\left(\frac{x-1}{x}\right)\log^{i}(x-1)\bigg\}.
	\end{align}
	Note that $\displaystyle \lim_{z \to 1}\frac{d^{r-i}}{dz^{r-i}}(z-1)^{k-\ell}= (r-i)!$ if $k-\ell =r-i$ and is $0$ otherwise. Therefore, we can write
	\begin{align*}
		\lim_{z \to 1}	\frac{d^r}{dz^r}(S(z,x))&= k!\sum_{i=0}^{r} \binom{r}{i} \sum_{j=k-r+i}^{k}  \left(\frac{\log(x)}{2}\right)^{k-j} \frac{(-1)^{j} (r-i)!}{(k-j)!(j-k+r-i)!} \bigg\{\frac{(-1)^{i+1}}{x} \log^{j-k+r}(x)   \nonumber \\
		&\quad+ \frac{(-1)^i}{x-1}\log^{j-k+r}(x-1)-\frac{(-1)^{k-j+i}}{x (x-1)}\log^{j-k+r-i}\left(\frac{x-1}{x}\right)\log^{i}(x-1)\bigg\}.
	\end{align*}
	Substitute $p=j-k+r-i$, and then replace $i$ by $r-i$ to get
	\begin{align*}
		\lim_{z \to 1}	\frac{d^r}{dz^r}(S(z,x))
		&= S_1(x)+S_2(x)+S_3(x),
	\end{align*}
	where
	\begin{align*}
		&S_1(x):=k!\sum_{i=0}^{r} \binom{r}{i} \sum_{p=0}^{i} \binom{i}{p}\frac{(-1)^{p+k+r}}{2^{i-p}(x-1)} \log^{i-p}(x) \log^{p+r-i}(x-1),\\
		&S_2(x):=k!\sum_{i=0}^{r} \binom{r}{i} \sum_{p=0}^{i} \binom{i}{p}\frac{(-1)^{p+k+r+1}}{2^{i-p}x}  \log^{r}(x) ,\\
		&S_3(x):=k!\sum_{i=0}^{r} \binom{r}{i} \sum_{p=0}^{i} \binom{i}{p}\frac{(-1)^{r+k+i+1}}{2^{i-p}x(x-1)} \log^{i-p}(x) \log^{p}\left(\frac{x-1}{x}\right)\log^{r-i}(x-1).
	\end{align*} 
Using the binomial theorem repeated, we see that
	\begin{align*}
		S_1(x)
		= \frac{(-1)^{k+r}k!}{x-1} \left( \frac{\log(x)}{2}\right)^r,\hspace{1.5mm}
		S_2(x)=\frac{(-1)^{k+r+1} k! }{x} \left( \frac{\log(x)}{2}\right)^r,\hspace{1.5mm}
		S_3(x)=\frac{(-1)^{k+r+1}   k! }{x(x-1)}    \left( \frac{\log(x)}{2}\right)^{r} .
	\end{align*}
	Therefore,
	\begin{align*}
		S_1(x)+S_2(x)+S_3(x)= (-1)^{k+r}   k!   \left( \frac{\log(x)}{2}\right)^{r} \left(\frac{1}{x-1}-\frac{1}{x}-\frac{1}{x(x-1)}\right)=0.
	\end{align*}
	Hence, for $0 \leq r \leq k$, $\lim_{z \to 1} \frac{d^r}{dz^r}(S(z,x))=0$. By L'Hospital's rule,
	\begin{align}\label{limitequation}
		\lim_{z \to 1}\frac{ S(z,x)}{(z-1)^{k+1}} = \frac{	\lim_{z \to 1} \frac{d^{k+1}}{dz^{k+1}}(S(z,x))}{(k+1)!}.
	\end{align} 
	We now use \eqref{r times} with $r=k+1$ and note that 
	$\displaystyle \lim_{z \to 1}\frac{d^{{k+1}-i}}{dz^{{k+1}-i}}(z-1)^{k-\ell}= (r-i)!$ if $\ell =i-1, i>0$  and is $0$ otherwise. Therefore,
	\begin{align*}
		\lim_{z \to 1}	\frac{d^{k+1}}{dz^{k+1}}S(z,x)&= k!\sum_{i=1}^{k+1} \binom{k+1}{i} \sum_{j=i-1}^{k}  \left(\frac{\log(x)}{2}\right)^{k-j} \frac{(-1)^{j} (k+1-i)!}{(k-j)!(j+1-i)!}\bigg\{  \frac{(-1)^{i+1}}{x}\log^{j+1}(x)    \nonumber \\
		&\quad+\frac{(-1)^i}{x-1}\log^{j+1}(x-1)-\frac{(-1)^{k-j+i}}{x (x-1)}\log^{j-i+1}\left(\frac{x-1}{x}\right)\log^{i}(x-1)\bigg\}.
	\end{align*}
	Substitute $p=j+1-i$ and then add and subtract $i=0$ term to get
	\begin{align*}
		&\lim_{z \to 1}	\frac{d^{k+1}}{dz^{k+1}}S(z,x)\nonumber\\
		&= k!\sum_{i=0}^{k+1} \binom{k+1}{i} \sum_{p=0}^{k+1-i}  \left(\frac{\log(x)}{2}\right)^{k+1-i-p}  (-1)^{p-1+i}\binom{k+1-i}{p}  \nonumber \\
		 &\quad\times\bigg\{ \frac{(-1)^{i+1}}{x}\log^{p+i}(x)+\frac{(-1)^i}{x-1}\log^{p+i}(x-1)+\frac{(-1)^{k-p}}{x (x-1)}\log^{p}\left(\frac{x-1}{x}\right)\log^{i}(x-1)\bigg\}\nonumber\\
		&\quad+ k!  \sum_{p=0}^{k+1} \binom{k+1}{p} \left(\frac{\log(x)}{2}\right)^{k+1-p} (-1)^{p}    \left\{\frac{\log^{p}(x-1)}{x-1}-\frac{\log^{p}(x)}{x}+\frac{(-1)^{k-p}\log^{p}(\tfrac{x-1}{x})}{x (x-1)}\right\}.
	\end{align*}
	The double sum on the right hand side can be shown to be zero using the binomial theorem as done before. Again using the binomial theorem in the second sum, we get
	\begin{align}\label{afterlimit}
		\lim_{z \to 1}	\frac{d^{k+1}}{dz^{k+1}}S(z,x)
		=\frac{k! (-1)^k}{x} \left(\frac{\log(x)}{2}\right)^{k+1}+\frac{k!}{x} \left(\frac{\log(x)}{2}-\log(x-1)\right)^{k+1}.
	\end{align}
	Substituting \eqref{afterlimit}, in \eqref{limitequation}, and  the resulting equation, in turn, in \eqref{beforelimit}, we see that 
	\begin{align*}
	L&	= x  \sum_{j=0}^{k} \binom{k}{j} \left(\frac{\log(x)}{2}\right)^{k-j} \sum_{\ell=0}^{j} \binom{j}{\ell} (-1)^{j} \gamma_{j-\ell} \left\{\frac{\log^\ell(x)}{x}-\frac{\log^\ell(x-1)}{x-1}+\frac{(-1)^{k-j}\log^\ell(\tfrac{x-1}{x})}{x (x-1)}\right\} \nonumber \\
	&\quad+\frac{1}{k+1}\left\{(-1)^k \left(\frac{\log(x)}{2}\right)^{k+1}+ \left(\frac{\log(x)}{2}-\log(x-1)\right)^{k+1}\right\}.
	\end{align*}
	Substitute this limit evaluation in \eqref{withLHS} and simplify to see that the proof is complete.
\end{proof}

		\subsection{The function $J(z, x)$ and its connection with $\Phi(z, x)$}
		For Re$(z)>0$ and $x>0$, define $J(z, x)$ by
		\begin{align}
			J(z,x) := -\frac{1}{\Gamma(z)} \int_{0}^{1} \frac{(\log(\tfrac{1}{y}))^{z-1}\log(1+y^x)}{1+y} \, dy. \label{J(z,x) def}
		\end{align}
		Clearly, $J(1, x)=-J(x)$, where $J(x)$ is defined in \eqref{JT-def}. 
		In Section \ref{integral-J}, we work with a more general integral, namely, $J_{r, \chi_1, \chi_2}(z, x)$ of which $J(z, x)$ is a special case. Since we prove the analyticity of this more general integral there, we refrain here from proving it separately for $J(z, x)$ at this juncture. 
		
		We now derive an identity which connects $J(z,x)$ with the Herglotz-Hurwitz function $\Phi(z, x)$.
		\begin{theorem}\label{J(z,x) F_z(x) relation}
			For $x>0$ and \textup{Re}$(z)>1$, we have
			\begin{align}\label{JF-rel}
				J(z,x) =\Phi(z,2x) -(2^{1-z}+1)  \Phi(z,x) + 2^{1-z}\Phi(z,\tfrac{x}{2})  + \frac{(2^{-z}-1)}{x^{z}} \zeta(z+1).
			\end{align}
		\end{theorem}
		\begin{proof}
			Employing the change of variable $y=e^{-t}$, we see that
			\begin{align*}
				J(z,x) &= -\frac{1}{\Gamma(z)} \int_{0}^{\infty}  \frac{t^{z-1} \log(1+e^{-xt})}{e^t+1} \, dt\\
				&= -\frac{1}{\Gamma(z)} \int_{0}^{\infty}  t^{z-1} \log(1+e^{-xt}) \left( \frac{1}{e^t-1} - \frac{1}{t} \right) \, dt  +\frac{1}{\Gamma(z)} \int_{0}^{\infty}  t^{z-1} \log(1+e^{-xt}) \left( \frac{2}{e^{2t}-1} - \frac{1}{t}  \right) \, dt.
			\end{align*}
			Replace $t$ by $t/2$ in the second integral and then combine the two integrals to get
			\begin{align}
				J(z,x) &=-\frac{1}{\Gamma(z)} \int_{0}^{\infty}  t^{z-1}  \left( \frac{1}{e^t-1} - \frac{1}{t} \right) \left(  \log(1+e^{-xt}) - 2^{1-z} \log(1+e^{\frac{-xt}{2}}) \right) \, dt \notag\\
				&= -I_3(z,x) +I_4(z,x) \label{halfway},
			\end{align}
			where
			\begin{align*}
				I_3(z,x)&:= \frac{1}{\Gamma(z)}\int_{0}^{\infty}  t^{z-1}  \left( \frac{1}{1-e^{-t}} - \frac{1}{t} \right) \left(  \log(1+e^{-xt}) - 2^{1-z} \log(1+e^{\frac{-xt}{2}}) \right) \, dt , \\
				I_4(z,x)&:= \frac{1}{\Gamma(z)} \int_{0}^{\infty}  t^{z-1}   \left(  \log(1+e^{-xt}) - 2^{1-z} \log(1+e^{\frac{-xt}{2}}) \right) \, dt .
			\end{align*}
			Use $\log(1+e^{w})=\log(1-e^{2w}) -  \log(1-e^{w})$ with $w=-xt/2$, add and subtract $2^{z-1} \log(1-e^{-xt})$ and again use the aforementioned identity with $w=-xt$ to get
			\begin{align*}
				I_3(z,x) &= \frac{1}{\Gamma(z)}\int_{0}^{\infty}  t^{z-1}  \left( \frac{1}{1-e^{-t}} - \frac{1}{t} \right) \left(  \log(1-e^{-2xt}) -(2^{1-z}+1)\log(1-e^{-xt}) + 2^{1-z} \log(1-e^{\frac{-xt}{2}}) \right) \, dt .
			\end{align*}
			Invoking \eqref{almost integral3} three times, it can be seen that
			\begin{align}
				I_3(z,x) 
				= - \left(  \Phi(z,2x) - (2^{1-z}+1)  \Phi(z,x) + 2^{1-z}\Phi(z,\tfrac{x}{2})  \right) \label{I3 final}.
			\end{align}
			Simplifying the integrand of $I_4(z, x)$ analogously to that done for $I_3(z, x)$, we obtain
			\begin{align}
				I_4(z,x) &=  \frac{1}{\Gamma(z)}\int_{0}^{\infty}  t^{z-1}  \left(  \log(1-e^{-2xt}) -(2^{1-z}+1)\log(1-e^{-xt}) + 2^{1-z} \log(1-e^{\frac{-xt}{2}}) \right) \, dt\nonumber\\
				&=(-(2x)^{-z}  \zeta(z+1)) -(2^{1-z}+1)( -x^{-z}\zeta(z+1)) -2^{1-z} (\tfrac{x}{2})^{-z}  \zeta(z+1) \notag\\  
				&= x^{-z}(2^{-z}-1)  \zeta(z+1) \label{I4 final},
			\end{align}
	  	 	where, in the second step above, we used the well-known result $\int_{0}^{\infty} u^{z-1} \log(1-e^{-u}) \, du = -\Gamma(z)\zeta(z+1)$.
	  	 	Substitute \eqref{I3 final} and \eqref{I4 final} in \eqref{halfway}, we arrive at \eqref{JF-rel}.
	  \end{proof}
 The above theorem gives the result of Radchenko and Zagier \cite[Equation (1.1)]{raza} as a corollary.
	  \begin{corollary}
	  	The relation \eqref{zaginteval} holds.
	  \end{corollary}
	  \begin{proof}
	  	Let $z\to1$ in Theorem \ref{J(z,x) F_z(x) relation} and employ Lemma \ref{power series coefficient of hhf} with $j=0$ to simplify the resulting equation along with the fact that $J(1, x)=-J(x)$.
	  \end{proof}
	  	 Another application of Theorem \ref{J(z,x) F_z(x) relation} lies in obtaining a nice functional equation for $J(z, x)$ which also  involves the alternating Mordell-Tornheim zeta function with the extra parameter $x$. Before stating this result, we first derive an identity which will be used to get the aforementioned functional equation for $J(z, x)$, and which is interesting in its own right since it connects $\Theta(z, x)$ with the  aforementioned alternating double series.
	  	 \begin{theorem}\label{J-lemma}
	  	 Let $\Theta(z, x)$ be defined in \eqref{Thetadef}. For \textup{Re}$(z)>1$ and $x>0$,
	  	\begin{align}\label{Theta identity}
	  		&  \Theta(z, 2x)  - (2^{1-z}+1) \Theta(z, x) + 2^{1-z} \Theta\left(z, \frac{x}{2}\right)=-\sum_{n=1}^{\infty} \sum_{m=1}^{\infty} \frac{(-1)^{n+m}}{nm(n+mx)^{z-1}}.
	  	\end{align}
	  	\end{theorem}
	  	\begin{proof}
	  	Starting with the integral representation
	  	\begin{align*}
	  		\frac{\Gamma(z-1)}{(n+mx)^{z-1}}= \int_{0}^{\infty} y^{z-2} e^{-(n+mx)y} \, dy,
	  	\end{align*}
	  	it is not difficult to see that
	  	\begin{align*}
	  		\Theta(z,x) =\frac{1}{\Gamma(z-1)} \int_{0}^{\infty} y^{z-2}\log(1-e^{-y}) \log(1-e^{-xy})  \, dy. 
	  	\end{align*}	   
	  	Using this representation for each of $\Theta$-functions on the left-hand side of \eqref{Theta identity}, we have
	  	\begin{align}
	  		&  \Theta(z, 2x)  - (2^{1-z}+1) \Theta(z, x) + 2^{1-z} \Theta\left(z, \frac{x}{2}\right)\nonumber\\
	  		&= \frac{1}{\Gamma(z-1)} \left(   \int_{0}^{\infty} \log(1-e^{-y}) \log(1-e^{-2xy}) y^{z-2}  \, dy \right.\notag \\
	  		& \hspace{2cm} \left. - (2^{1-z}+1)\int_{0}^{\infty} \log(1-e^{-y}) \log(1-e^{-xy}) y^{z-2}  \, dy +2^{1-z}\int_{0}^{\infty} \log(1-e^{-y}) \log(1-e^{-\tfrac{xy}{2}}) y^{z-2}  \, dy      \right) \notag \\
	  		&= \frac{1}{\Gamma(z-1)} \left(   \int_{0}^{\infty} \log(1-e^{-y}) \log(1+e^{-xy}) y^{z-2}  \, dy -2^{1-z} \int_{0}^{\infty} \log(1-e^{-y}) \log(1+e^{-\tfrac{xy}{2}}) y^{z-2}  \, dy \right). \label{integral almost}
	  	\end{align}
	  	Now make a change of variable $y=2t$ in the second integral on the right-hand side of \eqref{integral almost} to get
	  	\begin{align}\label{done almost}
	  		\Theta(z, 2x)  - (2^{1-z}+1) \Theta(z, x) + 2^{1-z} \Theta\left(z, \frac{x}{2}\right)= -\frac{1}{\Gamma(z-1)}\int_{0}^{\infty} \log(1+e^{-y}) \log(1+e^{-xy}) y^{z-2}  \, dy.
	  	\end{align}
	  	We now write $\log(1+e^{-y})=\sum_{n=1}^{\infty}(-1)^{n-1} e^{-ny}/n $ and  $\log(1+e^{-xy})=\sum_{m=1}^{\infty}(-1)^{m-1} e^{-mxy}/m $ and interchange the order of double sum and integration. This is justified by using the standard theorem on interchange of the order of summation and integration \cite[p.~30, Theorem 2.1]{temme} beginning with the sum over $n$, then using Watson's lemma \cite[p.~32, Theorem 2.4]{temme}, followed by another application of  \cite[p.~30, Theorem 2.1]{temme}, this time for the sum over $m$. Thus, the right-hand side of \eqref{done almost} becomes
	  	\begin{align}\label{to be referred}
	  	-\frac{1}{\Gamma(z-1)}\int_{0}^{\infty} \log(1+e^{-y}) \log(1+e^{-xy}) y^{z-2}  \, dy&=-\frac{1}{\Gamma(z-1)}\sum_{n=1}^{\infty} \sum_{m=1}^{\infty} \frac{(-1)^{n+m}}{nm} \int_{0}^{\infty}  e^{-(n+mx)y} y^{z-2}  \, dy\nonumber\\
	  	&=-\sum_{n=1}^{\infty} \sum_{m=1}^{\infty} \frac{(-1)^{n+m}}{nm(n+mx)^{z-1}}.
	  	\end{align}
	  Thus \eqref{Theta identity} follows from \eqref{done almost} and \eqref{to be referred}.
	  	\end{proof}
	  	\begin{remark}
	  	Letting $x=1$ in Theorems \ref{theta-3t} and \eqref{J-lemma} give two different representations for $(2^{1-z}+1)\Theta(z, 1)-\Theta(z, 2)-2^{1-z}\Theta(z, 1/2)$ which must equal each other. This leads us to the identity
	  	\begin{align*}
	  \int_{0}^{\infty}t^{z-2}\textup{Li}_2\left(\frac{e^{-t}}{e^{-t}+1}\right)\, dt=-\Gamma(z-1)\left\{(2^{-z}-1)\zeta(z+1)+\frac{1}{2}\sum_{n=1}^{\infty}\sum_{m=1}^{\infty}\frac{(-1)^{n+m}}{nm(n+m)^{z-1}}\right\},
	  	\end{align*}
	  	which can, of course, be obtained by letting $x=1$ in \eqref{to be referred} and using the functional equation \eqref{to be also referred} with $w=-e^{-t}$ while noting that $\int_{0}^{\infty}t^{z-2}\textup{Li}_2(-e^{-t})\, dt=(2^{-z}-1)\Gamma(z-1)\zeta(z+1)$.
	  	\end{remark}
	  	We are now ready to derive a two-term functional equation for $J(z, x)$. 
	  \begin{theorem}\label{new J theorem}
	  	For $x>0$ and \textup{Re}$(z)>1$,
	  	\begin{align}\label{J-functional}
	  		J(z,x)+x^{1-z}J\left(z,\frac{1}{x}\right) = -\sum_{n=1}^{\infty} \sum_{m=1}^{\infty} \frac{(-1)^{m+n}}{mn (n+mx)^{z-1}}.
	  	\end{align}
	  \end{theorem}
	  \begin{proof}
	  Representing the two instances of the $J$-function in \eqref{J-functional} by means of the right-hand side of \eqref{JF-rel} and simplifying, we see that
	  \begin{align*}
	  	&J(z,x)+x^{1-z}J(z,\tfrac{1}{x}) \notag \\
	  	&=  \left( \Phi(z,2x) + (2x)^{1-z}\Phi(z,\tfrac{1}{2x}) \right) - (2^{1-z}+1) \left( \Phi(z,x) + x^{1-z}\Phi(z,\tfrac{1}{x})  \right) + 2^{1-z} \left( \Phi(z,\tfrac{x}{2})  + \left(\tfrac{x}{2}\right)^{1-z} \Phi(z,\tfrac{2}{x})  \right) \notag\\ 
	  	&\quad+ (2^{-z}-1) (x+x^{-z})\zeta(z+1).
	  	\end{align*}
	  	The result is now obtained by employing Theorem \ref{decompostion theorem} three times followed by an application of Theorem \ref{J-lemma}.
	\end{proof}

	\section{A generalization of the Herglotz-Zagier-Novikov function}\label{hzn}
	
	The Herglotz-Zagier-Novikov function $\mathscr{F}(x;u,v)$ was defined in \eqref{script F def}
	for $u\in\mathbb{C}\backslash(1,\infty), u\neq0, v\in\mathbb{C}\backslash[1,\infty)$ and $x>0$. For the same values of $u$, $v$ and $x$, we define a more general integral $	\mathscr{F}_z(x;u,v) $ by
	\begin{align}
		\mathscr{F}_z(x;u,v) := \frac{1}{\Gamma(z)} \int_{0}^{1} \frac{\left( \log(\tfrac{1}{t})\right)^{z-1} \log(1-ut^x) }{v^{-1}-t} \, dt. \label{script F def z}
	\end{align}
	We first show that this integral is analytic in Re$(z)>0$. This is done by first transforming the integral by means of the change of variable $y=-\log(t)$ so that
	\begin{align*}
			\mathscr{F}_z(x;u,v) =\frac{1}{\Gamma(z)} \int_{0}^{\infty}e^{-y}y^{z-1}\frac{\log(1-ue^{-xy})}{v^{-1}-e^{-y}}\, dy.
	\end{align*}
	Let $f(z, y)$ denote the integrand of the above integral. Clearly, $f$ is continuous in both $z$ and $y$ and for a fixed $y\in(0,1)$, $f$ is analytic in Re$(z)>0$. We need only show that the integral converges uniformly at both the limits. To that end, let $S=\{z: a\leq\textup{Re}(z)\leq A\}$, where $0<a<A<\infty$. Let $0<y\leq 1$. Then $e^{-y}\leq1$ and $|e^{-y}y^{z-1}|\leq y^{a-1}$. If $u\neq 1$, $\frac{\log(1-ue^{-xy})}{v^{-1}-e^{-y}}=O_{u, v}(1)$ as $y\to0$, otherwise it is $O_v(y^{-\epsilon })$ for any $\epsilon>0$ as can be seen from \eqref{bounds}. Thus, in both the cases, $f(z,y)=O_{u, v}\left(y^{a-1-\epsilon}\right)$ as $y\to0$. 
	
	 Now let $y\geq1$. Then $|y^{z-1}|\leq y^{A-1}$. Since $\frac{\log(1-ue^{-xy})}{v^{-1}-e^{-y}}=O(1)$ as $y\to\infty$. Hence we see that the integral converges uniformly at both the limits. From \cite[p.~30-31]{temme}, we then conclude that $\mathscr{F}_z(x;u,v)$ is analytic in Re$(z)>0$.
	 
	 \begin{remark}
	 By the preceding logic, one can also see that $\mathscr{F}_z(x;u,1)$ is defined and analytic in \textup{Re}$(z)>1$. Indeed, this follows from the fact that $\frac{1}{1-e^{-y}}=O\left(\frac{1}{y}\right)$ as $y\to0^{+}$.
	 \end{remark}
	 
	 \begin{lemma} Let \textup{Re}$(z)>1$ and $x>0$. For  $u,v \in \mathbb{C}$ such that $|u|\leq1$ and $|v|\leq1$, $\mathscr{F}_z(x;u,v)$ can be expressed in the form of a doubly infinite series given by
	 	\begin{align}
	 		\mathscr{F}_z(x;u,v) = - \sum_{n=1}^{\infty} \sum_{m=1}^{\infty} \frac{u^m v^n}{m(n+mx)^z}. \label{script F series}
	 	\end{align}
	 \end{lemma}
	 \begin{proof}
	 Using the series representations of $\log(1-ut^x)$ and $1/(1-vt)$ and interchanging the order of summation and integration (which can be justified in a similar way as that done in the proof of Theorem \ref{J-lemma}), we see that
	 \begin{align*}
\mathscr{F}_z(x;u,v)&= -\frac{1}{\Gamma(z)} \sum_{m=1}^{\infty} \frac{u^m}{m} \sum_{n=0}^{\infty} v^{n+1} \int_{0}^{1} \left( \log(\tfrac{1}{t})\right)^{z-1}  t^{n+mx}   \, dt \\
&=- \frac{1}{\Gamma(z)} \sum_{m=1}^{\infty} \frac{u^m}{m} \sum_{n=0}^{\infty} v^{n+1}  \frac{\Gamma(z)}{(1+n+mx)^z} \\
&= -\sum_{m=1}^{\infty} \sum_{n=1}^{\infty} \frac{u^m v^{n}}{m (n+mx)^z},
\end{align*}
which leads to the right-hand side of \eqref{script F series} upon interchanging the order of summation.
	 	\end{proof}
	 
	 \begin{remark}\label{unnamed}
	 For $|u|<1, |v|<1$, the restriction on $z$ in the above lemma can be relaxed to \textup{Re}$(z)>0$.
	 \end{remark}
	 We now show that $\mathscr{F}_{z}(x; 1, 1)$ is connected to the Herglotz-Hurwitz function $\Phi(z, x)$.
	 \begin{lemma}\label{script F Phi relation}
		For \textup{Re}$(z)>1$ and $x>0$, we have
		\begin{align*}
			\mathscr{F}_z(x;1,1)  = - \Phi(z,x) + x^{-z} \zeta(z+1)- \frac{x^{1-z}}{z-1} \zeta(z).
		\end{align*}
	\end{lemma}
	 \begin{proof}
	 	Using \eqref{hhf-def}, we see that
	 	\begin{align*}
	 		\Phi(z,x)
	 		&= 	\sum_{m=1}^{\infty}  \frac{\zeta(z,mx)}{m}  -\frac{x^{1-z}}{z-1} \zeta(z) \\
	 		&= \sum_{m=1}^{\infty} \sum_{n=0}^{\infty} \frac{1}{m (n+mx)^z} -\frac{x^{1-z}}{z-1} \zeta(z) \\
	 		&= -\mathscr{F}_z(x;1,1) + x^{-z} \zeta(z+1) -\frac{x^{1-z}}{z-1} \zeta(z),
	 	\end{align*}
	 	where we used \eqref{script F series} in the last step. 
	 \end{proof}
	Our generalization of the Herglotz-Zagier-Novikov function satisfies an elegant functional equation which is the key component in the main result of the next section.
\begin{theorem}\label{small theorem} For  $u,v \in \mathbb{C}$ such that $|u|\leq1$ and $|v|\leq1$, $x>0$, and \textup{Re}$(z)>1$, we have
	\begin{align*}
		\mathscr{F}_z(x;u,v)  + x^{1-z}\mathscr{F}_z\left(\frac{1}{x};v,u\right) = -\sum_{n=1}^{\infty} \sum_{m=1}^{\infty} \frac{u^m v^{n}}{n m (n+mx)^{z-1}}. 
	\end{align*}
\end{theorem}
\begin{proof} We multiply and divide the summand by $n+mx$ to get,
	\begin{align*}
		\sum_{n=1}^{\infty} \sum_{m=1}^{\infty} \frac{u^m v^{n}}{n m (n+mx)^{z-1}} &= \sum_{n=1}^{\infty} \sum_{m=1}^{\infty} \frac{u^m v^{n}}{m (n+mx)^{z}} + x \sum_{n=1}^{\infty} \sum_{m=1}^{\infty} \frac{u^m v^{n}}{ n (n+mx)^{z}} \\
		&=\sum_{n=1}^{\infty} \sum_{m=1}^{\infty} \frac{u^m v^{n}}{m (n+mx)^{z}} + x^{1-z} \sum_{n=1}^{\infty} \sum_{m=1}^{\infty} \frac{u^m v^{n}}{ n (\tfrac{n}{x}+m)^{z}} \\
		&= \sum_{n=1}^{\infty} \sum_{m=1}^{\infty} \frac{u^m v^{n}}{m (n+mx)^{z}}  + x^{1-z} \sum_{m=1}^{\infty} \sum_{n=1}^{\infty} \frac{u^{n} v^m }{ m (n+\tfrac{m}{x})^{z}},
	\end{align*}
	where in the second series in the last step, we swapped the indices $m$ and $n$. Then interchange the order of summation in this second series and use \eqref{script F series} to complete the proof.
\end{proof}
\begin{remark}\label{decomposition as a special case}
If we let $u=v=1$ in the above theorem and use Lemma \ref{script F Phi relation}, we recover Theorem \ref{decompostion theorem}.
\end{remark}
\begin{corollary}
Let $\mathscr{F}(x; u, v)$ be defined in \eqref{script F def}. For  $u,v \in \mathbb{C}$ such that $|u|<1$ and $|v|<1$, $x>0$,
\begin{align}\label{1=z}
\mathscr{F}(x; u, v)+\mathscr{F}\left(\frac{1}{x}; v, u\right)=-\log(1-u)\log(1-v).
\end{align}
\end{corollary}
\begin{proof}
In view of Remark \ref{unnamed}, we let $z=1$ in Theorem \ref{small theorem} and simplify.
\end{proof}
	\section{A grand unification of Herglotz-type integrals $J(x)$ and $T(x)$ through roots of unity}\label{unification}
	
	There are several integrals in the literature which satisfy elegant functional equations, one of the most famous being the dilogarithm defined in \eqref{didef}. 
	Another interesting functional equation is due to Ramanujan who posed it as an interesting problem to the Journal of the Indian Mathematical Society \cite{ramanujan indian} and whose generalization we provide below is due to Berndt and Evans \cite{berndt-evans}. Let $f(\infty)$ denote $\lim_{x\to\infty}f(x)$, provided the limit exists, and $\infty$ if $f(x)\to\infty$ as $x\to\infty$. Let $g$ be a strictly increasing, differentiable function on $[0, \infty)$ with $g(0)=1$ and $g(\infty)=\infty$. For $n>0$ and $t\geq0$, let $v(t):=g^n(t)/g(1/t)$. If $\varphi(n):=\int_{0}^{1}\log(g(t))dv/v$ converges, then $\varphi(n)+\varphi(1/n)=2\varphi(1)$. The problem Ramanujan posed is the special case $g(t)=1+t$ of the result of Berndt and Evans.
	
	In this section, we obtain a grand generalization of the functional equations in \eqref{JT-FE}.
	
	\subsection{The function $J_{r, \chi_1, \chi_2}(z, x)$ and its functional equation}\label{integral-J}
	
	Recall the definition of $J_{r, \chi_1, \chi_2}(z, x)$ in \eqref{J r chi def}. We begin with the following lemma connecting $J_{r,\chi_1, \chi_2}(z,x)$ to $\mathscr{F}_z(x,u,v)$.

	\begin{lemma}\label{J r chi script F lemma}
		Let \textup{Re}$(z)>0, x>0$ and $r\geq2$. Let $\chi_1,\chi_2$ be any two Dirichlet characters modulo $r$. Then,
		\begin{align}\label{J as a double sum}
			J_{r,\chi_1, \chi_2}(z,x) = \sum_{j=1}^{r-1} \sum_{k=1}^{r-1} \chi_1(j) \chi_2 (k) \mathscr{F}_z(x ,\zeta_r^{j}, \zeta_r^{k}).
		\end{align}
	\end{lemma}
	\begin{proof}
		Using \eqref{script F def z}, we have
		\begin{align}\label{double sum}
			\sum_{j=1}^{r-1} \sum_{k=1}^{r-1} \chi_1(j) \chi_2 (k) \mathscr{F}_z(x ,\zeta_r^{j}, \zeta_r^{k}) &= \frac{1}{\Gamma(z)} \sum_{j=1}^{r-1} \chi_1(j) \sum_{k=1}^{r-1} \chi_2(k) \int_{0}^{1} \frac{\left(\log\left(\frac{1}{t}\right)\right)^{z-1}\left(\log \left( 1-\zeta_r^j t^x \right)\right)}{\zeta_r^{-k}-t} \, dt \nonumber\\
			&= \frac{1}{\Gamma(z)} \int_{0}^{1} \left(\log\left(\frac{1}{t}\right)\right)^{z-1}  \left(\sum_{k=1}^{r-1} \frac{\chi_2(k)}{\zeta_r^{-k}-t}\right) \left(\sum_{j=1}^{r-1} \chi_1(j)\log \left( 1-\zeta_r^j t^x \right)\right)\, dt.
		\end{align}
		We now simplify the sum over $k$. Since $\chi_2(r)=0$,
		\begin{align*}
			&\sum_{k=1}^{r-1} \frac{\chi_2(k)}{\zeta_r^{-k}-t} = \sum_{k=1}^{r} \frac{\chi_2(k)\zeta_r^{k}}{1-t\zeta_r^{k}} =\sum_{k=1}^{r} \chi_2(k)\zeta_r^{k} \sum_{\ell=0}^{\infty} t^\ell \zeta_r^{k\ell} = \sum_{\ell=0}^{\infty} t^{\ell} \sum_{k=1}^{r} \chi_2(k)\zeta_r^{(\ell+1)k} \\
			&= \sum_{b=0}^{\infty} \sum_{a=0}^{r-1} t^{br+a} \sum_{k=1}^{r} \chi_2(k)\zeta_r^{(br+a+1)k}  = \sum_{b=0}^{\infty} t^{br} \sum_{a=0}^{r-1} t^{a} \sum_{k=1}^{r} \chi_2(k)\zeta_r^{(a+1)k} = \frac{1}{1-t^r} \sum_{a=0}^{r-1} G(a+1,\chi_2) t^a \\
			&= \frac{1}{t(1-t^r)} \sum_{a=1}^{r} G(a,\chi_2) t^{a}.
		\end{align*}
		Substituting this expression in \eqref{double sum} and using  \eqref{J r chi def} and the fact that $\chi_1(r)=0$, we see that the proof of \eqref{J as a double sum} is complete.
	\end{proof}

	\begin{remark}\label{J-extended}
		As shown at the beginning of Section \ref{hzn}, the function $\mathscr{F}_z(x; u, v)$ is analytic in \textup{Re}$(z)>0$ whenever $u\in\mathbb{C}\backslash(1,\infty), u\neq0, v\in\mathbb{C}\backslash[1,\infty)$ and $x>0$. Thus, in particular, Lemma \ref{J r chi script F lemma} implies that $J_{r,\chi_1, \chi_2}(z,x)$ is analytic in \textup{Re}$(z)>0$.
	\end{remark}
	
	\noindent
	We are now ready to prove Theorem \ref{J r chi FE}.
	
\begin{proof}[Theorem \textup{\ref{J r chi FE}}][]
		Use lemma \ref{J r chi script F lemma} and Theorem \ref{small theorem} to see that
		\begin{align*}
			J_{r,\chi_1, \chi_2}(z,x) +x^{1-z} J_{r,\chi_2, \chi_1}\left(z,\frac{1}{x}\right) &=- \sum_{j=1}^{r} \sum_{k=1}^{r} \chi_1(j) \chi_2 (k) \left\{ \sum_{n=1}^{\infty} \sum_{m=1}^{\infty} \frac{\zeta_r^{jm+kn}}{m n (n+mx)^{z-1}} \right\}\\
			&= - \sum_{n=1}^{\infty} \sum_{m=1}^{\infty} \frac{G(m,\chi_1)G(n,\chi_2)}{m n (n+mx)^{z-1}}.
		\end{align*}
		In particular, the separability of Gauss sums for a primitive character $\chi$, that is, $G(a,\chi)= G(1,\chi) \bar\chi(a)$, implies \eqref{fe-primitive} whenever $\chi_1$ and $\chi_2$ are primitive. 
	\end{proof}

	\subsection{Analytic continuation of a double Lerch series using Crandall's method}
Observe that in the previous subsection, we proved \eqref{fe-general} for Re$(z)>1$. We want to let $z\to1$ on both sides of \eqref{fe-general} and then interchange the order of limit and integration as well as the order of limit and double sum to obtain some interesting identities. While the passage to the limit is straightforward as far as the left-hand side is concerned (because of the analyticity of $J_{r, \chi_1, \chi_2}(z, x)$ in Re$(z)>0$), it is quite non-trivial for the double sum on the right. We explain this by means of some classic examples.

First, consider the Dirichlet series associated to the M\"{o}bius function $\mu(n)$, that is, $\sum_{n=1}^{\infty}\mu(n)n^{-z}$. It is known to converge absolutely in Re$(z)>1$, in which case, we have $\sum_{n=1}^{\infty}\mu(n)n^{-z}=1/\zeta(z)$. However, putting $z=1$ is an altogether different affair. A theorem, due to Ingham, which comes in handy in this case, states that if $f$ is an arithmetic function with $|f(n)|\leq1$ and $F(z)=\sum_{n=1}^{\infty}f(n)n^{-z}$ is analytic for Re$(z)>1$, and if $F(z)$ can be analytically continued for Re$(z)\geq1$, then the series $\sum_{n=1}^{\infty}f(n)n^{-z}$ converges to $F(z)$ for Re$(z)\geq1$. We let $f(n)=\mu(n)$ and use this theorem along with the fact that $\zeta(z)$ is non-vanishing on the line Re$(z)=1$. This leads to $\sum_{n=1}^{\infty}\mu(n)/n=0$, which is equivalent to the prime number theorem. In particular,  $\lim_{z\to1}\sum_{n=1}^{\infty}\mu(n)n^{-z}=\sum_{n=1}^{\infty}\lim_{z\to1}\mu(n)n^{-z}$.

Another example which can be given in this regard is that of the Dirichlet $L$-function. It is known \cite[p.~79, Lemma 5.8]{zudilin-book} that the series defining the Dirichlet $L$-function converges uniformly in the half-plane Re$(z)>0$ for non-principal characters. In particular, $\lim_{z\to1}\sum_{n=1}^{\infty}\chi(n)n^{-z}=\sum_{n=1}^{\infty}\lim_{z\to1}\chi(n)n^{-z}$.

Note that in the case of the Dirichlet $L$-function, the series defining it is absolutely and uniformly convergent in Re$(z)>1$ but conditionally and uniformly convergent in $0<\textup{Re}(z)\leq1$. The uniform convergence in the latter region can be proved using Dirichlet's test. 

The question that arises then is, does there exist an analogue of Dirichlet's test for double series?  The only one which comes closer to it is due to Ogievecki\u{i} \cite[Theorem 3]{ogieveckii}. However, it is given only for double series of functions of real variable. 	

At this juncture, it is important to be aware of the contrast between results for single series versus double series. In his review of the book of Chelidze \cite{chelidze} in Mathematical Reviews, R. P. Boas says, \emph{`The general theme [of the book] is that most of the standard one-dimensional theorems do not generalize without some restriction either on the class of double series under consideration, or on the kind of limit used, or both.'}. 

Therefore, we resolve the problem of interchanging the order of $\lim_{z\to1}$ and double sum in \eqref{fe-general} by adopting a different approach. We first analytically extend the right-hand side of \eqref{fe-general} for any $z\in\mathbb{C}$ using Crandall's method \cite{crandall-original}, which involves a free parameter $\delta$, thereby obtaining analytically continued version of \eqref{fe-general}, namely, \eqref{fe general extended}, which is valid in Re$(z)>0$, and then let $z=1$. Dilcher \cite{dilcher-amm} has used Crandall's method for getting analytic continuation of character analogue of Mordell-Tornheim zeta function. Other work exploiting Crandall's method is that of Borwein and Dilcher \cite{Borwein-Dilcher}. 
 
Let $\mathscr{L}(s, \alpha)$ be the periodic zeta function defined by
	\begin{align*}
		\mathscr{L}(s, \alpha) = \sum_{n=1}^{\infty} \frac{e^{2\pi in\alpha}}{n^s}.
	\end{align*}
	Clearly, $\mathscr{L}(s, \alpha)=\textup{Li}_s(e^{2\pi i\alpha})$, where $\textup{Li}_s(\eta)$ is the polylogarithm function.
	For $\alpha\in\mathbb{R}$ and Re$(s)>1$, the above series converges absolutely. For $\alpha\in\mathbb{R}\backslash\mathbb{Z}$, the series converges conditionally for Re$(s)>0$. Moreover, 
	$\mathscr{L}(s, \alpha)$ can be analytically continued in the $s$-complex plane.
	
	Let $c,d,r\in\mathbb{N}$ such that $c,d<r$. For Re$(z)>1$ and $x>0$, we define $	\mathbb{W}_{c,d}(z,x)$ by
	\begin{align}\label{W defn}
		\mathbb{W}_{c,d}(z,x):=\sum_{n=1}^{\infty} \sum_{m=1}^{\infty} \frac{e^{ \frac{2\pi i mc}{r}}e^{\frac{2\pi i nd}{r}}}{m n (n+mx)^{z-1}}.
	\end{align}
	The above series is absolutely convergent for Re$(z)>1$. The following proposition gives the analytic continuation of $\mathbb{W}$ in the whole $z$-complex plane. Nakamura \cite[Theorem 2.1]{nakamura} has obtained analytic continuation of our series, but with $x=1$, using a different method. However, the representation that we obtain in the proposition below will be appropriate for our purposes.
	\begin{proposition}\label{crandall}
		$\mathbb{W}_{c,d}(z,x)$ has an analytic continuation in the whole $z$-complex plane. Further, for any $z \in \mathbb{C}$ and for any $\delta$ such that $0<\delta< \min\left(\frac{2\pi}{r},\frac{2\pi}{rx}\right)$, we have
		\begin{align}
			\mathbb{W}_{c,d}(z,x) &=\frac{1}{\Gamma(z-1)} \sum_{n=1}^{\infty} \sum_{m=1}^{\infty}  \frac{e^{ \frac{2\pi i mc}{r}}e^{\frac{2\pi i nd}{r}}}{m n (n+mx)^{z-1}} \Gamma(z-1,(n+mx)\delta) \notag \\ & + \frac{1}{\Gamma(z-1)} \sum_{u=0}^{\infty} \sum_{v=0}^{\infty} (-1)^{u+v} \frac{\mathscr{L}(1-u,\frac{c}{r})\mathscr{L}(1-v,\frac{d}{r})x^u \delta^{u+v+z-1}}{u!v! (u+v+z-1)}, \label{W result}
		\end{align}
		where $\Gamma(z,y)$ is the incomplete Gamma function. In particular,
		\begin{align}
			\mathbb{W}_{c,d}(1,x) = \mathscr{L}\left(1,\frac{c}{r}\right) \mathscr{L}\left(1,\frac{d}{r}\right).  \label{W special}
		\end{align}
	\end{proposition}
	\begin{proof}
		Let $0<\delta< \min\left(\frac{2\pi}{r},\frac{2\pi}{rx}\right)$. For $\textup{Re}(z)>1$, the definition of the incomplete Gamma function and the substitution $y=(n+mx)t$  leads to
		\begin{align*}
			\Gamma(z-1, (n+mx)\delta) = (n+mx)^{z-1} \int_{\delta}^{\infty} t^{z-2} e^{-(n+mx)t}\, dt. 
		\end{align*}
		Hence, from \eqref{W defn}, we see that
		\begin{align}
			&\Gamma(z-1) \mathbb{W}_{c,d}(z,x) \notag\\
			&= \sum_{n=1}^{\infty} \sum_{m=1}^{\infty} \frac{e^{ \frac{2\pi i mc}{r}}e^{\frac{2\pi i nd}{r}}}{m n} \int_{0}^{\infty} t^{z-2} e^{-(n+mx)t}\, dt \notag\\
			&= \sum_{n=1}^{\infty} \sum_{m=1}^{\infty} \frac{e^{ \frac{2\pi i mc}{r}}e^{\frac{2\pi i nd}{r}}}{m n}  \int_{0}^{\delta} t^{z-2} e^{-(n+mx)t}\, dt + \sum_{n=1}^{\infty} \sum_{m=1}^{\infty} \frac{e^{ \frac{2\pi i nc}{r}}e^{\frac{2\pi i nd}{r}}}{m n(n+mx)^{z-1}} 	\Gamma(z-1, (n+mx)\delta). \label{W middle}
		\end{align}
		Let
		\begin{align*}
			P(z,x):= \sum_{n=1}^{\infty} \sum_{m=1}^{\infty} \frac{e^{ \frac{2\pi i mc}{r}}e^{\frac{2\pi i nd}{r}}}{m n}  \int_{0}^{\delta} t^{z-2} e^{-(n+mx)t}\, dt.
		\end{align*}
		We can interchange the order of sum and integration in $P(z,x)$ using \cite[p.~30, Theorem 2.1]{temme} to get
		\begin{align}
			P(z,x)
			=  \int_{0}^{\delta} t^{z-2} \textup{Li}_1 \left( e^{\frac{2\pi i c}{r} -xt} \right)   \textup{Li}_1 \left( e^{\frac{2\pi i d}{r} -t} \right) \, dt. \label{P simplified}
		\end{align}
	Now for $s\in \mathbb{C} \backslash\mathbb{N}$ and Re$(s)>0$,
		\begin{align*}
			\textup{Li}_s \left( e^{\frac{2\pi i c}{r} -xt} \right) 
			= \sum_{a=0}^{\infty} \sum_{b=1}^{r} \frac{e^{ \frac{2\pi i (ar+b)c}{r}-(ar+b)xt}}{(ar+b)^s}
			= \frac{1}{r^s} \sum_{b=1}^{r} e^{\frac{2\pi ibc}{r}-bxt} \sum_{a=0}^{\infty} \frac{e^{ -arxt}}{\left(a+\frac{b}{r}\right)^s}.
		\end{align*}
		Using \cite[p.~29, Formula 1.11.8]{Erdelyi} with $z=e^{-rxt}$, where $0<t<\delta<\min\left(\frac{2\pi}{r},\frac{2\pi}{rx}\right)$, and $v=b/r$, for $s\notin\mathbb{N}$, we see that the inner sum transforms to
		\begin{align*}
		\sum_{a=0}^{\infty} \frac{e^{ -arxt}}{\left(a+\frac{b}{r}\right)^s}=e^{bxt}\left(\sum_{u=0}^{\infty} \frac{\zeta(s-u,\frac{b}{r})(-xrt)^{u}}{u!} +\Gamma(1-s)(rxt)^{s-1}\right).
		\end{align*} 
		Hence
		\begin{align}
			\textup{Li}_s \left( e^{\frac{2\pi i c}{r} -xt} \right) 
			= \frac{1}{r^s}\sum_{u=0}^{\infty} \frac{(-xrt)^u}{u!}  \sum_{b=1}^{r} e^{\frac{2\pi ibc}{r}} \zeta\left(s-u,\frac{b}{r}\right) + \frac{1}{r} \Gamma(1-s) (xt)^{s-1} \sum_{b=1}^{r} e^{\frac{2\pi i cb}{r}}.\label{polylog sum}
		\end{align}
		Observe that $1\leq c<r$ implies $\sum_{b=1}^{r} e^{\frac{2\pi i cb}{r}}=0$. Also, $ \sum_{b=1}^{r} e^{\frac{2\pi ibc}{r}} \zeta\left(s-u,\frac{b}{r}\right) =r^{s-u} \mathscr{L}(s-u, \frac{c}{r})$. Thus, \eqref{polylog sum} becomes
		\begin{align}
			\textup{Li}_s \left( e^{\frac{2\pi i c}{r} -xt} \right)=  \sum_{u=0}^{\infty} \mathscr{L}\left(s-u, \frac{c}{r}\right) \frac{(-xt)^u}{u!}. \label{Polylog final}
		\end{align}
		Letting $s \to 1$ in \eqref{Polylog final}, we get,
		\begin{align}
			\textup{Li}_1 \left( e^{\frac{2\pi i c}{r} -xt} \right)=  \sum_{u=0}^{\infty} \mathscr{L}\left(1-u, \frac{c}{r}\right) \frac{(-xt)^u}{u!}.  \label{Polylog final s=1}
		\end{align}
	Invoking \eqref{Polylog final s=1} twice in \eqref{P simplified}, we get
		\begin{align}
			P(z,x)= \int_{0}^{\delta} t^{z-2} \left( \sum_{u=0}^{\infty} \mathscr{L}\left(1-u, \frac{c}{r}\right) \frac{(-xt)^u}{u!}  \right) \left( \sum_{v=0}^{\infty} \mathscr{L}\left(1-v, \frac{d}{r}\right) \frac{(-t)^v}{v!} \right) \, dt. \label{P final}
		\end{align}
		Now we show that $\sum_{u=0}^{\infty} \mathscr{L}\left(1-u, \frac{c}{r}\right) \frac{(-xt)^u}{u!}$ converges absolutely and uniformly in $t$ on any compact subsets of $(0, \delta)$. We separate 
		the first two terms of this series so that
		\begin{align*}
			\sum_{u=0}^{\infty} \mathscr{L}\left(1-u, \frac{c}{r}\right) \frac{(-xt)^u}{u!} = \mathscr{L}\left( 1, \frac{c}{r}-1\right)-xt \mathscr{L}\left( 0, \frac{c}{r}-1 \right) +\sum_{u=2}^{\infty} \mathscr{L}\left(1-u, \frac{c}{r}\right) \frac{(-xt)^u}{u!}.
			\end{align*}
		Now let $a=1-\frac{c}{r}$, $s=1-u$, where $u\geq2$, and $x=1$ in the functional equation given in \cite[Theorem 10]{navas} to get
		\begin{align*}
			\mathscr{L}\left( 1-u, \frac{c}{r}-1 \right) = \frac{\Gamma(u)}{(2\pi)^u} \left\{ e^{\pi i \left( -\frac{u}{2}-\frac{2c}{r}+2 \right)} \zeta\left( u,1-\frac{c}{r} \right) + e^{\pi i \left( \frac{u}{2}-\frac{2c}{r} \right)} \zeta\left( u,\frac{c}{r} \right) \right\}.
		\end{align*}
		Since $\mathscr{L}\left( 1-u, \frac{c}{r}-1 \right)=\mathscr{L}\left( 1-u, \frac{c}{r} \right)$ and the Hurwitz zeta function is bounded as $u\to\infty$, we can see $\mathscr{L}\left( 1-u, \frac{c}{r} \right)=O\left(\frac{\Gamma(u)}{(2\pi)^u}\right)$ as $u\to\infty$. Since $0<t<\delta<\frac{2\pi}{x}$, the series $\sum_{u=0}^{\infty} \mathscr{L}\left(1-u, \frac{c}{r}\right) \frac{(-xt)^u}{u!}$  converges absolutely and uniformly on any compact subsets of $(0, \delta)$. Moreover,
		\begin{align}
			\sum_{u=0}^{\infty} \mathscr{L}\left(1-u, \frac{c}{r}\right) \frac{(-xt)^u}{u!} = \mathscr{L}\left( 1, \frac{c}{r}-1\right)-xt \mathscr{L}\left( 0, \frac{c}{r}-1 \right) + O\left( \sum_{u=2}^{\infty} \frac{\Gamma(u)}{u!} \left(\frac{-xt}{2\pi}\right)^u \right) = O(1), \label{Scr L bound}
		\end{align}
		Thus, by \cite[p.~30, Theorem 2.1]{temme} interchanging the order of integration and double sum in \eqref{P final}, for Re$(z)>1$, we have
		\begin{align}
			P(z,x)
			= \sum_{u=0}^{\infty} \sum_{v=0}^{\infty} (-1)^{u+v} \frac{\mathscr{L}\left( 1-u,\frac{c}{r} \right)\mathscr{L}\left( 1-v,\frac{d}{r} \right)x^u \delta^{u+v+z-1}}{u!v!(u+v+z-1)}. \label{P solved}
		\end{align} 
		Thus, from \eqref{W middle} and \eqref{P solved}, we get \eqref{W result} for Re$(z)>1$. We now show that $\mathbb{W}_{c,d}(z,x)$ can be analytic continued to an entire function in $z$.	Use the asymptotic expansion of the incomplete gamma function \cite[p.~179, Equation 8.11.2]{nist}, namely, for any $z$ and $\delta$ and large values of $m$ and $n$,   
		\begin{align*}
			\Gamma(z-1,(n+mx)\delta) = ((n+mx)\delta)^{z-2} e^{-(n+mx)\delta} \left( 1+ O_z\left(  \frac{1}{(n+mx)\delta}\right) \right)
		\end{align*}
		in the first series on the right-hand side of \eqref{W result} so as to have
		\begin{align*}
			&\frac{1}{\Gamma(z-1)} \sum_{n=1}^{\infty} \sum_{m=1}^{\infty}  \frac{e^{ \frac{2\pi i mc}{r}}e^{\frac{2\pi i nd}{r}}}{m n (n+mx)^{z-1}} \Gamma(z-1,(n+mx)\delta) \\ 
			&= \frac{\delta^{z-2}}{\Gamma(z-1)} \sum_{n=1}^{\infty} \sum_{m=1}^{\infty}  \frac{e^{ \frac{2\pi i mc}{r}}e^{\frac{2\pi i nd}{r}}}{m n (n+mx)} e^{-(n+mx)\delta} \left( 1+ O_z\left(  \frac{1}{(n+mx)\delta}\right) \right),
		\end{align*}
		which implies that the first series on the right-hand side of \eqref{W result} converges absolutely and uniformly for any $z\in\mathbb{C}$. In the second double series, for any $z \in \mathbb{C}$, we see that $u+v+\textup{Re}(z)-1>1$ for all but first finitely many $u$ and $v$. Therefore, by using the bound shown in \eqref{Scr L bound} and the fact that  $0<\delta,x\delta<2\pi$, we see that the second expression on the right-hand side of \eqref{W result} converges for any $z\in\mathbb{C}$. This completes the proof of analytic continuation of $	\mathbb{W}_{c,d}(z,x)$ in the whole $z$-complex plane.
		
		In order to obtain \eqref{W special}, we separate the $u=v=0$ term of the second double series and then let $z\to1$ in \eqref{W result}. 
	\end{proof}
\begin{remark}\label{gauss - analytic continuation}
Using \eqref{W defn} and the definition of Gauss sum, we see that for any Dirichlet characters $\chi_1$ and $\chi_2$ modulo $r$, for $\textup{Re}(z)>1$,
\begin{align}\label{remark identity}
	\sum_{n=1}^{\infty} \sum_{m=1}^{\infty} \frac{G(m,\chi_1)G(n,\chi_2)}{m n (n+mx)^{z-1}} = \sum_{c=0}^{r} \sum_{d=0}^{r} \chi_1(c) \chi_2(d) \mathbb{W}_{c,d}(z,x).
\end{align}
Therefore, with the help of Proposition \ref{crandall}, we see that the right-hand side of the above equation provides the analytic continuation of the left-hand side in the whole complex plane.
\end{remark}
	
\begin{theorem}\label{extended version}
Let $\chi_1$,$\chi_2$ be any Dirichlet characters modulo $r$. Let $J_{r,\chi_1, \chi_2}(z,x) $ be defined in \eqref{J r chi def}. For $\textup{Re}(z)>0$, $x>0$, $r>1$ and $\delta$ such that $0<\delta< \min\left(\frac{2\pi}{r},\frac{2\pi}{rx}\right)$, we have
\begin{align}\label{fe general extended}
	J_{r,\chi_1, \chi_2}(z,x) +x^{1-z} J_{r,\chi_2, \chi_1}\left(z,\frac{1}{x}\right) = -\sum_{c=0}^{r} \sum_{d=0}^{r} \chi_1(c) \chi_2(d) \mathbb{W}_{c,d}(z,x),
\end{align}
where $\mathbb{W}_{c,d}(z,x)$ is defined in \eqref{W result}.
\end{theorem}
\begin{proof}
This follows from Theorem \ref{J r chi FE}, Remarks \ref{J-extended} and \ref{gauss - analytic continuation}, and the principle of analytic continuation.
\end{proof}
When $z=1$ in Theorem \ref{extended version}, the right-hand side of \eqref{fe general extended} simplifies considerably. To see this, define
\begin{equation*}
	\mathfrak{D}_{\chi_1, \chi_2}=\begin{cases} 
	\pi^2 L(0,\chi_1) L(0,\chi_2), & \text{if } \chi_1 \text{ and } \chi_2 \text{ are odd,} \\
	- 4 L'(0,\chi_1) L'(0,\chi_2), & \text{if }\chi_1 \text{ and } \chi_2 \text{ are even,} \\
	-2 \pi i L(0,\chi_1) L'(0,\chi_2), &\text{if } \chi_1 \text{ is odd and } \chi_2 \text{ is even,} \\
	-2 \pi i L'(0,\chi_1) L(0,\chi_2), & \text{if } \chi_1 \text{ is even and } \chi_2 \text{ is odd.}  \\
\end{cases}
\end{equation*}
\begin{corollary}\label{corollary L zero chi} 
	Let $\chi_1,\chi_2$ be non-trivial Dirichlet characters modulo $r$. Then,
	\begin{align*}
		J_{r,\chi_1, \chi_2}(1,x) + J_{r,\chi_2, \chi_1}\left(1,\frac{1}{x}\right) =
		\mathfrak{D}_{\chi_1, \chi_2}.
	\end{align*}
\end{corollary}
\begin{proof}
	If $\chi$ is an odd character, then using \cite[p.~71, Theorem 3.44]{rudin}, we see that
	\begin{align}
		\sum_{m=1}^{\infty} \frac{G(m,\chi)}{m} &= \sum_{m=1}^{\infty} \frac{1}{m} \left( \sum_{j=1}^{r} \chi(j) \zeta_r^{jm} \right) = - \sum_{j=1}^{r} \chi(j) \log(1-\zeta_r^j) = i \sum_{j=1}^{r-1} \chi(j)  \tan^{-1} \left( \frac{ \sin \left( \frac{2\pi j}{r} \right) }{1-\cos \left( \frac{2\pi j}{r} \right) } \right) \notag \\
		&= i \sum_{j=1}^{r-1} \chi(j) \left( \frac{\pi}{2} - \frac{2\pi j}{2r} \right) = i \pi \left( \frac{1}{2}\sum_{j=1}^{r-1} \chi(j) - \frac{1}{r} \sum_{j=1}^{r-1} j\chi(j)   \right) = -\frac{i\pi}{r} \sum_{j=1}^{r-1} j\chi(j) = i\pi L(0,\chi), \label{Gauss sum odd}
	\end{align}
	where in the last step, we used \cite[Corollary 10.3.2]{cohen}. If $\chi$ is an even character,
	\begin{align}
		\sum_{m=1}^{\infty} \frac{G(m,\chi)}{m} &=- \sum_{j=1}^{r} \chi(j) \log(1-\zeta_r^j) = - \frac{1}{2} \sum_{j=0}^{r} \chi_1(j) \log \left(  2-2 \cos\left(\frac{2 \pi j}{r}\right)   \right) \notag \\
		&= - \frac{1}{2} \sum_{j=0}^{r} \chi_1(j) \log \left(  4 \sin^2\left(\frac{\pi j}{r}\right)   \right) =-\sum_{j=0}^{r} \chi_1(j) \log \left(  2 \sin\left(\frac{\pi j}{r}\right)   \right) \notag \\
		&= -\sum_{j=0}^{r} \chi_1(j) \log \left(  2    \right) -\sum_{j=0}^{r} \chi_1(j) \log \left(  \sin\left(\frac{\pi j}{r}\right)   \right) = 2 L'(0,\chi), \label{Gauss sum even}
	\end{align} 
	where we used \cite[Corollary 10.3.5]{cohen} in the last step. Now let $z\to1$ in \eqref{remark identity}
	 and use \eqref{W special} to get
	\begin{align}\label{passing the limit}
		\lim_{z \to 1} \sum_{n=1}^{\infty} \sum_{m=1}^{\infty} \frac{G(m,\chi_1)G(n,\chi_2)}{m n (n+mx)^{z-1}} &= \sum_{c=0}^{r} \sum_{d=0}^{r} \chi_1(c) \chi_2(d) \lim_{z \to 1} \mathbb{W}_{c,d}(z,x) \nonumber\\
		&= \sum_{c=0}^{r} \sum_{d=0}^{r} \chi_1(c) \chi_2(d) \mathscr{L}\left(1,\frac{c}{r}\right) \mathscr{L}\left(1,\frac{d}{r}\right) \nonumber\\
		&= \left(\sum_{m=1}^{\infty} \frac{G(m,\chi_1)}{m} \right) \left(\sum_{n=1}^{\infty} \frac{G(n,\chi_2)}{n} \right).
	\end{align}
	Thus,
	\begin{align}
		J_{r,\chi_1, \chi_2}(1,x) + J_{r,\chi_2, \chi_1}\left(1,\frac{1}{x}\right) = - \left(\sum_{m=1}^{\infty} \frac{G(m,\chi_1)}{m} \right) \left(\sum_{n=1}^{\infty} \frac{G(n,\chi_2)}{n} \right). \label{J r chi FE z=1}
	\end{align}
	From \eqref{Gauss sum odd}, \eqref{Gauss sum even} and \eqref{J r chi FE z=1}, we get the required result.
\end{proof}

\begin{corollary}\label{corollary primitive character}
	Let $\chi$ be a primitive Dirichlet character modulo $r$. Then,
	\begin{align*}
		J_{r,\chi, \bar\chi}(1,x) + J_{r,\bar\chi, \chi}\left(1,\frac{1}{x}\right) = -\chi(-1) r |L(1,\chi)|^2.
	\end{align*}
\end{corollary}
\begin{proof}
	Employ the separability of Gauss sums $G(m,\chi)=G(1,\chi)\bar{\chi}(m)$ in Theorem \ref{J r chi FE} along with the fact that $G(1,\chi)G(1,\bar\chi)$ is $r$ for $\chi$ even and $-r$ for $\chi$ odd  to see that
	\begin{align*}
		J_{r,\chi, \bar\chi}(1,x) + J_{r,\bar\chi, \chi}\left(1,\frac{1}{x}\right) = -\chi(-1)r \sum_{n=1}^{\infty} \sum_{m=1}^{\infty} \frac{\bar\chi(m) \chi(n)}{m n}= -\chi(-1) r L(1,\chi) L(1,\bar\chi) = -\chi(-1) r |L(1,\chi)|^2.
	\end{align*}
	This completes the proof.
\end{proof}

\begin{corollary}\label{from introduction}
	Let $\chi$ be a real primitive Dirichlet characters modulo $r$. Then,
	\begin{align*}
		J_{r,\chi, \chi}(1,1) = -\frac{1}{2} \chi(-1) r L^2(1,\chi).
	\end{align*}
\end{corollary}
\begin{proof}
	Let $x=1$ in Corollary \ref{corollary primitive character} and use the fact $|L(1,\chi)|^2=L^2(1,\chi)$.
\end{proof}

\begin{corollary}\label{unnamed corollary}
The functional equations for $J(x)$ and $T(x)$, given in \eqref{JT-FE}, hold.
\end{corollary}
\begin{proof}
To prove the relation for $J(x)$, let $r=2$, $\chi_1=\chi_2=\chi$, the character modulo $2$, and $z=1$ in Theorem \ref{extended version}, then use \eqref{W special}, and the facts that $\mathscr{L}(1,1/2)=-\log(2)$ and $J_{2, \chi, \chi}(1, x)=-J(x)$. 

For the relation for $T(x)$, we let $r=4$, $\chi_1=\chi_2=\chi$, the non-trivial character modulo $4$, and $z=1$ in Corollary \ref{corollary L zero chi}, and the facts that $L(0, \chi)=1/2$, $J_{4, \chi, \chi}(1,x)=4 T(x)$. 
\end{proof}

\subsection{Explicit evaluations of certain series and integrals}
As mentioned in the introduction, one of the applications of Theorem \ref{J r chi FE} is to explicitly evaluate the series or integrals occurring in it. We first prove \eqref{new series evaluation}.
	\begin{proof}[ \textup{\eqref{new series evaluation}}][]
	 Let $z=2, x=1$ and  $\chi_1=\chi_2=\chi$, the non-trivial Dirichlet character modulo $6$.  Note that
	 \begin{align}\label{B(6, 2)}
	 	\sum_{j=1}^{6} \chi(j) \tan^{-1} \left( \frac{ \sin \left( \frac{\pi j}{3} \right) t^x}{1-\cos \left( \frac{\pi j}{3} \right) t^x} \right) = \tan^{-1} \left( \frac{\sqrt{3}t^x}{2-t^x} \right)  -  \tan^{-1} \left( \frac{-\sqrt{3}t^x}{2-t^x} \right) =2 \tan^{-1} \left( \frac{\sqrt{3}t^x}{2-t^x} \right).
	 \end{align}
	From \eqref{B(6, 2)}, Lemma \ref{J r chi odd even lemma}  below, and the above choices of $z$, $x$, $\chi_1$ and $\chi_2$, we have
	 \begin{align*}
J_{6, \chi, \chi}(2, 1)&=-2i\sqrt{3}\int_{0}^{1}\frac{\log(t)}{t^2-t+1} \tan^{-1} \left( \frac{t\sqrt{3}}{2-t}\right)\, dt\nonumber\\
&=\frac{1}{36}\bigg[3\log^{3}(3)-5\pi^2\log(3)+\pi\sqrt{3}\big\{\psi'\left(\tfrac{1}{6}\right)+\psi'\left(\tfrac{1}{3}\right)-\psi'\left(\tfrac{2}{3}\right)-\psi'\left(\tfrac{5}{6}\right)\big\}+\left(18\log(3)-30\pi i\right)\textup{Li}_2\big(\tfrac{3-i\sqrt{3}}{6}\big)\nonumber\\
&\qquad\quad+\left(18\log(3)+30\pi i\right)\textup{Li}_2\big(\tfrac{3+i\sqrt{3}}{6}\big)-36\left(\textup{Li}_3\big(\tfrac{3-i\sqrt{3}}{6}\big)+\textup{Li}_3\big(\tfrac{3+i\sqrt{3}}{6}\big) \right) \bigg]
,
	 \end{align*}
	 where we evaluated the integral using \emph{Mathematica 12}. Therefore \eqref{fe-general} leads to \eqref{new series evaluation}.
		\end{proof}
Further evaluations are discussed now. Letting $z=2, r=4, x=1$ and $\chi_1=\chi_2=\chi$ to be the non-trivial Dirichlet character modulo $4$ in \eqref{fe-general}, and observing that
\begin{equation*}
\sum_{\ell=1}^{r} G(\ell,\chi_2) t^{\ell}=2it(1-t^2),\hspace{4mm} \sum_{j=1}^{r} \chi_1(j) \log \left( 1-\zeta_r^j t\right)=-2i\tan^{-1}(t),\hspace{4mm}
\end{equation*}
we get
\begin{align*}
	-2	\int_{0}^{1} \frac{ \log(t) \tan^{-1}(t)}{1+t^2} \, dt= \sum_{n=1}^{\infty} \sum_{m=1}^{\infty} \frac{\chi(m)\chi(n)}{m n (m+n)}.
\end{align*}
However, since \cite[p.~253]{valean}
\begin{align*}
	\int_{0}^{1} \frac{ \log(t) \tan^{-1}(t)}{1+t^2} \, dt = \int_{0}^{\frac{\pi}{4}} y \log(\tan(y)) \, dy
	= \frac{-1}{16} (4\pi G - 7 \zeta(3)),
\end{align*}
where $G$ is Catalan's constant, we arrive at
\begin{align}\label{explicit-evaluation}
	\sum_{n=0}^{\infty}\sum_{m=0}^{\infty}\frac{(-1)^{m+n}}{(2m+1) (2n+1)(2m+2n+2)}=\frac{1}{8}(4\pi G-7\zeta(3)).
\end{align}
This is not a new result but essentially the one obtained by Bailey and Borwein \cite[p.~318]{bailey-borwein} although our proof is new. 
We now obtain another explicit evaluation with the help of \eqref{fe-general}. 
\begin{corollary}
Let $k\in\mathbb{N}$. Then
	\begin{align}\label{2k}		
	\int_{0}^{1} \frac{(\log(\tfrac{1}{y}))^{2k-1}\log(1+y)}{1+y} \, dy = \Gamma(2k) \left( \frac{1}{2}(2k-1) (2^{-2k}-1) \zeta(2k+1) - \sum_{j=0}^{k-1} (2^{1-2j}-1) \zeta(2j)   \zeta(2k-2j+1) \right).
\end{align}
\end{corollary}
\begin{proof}
Let $z=2k, k\in\mathbb{N}$, $x=1$, $r=2$, $\chi_1=\chi_2=\chi$, the Dirichlet character modulo $2$ in \eqref{fe-general} so that 
\begin{equation*}
	\frac{2}{\Gamma(2k)}\int_{0}^{1}\frac{\log^{2k-1}(1/y)\log(1+y)}{1+y}\, dy=\sum_{n=1}^{\infty} \sum_{m=1}^{\infty} \frac{(-1)^{n+m}}{nm(n+m)^{2k-1}}.
\end{equation*}
Now specializing Theorem 3.4 of Tsumura \cite{tsumura} by letting $k=l=1$ and $h=2k-1$, and simplifying, we are led to
\begin{align*}
	\sum_{n=1}^{\infty} \sum_{m=1}^{\infty} \frac{(-1)^{n+m}}{nm(n+m)^{2k-1}} = - 2(2k-1) \phi(0) \phi(2k+1) + 4 \phi(0) \sum_{j=0}^{k-1} \phi(2j)   \zeta(2k-2j+1),
\end{align*}
where $\phi(s) = (2^{1-s}-1) \zeta(s)$. This gives \eqref{2k} upon simplification.
\end{proof}
Another explicit series evaluation is given next.
\begin{corollary}\label{before bbg}
	Let $\textup{Li}_s(\eta)$ denote the polylogarithm function. Then
		\begin{align*}
		\sum_{n=1}^{\infty} \sum_{m=1}^{\infty} \frac{(-1)^{m+n}}{mn (n+m)^2} =\frac{1}{24} \left( -\pi^4  -4\pi^2 \log^2(2) +4 \log^4(2) + 96 \textup{Li}_4 \left( \frac{1}{2} \right) + 84 \zeta(3)\log(2)\right).
	\end{align*}
\end{corollary}
\begin{proof}
Let $z=3$ and $x=1$ in Theorem \ref{new J theorem} to obtain 
	\begin{align*}
	2J(3,1) = -\sum_{n=1}^{\infty} \sum_{m=1}^{\infty} \frac{(-1)^{m+n}}{mn (n+m)^2}.
\end{align*}
However, from \cite[Equation 25, p.213]{Gastmans},
\begin{align*}
	 \int_{0}^{1} \frac{\log^2(y) \log(1+y)}{1+y} \, dy = \frac{1}{24} \left( -\pi^4  -4\pi^2 \log^2(2) +4 \log^4(2) + 96 \textup{Li}_4 \left( \frac{1}{2} \right) + 84 \log(2) \zeta(3)\right).
\end{align*}
The result now follows from the two equations given above and the definition of $J(z, x)$ given in \eqref{J(z,x) def}.
\end{proof}
\textbf{Alternative proof of Corollary \ref{before bbg}.} In \cite[p.~291]{bbg}, Borwein, Borwein and Girgensohn have shown that
\begin{align*}
\sum_{n=1}^{\infty}\frac{(-1)^{n}}{n^3}\sum_{k=1}^{n-1}\frac{1}{k}=\frac{1}{48} \left( -\pi^4  -4\pi^2 \log^2(2) +4 \log^4(2) + 96 \textup{Li}_4 \left( \frac{1}{2} \right) + 84 \zeta(3)\log(2)\right).
\end{align*}
The result now follows upon observing that
\begin{align*}
\sum_{n=1}^{\infty}\frac{(-1)^{n}}{n^3}\sum_{k=1}^{n-1}\frac{1}{k}=\sum_{k=1}^{\infty}\frac{1}{k}\sum_{n=k+1}^{\infty}\frac{(-1)^{n}}{n^3}=\sum_{k=1}^{\infty}\sum_{n=1}^{\infty}\frac{(-1)^{n+k}}{k(k+n)^3}=\frac{1}{2}\sum_{k=1}^{\infty}\sum_{n=1}^{\infty}\frac{(-1)^{n+k}}{kn(k+n)^2}.
\end{align*}
	\subsection{Further identities through differentiation }\label{extend}
	In this section, we show that the differentiating \eqref{fe-general} with respect to $z$ yields further new and  interesting identities. This theorem given below has \eqref{loglog2z2} and \eqref{loglog4z2} as its special cases.
\begin{theorem}\label{diff id}
Let $r>1$ be a natural number and let $\zeta_r=e^{2\pi i/r}$.Let $\chi_1$ and $\chi_2$ are any Dirichlet characters modulo $r$. For \textup{Re}$(z)>1$, define
	\begin{align*}
		K_{r, \chi_1, \chi_2}(z):=\frac{1}{\Gamma(z)}\int_{0}^{1}\frac{\left(\log\left( \frac{1}{t} \right)\right)^{z-1}\log\log\left( \frac{1}{t} \right)}{t(1-t^r)} \left( \sum_{\ell=1}^{r} G(\ell,\chi_2) t^{\ell}\right) \left(  \sum_{j=1}^{r} \chi_1(j) \log \left( 1-\zeta_r^j t \right) \right) \, dt.
	\end{align*}
Then
\begin{align*}
K_{r, \chi_1, \chi_2}(z)=\frac{1}{2}\left\{\sum_{n=1}^{\infty}
\sum_{m=1}^{\infty}\frac{G(m, \chi_1)G(n, \chi_2)\log(n+m)}{nm(n+m)^{z-1}}-\psi(z)\sum_{n=1}^{\infty}
\sum_{m=1}^{\infty}\frac{G(m, \chi_1)G(n, \chi_2)}{nm(n+m)^{z-1}}\right\},
\end{align*}
\end{theorem}
\begin{proof}
Let $x=1$ in \eqref{fe-general} and differentiate both sides of the resulting identity with respect to $z$ to get
\begin{align}\label{just final1}
2\left(K_{r, \chi_1, \chi_2}(z)-\psi(z)J_{r, \chi_1, \chi_2}(z, 1)\right)=\sum_{n=1}^{\infty}
\sum_{m=1}^{\infty}\frac{G(m, \chi_1)G(n, \chi_2)\log(n+m)}{nm(n+m)^{z-1}}.
\end{align}
Employing \eqref{fe-general} once again with $x=1$ then leads to the desired result. 
\end{proof}

\begin{corollary}
The identities \eqref{loglog2z2} and \eqref{loglog4z2} hold.
\end{corollary}
\begin{proof}
For \eqref{loglog2z2}, let $r=2$,  $\chi_1=\chi_2=\chi$, the Dirichlet character modulo $2$ in \eqref{just final1}. Then we find that
$\sum_{\ell=1}^{r} G(\ell,\chi_2) t^{\ell}=-t+t^2, \sum_{j=1}^{r} \chi_1(j) \log \left( 1-\zeta_r^j t\right)=\log(1+t)$ and, $G(m, \chi)=(-1)^m$. Now letting $z=2$ in the resulting identity and using the fact \cite[p.~245]{subbarao}  $\sum_{n=1}^{\infty}\sum_{m=1}^{\infty}\frac{(-1)^{n+m}}{nm(n+m)}=\frac{1}{4}\zeta(3)$ along with the value $\psi(2)=1-\gamma$, we obtain \eqref{loglog2z2} upon simplification.

As for \eqref{loglog4z2}, we let $z=2$, $r=4$ and $\chi_1=\chi_2=\chi$, the odd Dirichlet character modulo $4$ in \eqref{just final1}. Then noting that $\sum_{\ell=1}^{r} G(\ell,\chi_2) t^{\ell}=2it(1-t^2), \sum_{j=1}^{r} \chi_1(j) \log \left( 1-\zeta_r^j t\right)=-2i\tan^{-1}(t)$, and using \eqref{explicit-evaluation}, we obtain \eqref{loglog4z2}.
\end{proof}

\begin{remark}
Since \eqref{new series evaluation} is known, one may take $r=6, \chi_1=\chi_2=\chi$, the non-trivial Dirichlet character modulo $6$ to get an identity similar to \eqref{loglog2z2} and \eqref{loglog4z2}.
\end{remark}
	\subsection{Examples and illustrations}\label{ei}
	
	Here we illustrate Theorem \ref{J r chi FE} with the help of two examples. Before that, a lemma to simplify the integrand of $J_{r, \chi_1, \chi_2}(z, x)$ depending upon the parity of the character $\chi_1$ is derived.
	
		\begin{lemma}\label{J r chi odd even lemma}
		If $\chi_1$ is an odd Dirichlet character, then 
		\begin{align*}
			J_{r,\chi_1, \chi_2}(z,x) = \frac{-i}{\Gamma(z)} \int_{0}^{1} \frac{\left( \log \left( \frac{1}{t} \right) \right)^{z-1}}{t(1-t^r)} \left( \sum_{\ell=1}^{r} G(\ell,\chi_2) t^{\ell}\right) \left(  \sum_{j=1}^{r} \chi_1(j) \tan^{-1} \left( \frac{ \sin \left( \frac{2\pi j}{r} \right) t^x}{1-\cos \left( \frac{2\pi j}{r} \right) t^x} \right) \right) \, dt,
		\end{align*}
		whereas if $\chi_1$ is even, then
		\begin{align*}
			J_{r,\chi_1, \chi_2}(z,x) = \frac{1}{2\Gamma(z)} \int_{0}^{1} \frac{\left( \log \left( \frac{1}{t} \right) \right)^{z-1}}{t(1-t^r)} \left( \sum_{\ell=1}^{r} G(\ell,\chi_2) t^{\ell}\right) \left(  \sum_{j=1}^{r} \chi_1(j) \log \left( 1+t^{2x} - 2 \cos \left( \frac{2\pi j}{r} \right) t^x \right) \right) \, dt.
		\end{align*}
	\end{lemma}
	\begin{proof}
		Consider the following finite sum over $j$ occurring in the definition of \eqref{J r chi def}:
		\begin{align*}
			\sum_{j=1}^{r} \chi_1(j) \log \left( 1-\zeta_r^j t^{x} \right)
			&= \sum_{j=1}^{r} \chi_1(j) \log \left( 1- \cos\left(\frac{2 \pi j}{r}\right)t^x + i \sin\left(\frac{2 \pi j}{r}\right)  t^{x} \right) \\
			&=: A+B,
		\end{align*}
		where
		\begin{align*}
			A =  \sum_{j=0}^{r} \chi_1(j) \left\{ \frac{1}{2} \log \left(  1+t^{2x}-2 \cos\left(\frac{2 \pi j}{r}\right)t^{x}   \right) \right\},\hspace{4mm}
			B = -i \sum_{j=0}^{r} \chi_1(j)  \tan^{-1} \left( \frac{ \sin \left( \frac{2\pi j}{r} \right) t^x}{1-\cos \left( \frac{2\pi j}{r} \right) t^x} \right).
		\end{align*}
		On replacing $j$ by $r-j$ in $A$ and adding the resulting expression to $A$, we see that
		\begin{align*}
			2 A =  \frac{1}{2} \sum_{j=0}^{r} \left(\chi_1(j) + \chi_1(r-j) \right)  \log \left(  1+t^{2x}-2 \cos\left(\frac{2 \pi j}{r}\right)t^{x}   \right).
		\end{align*}
		Hence, $A=0$ if $\chi_1$ is an odd character. Now replace $j$ by $r-j$ in $B$ and add the resulting expression to $B$ so that
		\begin{align*}
			2 B = -i \sum_{j=0}^{r} \left(\chi_1(j) - \chi_1(r-j)\right)  \tan^{-1} \left( \frac{ \sin \left( \frac{2\pi j}{r} \right) t^x}{1-\cos \left( \frac{2\pi j}{r} \right) t^x} \right).
		\end{align*}
		Hence, $B=0$ if $\chi_1$ is an even character. This completes the proof.
	\end{proof}

\begin{example}
Let $\chi$ be the non-trivial Dirichlet character modulo 6 defined below.
\begin{center}
	\begin{tabular}{|c|c c c c c c|} 
		$j$  & 1 & 2 & 3 & 4 & 5 & 6 \\
		\hline 
		$\chi(j)$  & 1 & 0 & 0 & 0 & -1 & 0
	\end{tabular}
\end{center}
Observe that
\begin{align*}
	G(\ell,\chi) =
	\begin{cases} 
		i \sqrt{3} & \text{if } \ell \equiv 1,2 \ (\text{mod } 6), \\
		-i \sqrt{3} & \text{if }\ell \equiv 4,5 \ (\text{mod } 6), \\
		0 &\text{else. } \\
	\end{cases}
\end{align*}
Then $\sum_{\ell=1}^{6} G(\ell,\chi) t^{\ell} = i\sqrt{3} t(1+t-t^3-t^4)$. From \eqref{B(6, 2)} and the fact that $\chi$ is odd, we get, using Lemma \ref{J r chi odd even lemma},
\begin{align*}
	J_{6,\chi, \chi}(z,x) = \frac{2\sqrt{3}}{\Gamma(z)} \int_{0}^{1} \frac{\left( \log \left( \frac{1}{t} \right) \right)^{z-1}}{t^2-t+1} \tan^{-1} \left( \frac{\sqrt{3}t^x}{2-t^x} \right)  \, dt.
\end{align*}
Put $z=1$ in \eqref{fe general extended} to get  
\begin{align}
	J_{6,\chi, \chi}(1,x) + J_{6,\chi, \chi} \left(1,\frac{1}{x}\right)  =\frac{4 }{9}\pi^2.  \label{illustration 1 z=1}
\end{align}
Put $x=1$ in \eqref{illustration 1 z=1} to get
\begin{align*}
	\int_{0}^{1} \frac{1}{t^2-t+1} \tan^{-1} \left( \frac{\sqrt{3}t}{2-t} \right)  \, dt = \frac{1}{9\sqrt{3}} \pi^2.
\end{align*}
\end{example}

\begin{example}
 Let $\chi$ be the non-trivial primitive Dirichlet character modulo 16 defined below.
\begin{center}
	\begin{tabular}{|c|c c c c c c c c c c c c c c c c|} 
		$j$  & 1 & 2 & 3 & 4 & 5 & 6 & 7 & 8 & 9 &10 &11 &12&13 & 14& 15 &16 \\
		\hline 
		$\chi(j)$  & 1 & 0 & $i$ & 0 & $-i$ & 0 &-1 & 0 &-1 & 0 &$-i$& 0 &$i$& 0& 1& 0   
	\end{tabular}
\end{center}
Since
\begin{align*}
	G(\ell,\bar{\chi}) =
	\begin{cases} 
		-4 (-1)^{\frac{7}{8}} & \text{if } \ell \equiv \pm 1 \ (\text{mod } 16), \\
		-i 4 (-1)^{\frac{7}{8}} & \text{if }\ell \equiv \pm 3 \ (\text{mod } 16), \\
		i 4 (-1)^{\frac{7}{8}} & \text{if }\ell \equiv \pm 5 \ (\text{mod } 16), \\
		4 (-1)^{\frac{7}{8}} & \text{if }\ell \equiv \pm 7 \ (\text{mod } 16), \\
		0 &\text{else }, \\
	\end{cases}
\end{align*}
we have
$\sum_{\ell=1}^{16} G(\ell,\bar{\chi}) t^{\ell} = -4 (-1)^{\frac{7}{8}} t(1+it^2-it^4-t^6-t^8-it^{10}+it^{12}+t^{14})$.
	Also,
	\begin{align*}
	&\sum_{j=1}^{16} \chi(j) \log \left( 1+t^{2x} - 2 \cos \left( \frac{2\pi j}{16} \right) t^x \right) =  \log\left( \frac{1+t^{2x} - (\sqrt{2+\sqrt{2}})t^x}{1+t^{2x} + (\sqrt{2+\sqrt{2}})t^x}  \right) +i \log\left( \frac{1+t^{2x} - (\sqrt{2-\sqrt{2}})t^x}{1+t^{2x} + (\sqrt{2-\sqrt{2}})t^x}  \right) .
\end{align*}
Hence, from Lemma \ref{J r chi odd even lemma},
\begin{align*}
	J_{16,\chi, \bar\chi}(z,x) = \frac{4(-1)^{\frac{7}{8}}}{\Gamma(z)} \int_{0}^{1} & \frac{\left( \log \left( \frac{1}{t} \right) \right)^{z-1}(t^2-1)(t^4+(1+i)t^2+1)}{t^8+1} \\
	&\times \left\{\log\left( \frac{1+t^{2x} - (\sqrt{2+\sqrt{2}})t^x}{1+t^{2x} + (\sqrt{2+\sqrt{2}})t^x}  \right) 
	+ i \log\left( \frac{1+t^{2x} - (\sqrt{2-\sqrt{2}})t^x}{1+t^{2x} + (\sqrt{2-\sqrt{2}})t^x}  \right) \right\}  \, dt.
\end{align*}
Observe that $	J_{16,\bar\chi, \chi}(z,x)=\overline{	J_{16,\chi, \bar\chi}(z,x)}$.
Put $z=1$ and $\chi_1=\chi=\bar{\chi_2}$ in \eqref{fe general extended} to get  
\begin{align*}
	J_{16,\chi, \bar\chi}(1,x) + J_{16,\bar\chi, \chi} \left(1,\frac{1}{x}\right) 
	= -4  \left( \log^{2} \left( 1+\sqrt{2} +\sqrt{4 +2\sqrt{2}} \right) + \log^{2} \left( 1-\sqrt{2} +\sqrt{4 -2\sqrt{2}} \right) \right).  
\end{align*}
\end{example}

\subsection{Tables}\label{tables}
Some more examples illustrating Theorem \ref{J r chi FE} are given here in the form of a table. To facilitate the calculations involved, we first simplify the integrand in \eqref{J r chi def}.

Let $\chi_2$ be any Dirichlet character modulo $r$. 	Observe that $t$ divides $ \sum_{\ell=1}^{r} G(\ell,\chi_2) t^{\ell}$. Moreover, for a proper divisor $d$ of $r$, the $d^{\textup{th}}$cyclotomic polynomial, which we denote by $\varphi_d(t)$, divides $\sum_{\ell=1}^{r} G(\ell,\chi_2) t^{\ell}$. To prove this, it suffices to show that for $\zeta_d=e^{2\pi i/d}$, where $d<r$, $(t-\zeta_d)$ is a factor of $\sum_{\ell=1}^{r} G(\ell,\chi_2) t^{\ell}$. Indeed, if $n=r/d>1$, 
\begin{align*}
	\sum_{\ell=1}^{r} G(\ell,\chi_2) \zeta_d^{\ell}=\sum_{m=1}^{r-1}\chi(m)\sum_{\ell=1}^{r}e^{\frac{2\pi i(m+n)\ell}{r}}=0,
\end{align*}
since the inner sum vanishes unless $m+n=r$. However, if $m+n=r$, $\chi(m)=0$ since $\textup{gcd}(m, r)>1$.

Therefore, one can rewrite the integral $J_{r,\chi_1, \chi_2}(z,x) $ in \eqref{J r chi def} as
\begin{align}\label{rewritten J}
	J_{r,\chi_1, \chi_2}(z,x)=\frac{1}{\Gamma(z)} \int_{0}^{1} \left(\log \left( \frac{1}{t} \right) \right)^{z-1}\frac{ A_{\chi_2}(t)}{\varphi_r(t)}B_{\chi_1}(t) \, dt,
\end{align}
where
$B_{\chi_1}(t):=\sum_{j=1}^{r} \chi_1(j) \log \left( 1-\zeta_r^j t^{x} \right) $ and  $A_{\chi_2}(t):=\displaystyle\frac{-\sum_{\ell=1}^{r} G(\ell,\chi_2)t^{\ell}}{t\prod\limits_{d|r\atop{d<r}}\varphi_d(t)}$ is a polynomial of degree $\phi(r)-1$ if $\chi_2$ is a principal character, of degree $\phi(r)-2$ if $\chi_2$ is primitive, and of degree $\leq\phi(r)-2$ if $\chi_2$ is non-principal and imprimitive. Here, $\phi$ is the Euler totient function.

In the examples discussed in Table 2, we always consider $\chi_1=\chi$ and $\chi_2=\overline{\chi}$, where the character $\chi$ is defined in Table 1. Each entry in the first column of Table 1 is of the form $[r, j]$, which denotes the character $\chi$ with modulus $r$ and index $j$. The \emph{Mathematica 12} command for evaluating the Dirichlet character of modulus $r$ and index $j$ at the number $n$ is \url{DirichletCharacter[r, j, n]}.  

 With $J_{r, \chi_1, \chi_2}(z, x)$  defined in  \eqref{rewritten J}, the functional equation depicted in Table 2 is of the form 
\begin{align*}
	J_{r,\chi, \bar\chi}(1,x) + J_{r,\bar\chi, \chi}\left(1,\frac{1}{x}\right) = C_\chi,
\end{align*}
where $C_\chi$ is the right-hand side of \eqref{fe general extended} when $z=1$ (or, equivalently, that of \eqref{J r chi FE z=1}) with $\chi_1=\chi=\bar{\chi_2}$.

\renewcommand{\arraystretch}{1.1} 
\begin{table}[h!]
	\centering
	\begin{tabular}{||c||c|c|c| c| c| c| c| c| c| c| c| c|| c|| c|} 
		\hline
			\backslashbox{[r,j]}{n}  & 1 & 2 & 3 & 4 & 5 & 6 & 7 & 8 & 9 &10 &11 &12& $\chi(-1)$ & \text{Conductor}\\
		\hline 
		[2,1]&  1 & 0 & - & - & - & - & - &- & - &- & - &-& 1 & 1\\ 
		\hline 
		[3,2]& 1 &-1 & 0 & - & - & - & - &- & - &- & - &-& -1 & 3\\
		\hline
		[4,2]& 1 & 0 & -1& 0& - & - & - &- & - &- & - &-& -1 & 4\\ 
		\hline
		[5,3]&1 &-1 &-1 &1 &0 & - & - &- & - &- & - &-&1 & 5\\ 
		\hline
		[5,4]&1 & -$i$& $i$ &-1& 0 & - & - &- & - &- & - &-& -1 &5\\ 
		\hline
		[8,1]& 1 & 0 & 1 & 0& 1 &0& 1 &0& - &- & - &-& 1 & 1\\ 
		\hline 
		[8,4]& 1 & 0 & 1 & 0& -1 &0& -1 &0& - &- & - &-&  -1 & 8\\ 
		\hline
		[10,3]& 1 & 0 & -1 &0 & 0 & 0  &-1 & 0 &1 & 0 &-& - & 1 &5\\ 
		\hline
		[12,2]& 1 & 0 & 0 &0 & -1 & 0  &1 & 0 &0 & 0 &-1& 0 & -1 & 3\\ 
		\hline 
		[12,3]& 1 & 0 & 0 &0 & 1 & 0  &-1 & 0 &0 & 0 &-1& 0 &-1 &4\\
		\hline
	\end{tabular}
	\caption{List of Dirichlet Characters.}
	\label{table even char}
\end{table}

\renewcommand{\arraystretch}{1.7} 
\small\begin{table}[h!]
	\centering
	\begin{tabular}{||c|c|c|c||} 
		\hline 
		$[r,j]$& $A_{\bar{\chi}}(t)$& $B_{\chi}(t)$& $C_{\chi}$\\
		\hline 
		$	[2,1]$ & $-1$&$\log(1+t^x)$ & $-\log^2(2)$\\
		\hline
		[3,2]& $i \sqrt{3}$&$-2i \tan^{-1}\left(\frac{\sqrt{3}t^x}{2+t^x}\right)$& $\frac{1}{9}\pi^2$\\
		\hline 
		
		[4,2]& $2 i$&$- 2 i \tan^{-1}(t^x)$& $\frac{1}{4}\pi^2$ \\
		\hline
		
		[5,3] & $-\sqrt{5}(t^2-1)$ &$\log\left(2+2t^{2x} - (\sqrt{5}-1)t^x\right)-\log\left(2+2t^{2x} + (\sqrt{5}+1)t^x \right)$ &$-\log^2 \left(\frac{2 }{3-\sqrt{5}}\right)$  \\
		\hline
		
		[5,4]& $i (-15 + 20 i)^\frac{1}{4} (t^2+ (1 -i) t+ 1))$&$- 2 i\left( \tan^{-1}\left(\frac{\sqrt{10+2\sqrt{5}}t^x}{4+(1-\sqrt{5}) t^x}\right)-i\tan^{-1}\left(\frac{\sqrt{10-2\sqrt{5}}t^x}{4+(1+\sqrt{5}) t^x}\right)\right)$& $\frac{2}{5}\pi^2$\\
		\hline
		[8,1]& $-4t^3$&$\log\left(1-\sqrt{2} t^x+t^{2x}\right)+\log\left(1+\sqrt{2}t^x+t^{2x}\right)$&$-\log^2(2)$ \\
		\hline
		[8,4]& $2\sqrt{2}i (t^2+1)$&$- 2 i\left( \tan^{-1}\left(\frac{\sqrt{2}t^x}{2-\sqrt{2} t^x}\right)+\tan^{-1}\left(\frac{\sqrt{2}t^x}{2+\sqrt{2} t^x}\right)\right)$& $\pi^2$ \\
		\hline
		[10,3]& $-\sqrt{5} (t^2-1)$ &$\log\left(1-\frac{1+\sqrt{5}}{2}t^x+t^{2x}\right)-\log\left(1-\frac{1-\sqrt{5}}{2}t^x+t^{2x}\right)$&$-\log^2\left(\frac{7+3\sqrt{5}}{2}\right)$ \\
		\hline
		
		[12,2]& $2\sqrt{3}i t$ &$- 2 i\left( \tan^{-1}\left(\frac{t^x}{2-\sqrt{3} t^x}\right)-\tan^{-1}\left(\frac{t^x}{2+\sqrt{3} t^x}\right)\right)$& $\frac{4}{9}\pi^2$ \\
		\hline
		[12,3]& $2i (t^2+1)$ &$- 2 i\left( \tan^{-1}\left(\frac{t^x}{2-\sqrt{3} t^x}\right)+\tan^{-1}\left(\frac{t^x}{2+\sqrt{3} t^x}\right)\right)$& $\pi^2$ \\
		\hline
	\end{tabular}
	\caption{Expressions for $A_{\bar{\chi}}(t), B_{\chi}(t)$ and $C_{\chi}$ occurring in \eqref{rewritten J}.}
	\label{table odd char}
\end{table}

\normalsize

\begin{remark}
\textup{(i)} Note that $C_{\chi}$ is a multiple of $\pi^2$ if $\chi$ is odd, and is a multiple of the square of log if $\chi$ is even. Compare, for example, $[8, 4]$ with $[8, 1]$.

\noindent
\textup{(ii)} The  degrees of $A_{\bar\chi}(t)$ in $[12,2]$ and $[12,3]$ are different even though both the characters are odd and imprimitive.
\end{remark}
\subsection{Extension to periodic functions}

The earlier results in this section which concern Dirichlet characters can be extended to accommodate periodic functions $f$ of period $r$ satisfying $f(r)=0$. We give below an example to illustrate this.

Let $f(n)$ be a periodic function defined on $\mathbb{N}$ with a period $4$ as defined in the table below,
\begin{center}
	\begin{tabular}{|c|c c c c|} 
		$j$  & 1 & 2 & 3 & 4   \\
		\hline 
		$f(j)$  & -1 & $\pi$ & $e$ & 0
	\end{tabular}
\end{center}
Define
\begin{align}
	J_{4,f,f}(z,x):= \sum_{j=1}^{3} \sum_{k=1}^{3} f(j) f(k) \mathscr{F}_z(x ,i^{j}, i^{k}). \label{Jzx start}
\end{align}
%
From \eqref{script F def z}, 
\begin{align}
	J_{4,f,f}(z,x)
	&= \frac{1}{\Gamma(z)} \int_{0}^{1} \left( \log(\tfrac{1}{t})\right)^{z-1} \sum_{j=1}^{3} \left( f(j)  \log(1-i^{j}t^x) \right) \sum_{k=1}^{3} \left( \frac{f(k) }{i^{4-k}-t} \right) \, dt.  \label{Jzx mid}
\end{align}
The sum over $k$ can be simplified as
\begin{align}
	\sum_{k=1}^{3} \frac{f(k)}{i^{4-k}-t}  = \sum_{k=1}^{4}  \frac{f(5-k) }{i^{k-1}-t} =-\frac{i+ie+\pi +(i+e+ie-1)t +(e+\pi-1)t^2}{t^3+t^2+t+1}. \label{Jzx second}
\end{align}
Also,
\begin{align} 
	\sum_{j=1}^{3} f(j) \log(1-i^{j}t^x) 
	&= \frac{1}{2}(e-1) \log(1+t^{2x}) + \pi \log(1+t^x) +i(e+1) \tan^{-1}(t^x). \label{Jzx third}
\end{align}
Substitute \eqref{Jzx second} and \eqref{Jzx third} in \eqref{Jzx mid} to get
\begin{align*}
	J_{4,f,f}(z,x)= -\frac{1}{\Gamma(z)}  \int_{0}^{1}& \left( \log(\tfrac{1}{t})\right)^{z-1} \left( \frac{i+ie+\pi +(i+e+ie-1)t +(e+\pi-1)t^2}{t^3+t^2+t+1} \right) \notag \\&\times \left( \frac{1}{2}(e-1) \log(1+t^{2x}) + \pi \log(1+t^x) +i(e+1) \tan^{-1}(t^x) \right) \, dt. 
\end{align*}
From \eqref{small theorem} and \eqref{Jzx start}
\begin{align}
	J_{4,f,f}(z,x) +x^{1-z}J_{4,f,f}\left(z,\frac{1}{x}\right) 
	& =  - \sum_{n=1}^{\infty} \sum_{m=1}^{\infty} \frac{\left(\sum_{j=1}^{4}f(j)i^{mj}\right)\left( \sum_{k=1}^{4}f(k)i^{nk}\right)}{n m (n+mx)^{z-1}} \notag \\
	&= - \sum_{n=1}^{\infty} \sum_{m=1}^{\infty} \frac{g(m) g(n)}{n m (n+mx)^{z-1}}, \label{Jzx alternative}
\end{align}
where, $g(m)$ is the Gauss-type sum associated with the periodic function $f$. We then have
\begin{align*}
	g(m)=\sum_{j=1}^{4}f(j)i^{mj}= \begin{cases} 
		-1+e+\pi, & \text{if }m \equiv 0 \ (\text{mod } 4),\\
		-i-ie-\pi, & \text{if } m \equiv 1 \ (\text{mod } 4), \\
		1-e+\pi, & \text{if }m \equiv 2 \ (\text{mod } 4), \\
		i+ie-\pi, & \text{if }m \equiv 3 \ (\text{mod } 4),
	\end{cases}
\end{align*}
which is a periodic function of a period 4. 
Using Proposition \ref{crandall}, one can let $z\to 1$ in \eqref{Jzx alternative} so that 
\begin{align*}
	J_{4,f,f}(1,x) + J_{4,f,f}\left(1,\frac{1}{x}\right) = - \left( \sum_{n=1}^{\infty} \frac{g(n)}{n} \right)^2 
	&= -\frac{1}{16}\left(\pi i(e+1) +2(e+2\pi-1)\log(2)\right)^2. 
\end{align*}
Hence, $J_{4,f,f}(1,1) = -\frac{1}{32}\left(\pi i(e+1) +2(e+2\pi-1)\log(2)\right)^2.$

\begin{remark}
A famous problem of Chowla  asks whether there exists a non-zero integer-valued function $g$ with prime period $r$ such that $g(r) = 0$, $\sum_{m=1}^{r}g(m)=0$ and $\sum_{n=1}^{\infty}g(n)/n=0$. This problem was settled in the negative by Baker, Birch and Wirsing \cite{bbw}.

Let $r$ be a prime and suppose $g$ is an $r$-periodic integer-valued  function satisfying the first two properties. Then using \cite[p.~160, Theorem 8.4]{Apostol}, one can get an $f$ which is also $r$-periodic. With these $f$ and $g$, one can get an identity of the form
$J_{r,f,f}(1,x) + J_{r,f,f}\left(1,\frac{1}{x}\right) = - \left( \sum_{n=1}^{\infty} g(n)/n \right)^2$.  Can this identity be used to give another proof of the result of Baker, Birch and Wirsing?
\end{remark}

\begin{remark}
The result in \eqref{Jzx alternative} can be extended starting with two different $r$-periodic functions $f_1$ and $f_2$ such that $f_1(r)=0=f_2(r)$. The result will then be of the form
\begin{align*}
	J_{r,f_1,f_2}(z,x) +x^{1-z}J_{r,f_2,f_1}\left(z,\frac{1}{x}\right) 
= - \sum_{n=1}^{\infty} \sum_{m=1}^{\infty} \frac{g_1(m) g_2(n)}{n m (n+mx)^{z-1}},
\end{align*}
where $g_k(m):=\sum_{j=1}^{r}f_k(j)e^{2\pi im j/r}$ for $k=1, 2$.
\end{remark}
\section{Concluding remarks}	

As shown in this paper, the insertion of the parameter $x$ in $\zeta_{\textup{MT}}(1, 1, z-1)$ defined in \eqref{MT} so as to obtain $\Theta(1, 1, z-1, x)$, that is,  $\sum\limits_{n=1}^{\infty}\sum
\limits_{m=1}^{\infty}\displaystyle\frac{1}{nm(n+mx)^{z-1}}$ has many ramifications: it is connected with a generalization of the Herglotz-Zagier function $F(x)$, namely, $\Phi(z, x)$, it admits a three-term functional equation. Moreover, its character analogue $\sum\limits_{n=1}^{\infty} \sum\limits_{m=1}^{\infty} \displaystyle\frac{\chi_1(m)\chi_2(n)}{ n m(n+mx)^{z-1}}$, or more generally, the right-hand side of \eqref{fe-general}, induces an interesting functional equation for an integral $J_{r, \chi_1, \chi_2}(z, x)$ whose two special cases have occurred in the literature. Another functional equation derived in this paper is the two-term functional equation for the Ishibashi functions $\Phi_k(x)$, which was missing in the literature and which, for $k=1$, gives the one for the Herglotz-Zagier function $F(x)$ as a special case. The case $\Theta(r, s, z-1, x)$ for higher $r$ and $ s$ deserves separate study and will be undertaken in a future work.

The current study generates further interesting questions worth investigating which are listed below.\\

 \noindent		
(1) The special case $J(x)$ of $J_{r, \chi_1, \chi_2}(z, x)$ has been previously evaluated by Herglotz \cite{herglotz} and by Muzzaffar and Williams \cite{muzwil} for values of $x$ of the form $n+\sqrt{n^2-1}$. The techniques for their evaluations stem from algebraic number theory. Radchenko and Zagier \cite{raza} not only evaluated $J(x)$ for $x$ of the form $n+\sqrt{n^2+1}$ but also when it takes rational values. Muzzaffar and Williams also gave explicit evaluations for $T(n+\sqrt{n^2-1})$ defined in \eqref{JT-def}.

Since various choices can be made for $r$, $z$, $x$, as well as for the Dirichlet characters $\chi_1$, $\chi_2$ in the definition of $J_{r, \chi_1, \chi_2}(z, x)$, questions similar to above can be asked more generally, that is, can we explicitly evaluate this integral at these special choices of the parameters? What techniques would be required to obtain such explicit evaluations? A few such integrals have been given in Table 2. 

\vspace{0.2cm}
 \noindent
(2) The two-term functional equation for the Ishibashi function $\Phi_k(x)$ is quite involved and could still possibly be made more explicit. The limit $L_k^{*}(x)$ in \eqref{lk*} remains to be evaluated. It is clear from its definition that for its evaluation we need the full power series expansion of $I(z, x)$ around $z=1$. In view of \eqref{i=h1-h2}, this, in turn, requires us to obtain the Laurent series expansions of $H_1(z, x)$ and $H_2(z, x)$. The difficulty is in getting the Laurent series expansion of the integral in \eqref{h2eval}. Now \eqref{h1eval} implies that as $z\to1$,
$H_1(z, x)=\frac{\frac{1}{2}\log^{2}(x)}{z-1}+O(1)$. Moreover, $\frac{1}{z-1}((x+1)^{1-z}-x^{1-z}-1)\zeta(z+1)=\frac{-\pi^2/6}{z-1}+O(1)$, and hence \eqref{h2eval} and the analyticity of $I(z, x)$ at $z=1$ implies that as $z\to 1$,
\begin{equation*}
\int_{0}^{\infty}\frac{t^{z-2}}{\Gamma(z)}\left\{\textup{Li}_2\left( \frac{e^{-(x+1)t}-e^{-xt}}{1-e^{-xt}} \right)+\textup{Li}_2\left( \frac{e^{-(x+1)t}-e^{-t}}{1-e^{-t}} \right)\right\}dt=\frac{-\frac{1}{2}\log^{2}(x)-\frac{\pi^2}{6}}{z-1}+O(1).
\end{equation*}
However, we need the Laurent series coefficients for evaluating $L_k^{*}(x)$, not just the principal part. Another way to look at $L_{k}^{*}(x)$ is through Proposition \ref{another lk*}, which requires evaluating the integrals occurring in it in closed-form, which looks difficult. 

\vspace{0.2cm}
\noindent
(3) It would be worthwhile to see if the constants $\Phi_k(1)$ occurring in \eqref{two-term eqn for Ishibashi} could be explicitly evaluated. Zagier \cite[p.~173]{zagier} evaluated $\Phi_1(1)$ explicitly by exploiting the fact that $\sum_{1\leq n<m\leq N}\frac{1}{mn}$ and $\sum_{1\leq m<n\leq N}\frac{1}{mn}$ are equal. This evaluation is given in \eqref{ef1}. The symmetry, however, breaks down for higher $k$ since one has to deal with $\sum_{1\leq n<m\leq N}\frac{\log(m)}{mn}$ and $\sum_{1\leq m<n\leq N}\frac{\log(m)}{mn}$, which are \emph{not} equal.

\vspace{0.2cm}
\noindent
(4) The three-term functional equation for $\Phi(z, x)$ still remains to be found whereas the one for $\Theta(z, x)$ was derived in this paper. In view of \eqref{decomposition}, it seems that $\Phi(z, x)$ is more fundamental than $\Theta(z, x)$. 

Once the three-term functional equation for $\Phi(z, x)$ is obtained, it will give, via differentiation with respect to $z$, the three-term functional equation for $\Phi_k(x)$. The one for $\Phi'_k(x)$ is obtained in Theorem \ref{Thm 3term-curious}, whose special case $k=0$ is due to Zagier; see the equation before (7.8) in \cite{zagier}. 

\vspace{0.2cm}
\noindent
(5) For simplicity, we have considered $x>0$ while deriving our functional equations. However, they seem to hold for Re$(x)>0$ as well. 

\vspace{0.2cm}
\noindent
(6) The integral in \eqref{J(z,x) def} as well as the ones of similar type are used in evaluation of Feynman and relativistic phase space integrals \cite[p.~1233]{kolbig},  \cite{Gastmans}. It would then be interesting to see if our more general integral in \eqref{J r chi def} could also be used for this purpose.
	
\vspace{0.2cm}
\noindent
(7) Note that the identity we have obtained in \eqref{1=z} is valid in $|u|<1$ and $|v|<1$ whereas the same identity is obtained by Kumar and Choie \cite[Equation (2.6)]{kumar-choie} for $u, v\in\{w: |w|\leq1, w\neq1\}$. In order to obtain this result for $u, v$ on the unit circle such that $u\neq1, v\neq1$, one may let $z\to 1$ in Theorem \ref{small theorem} but then the interchange of limit and summation poses a problem.  The right approach is to analytically continue the right-hand side of Theorem \ref{small theorem} and then let $z\to 1$ in the analytically extended version of the theorem (which exists for Re$(z)>0$). 

\vspace{0.2cm}
\noindent
(8) Does there exist a three-term functional equation of the type in Theorem \ref{theta-3t} for the right-hand side of \eqref{fe-general}, or, for that matter, for the right-hand side of Theorem \ref{small theorem}?

\vspace{0.2cm}
\noindent
(9) In Theorem \ref{J(z,x) F_z(x) relation}, we gave a relation between $\Phi(z, x)$ and $J(z, x)$. Since the latter is a special case of $J_{r, \chi_1, \chi_2}(z, x)$, it may be of interest to find a relation between $\Phi(z, x)$ and $J_{r, \chi_1, \chi_2}(z, x)$.

\vspace{0.2cm}
\noindent
(10) From \eqref{passing the limit}, we have $	\lim_{z \to 1} \sum_{n=1}^{\infty} \sum_{m=1}^{\infty} \frac{G(m,\chi_1)G(n,\chi_2)}{m n (n+mx)^{z-1}} =	 \sum_{n=1}^{\infty} \sum_{m=1}^{\infty}\lim_{z \to 1} \frac{G(m,\chi_1)G(n,\chi_2)}{m n (n+mx)^{z-1}} $. If we first differentiate both sides of \eqref{fe-general} with respect to $z$, let $z\to1$, $x=1$, and \emph{formally} interchange the order of $\lim_{z\to1}\frac{d}{dz}$ and the double sum, then we obtain identities of the kind
\vspace{-1.5mm}
\begin{align*}
	\int_{0}^{1} \frac{\log(\log(\tfrac{1}{y}))\log(1+y)}{1+y} \, dy &= - \frac{1}{2} \gamma \log^2(2) 
	- \frac{1}{2} \sum_{n=1}^{\infty} \sum_{m=1}^{\infty}  \frac{(-1)^{m+n}\log(m+n)}{nm},
\end{align*}
\vspace{-1.5mm}
which seem to hold true, at least numerically. However, we are unable to justify the interchange.  

\vspace{0.2cm}
\noindent
(11) The identity in \eqref{fe-general} seems to hold for Re$(z)>0$ numerically.  While the left-hand side is well-defined and analytic in this region, the convergence/divergence of the right-hand side in the strip 
 $0<\textup{Re}(z)\leq1$, and for $z\neq1$, seems to be delicate and is yet to be understood.

\vspace{0.2cm}
\noindent
(12) A transformation for the double cotangent function given by Tanaka \cite[Equation (1.8)]{tanaka} is equivalent to Ramanujan's formula \eqref{w1.26}. The latter is a two-term functional equation for a Herglotz-Zagier type function. Thus it would be of merit to see the possible  connections of multiple cotangent functions of Tanaka with Herglotz/ Ishibashi-Herglotz functions. 

\vspace{0.2cm}
\noindent
(13) From the discussion following the statement of Theorem \ref{principal part}, it is clear that the behavior of $\Theta(1, 1, z-1, 1)$ as $z\to1$ is quite different from that of $\Theta(z, z, 0,1)(=\zeta^{2}(z))$. To add to the mystery, the behavior of $\Theta(z, z, z-1, 1)$ as $z\to1$ seems to be very different from the above two cases which certainly merits its independent investigation! 

	\begin{center}
		\textbf{Acknowledgements}
	\end{center}
The authors sincerely thank Rahul Kumar for sending us copies of \cite{Gastmans} and \cite{ogieveckii}. They also thank
Wadim Zudilin, Kohji Matsumoto, Takashi Nakamura, M.~Lawrence Glasser, Don Zagier, Jan Min\'{a}\v{c} and Tung T. Nguyen for helpful discussions. The first author is supported by the Swarnajayanti Fellowship grant SB/SJF/2021-22/08 of SERB (Govt. of India). The second author is an Institute Postdoc supported by the grant MIS/IITGN/R\&D/MATH/AD/2324/058 of IIT Gandhinagar. Both the authors sincerely thank the respective agencies for their generous support.

	\end{document}